\PassOptionsToPackage{dvipsnames}{xcolor}
\documentclass[reqno]{amsart}

\usepackage{amsmath,amssymb,amsthm,verbatim, csquotes}
\usepackage{hyperref}
\usepackage{xcolor}
\usepackage{esint}
\usepackage{mathtools}
\mathtoolsset{showonlyrefs}
\usepackage{geometry}
\numberwithin{equation}{section}

\newtheorem{theorem}{Theorem}[section]
\newtheorem{lemma}[theorem]{Lemma}
\newtheorem{cor}[theorem]{Corollary}
\newtheorem{prop}[theorem]{Proposition}

\theoremstyle{definition}
\newtheorem{definition}[theorem]{Definition}
\newtheorem{remark}[theorem]{Remark}
\newtheorem{example}[theorem]{Example}
\newtheorem{ass}[theorem]{Assumption}

\newcommand{\graph}{\mathrm{graph}}

\newcommand{\sing}{\mathrm{sing}}

\newcommand{\del}{\partial}

\newcommand{\mdiv}{\mathrm{div}}

\newcommand{\cL}{\mathcal{L}}

\newcommand{\bY}{\mathbf{Y}}

\newcommand{\cH}{\mathcal{H}}

\newcommand{\cA}{\mathcal{A}}

\newcommand{\osc}{\mathrm{osc}}
\newcommand{\grad}{\mathrm{grad}}

\newcommand{\D}{\mathcal{D}}

\newcommand{\R}{\mathbb{R}}
\newcommand{\N}{\mathbb{N}}

\newcommand{\haus}{\mathcal{H}}

\newcommand{\eps}{\varepsilon}

\newcommand{\ver}{\R^n}
\newcommand{\hor}{{P_{\pi/2}}}

\newcommand{\spt}{\mathrm{spt}}
\newcommand{\Gr}{\mathrm{Gr}}       

\newcommand{\interior}[1]{ {\kern0pt#1}^{\mathrm{o}} }

\newcommand*\dif{d}
\DeclareMathOperator{\dive}{div}

\makeatletter
\def\subsection{\@startsection{subsection}{2}%
  \z@{.5\linespacing\@plus.7\linespacing}{.3\linespacing}%
  {\normalfont\bfseries}}
\makeatother

\begin{document}

\title{Regularity of minimal surfaces with capillary boundary conditions}

\author{Luigi De Masi}
\address{\textit{Luigi De Masi: } Dipartimento di Matematica, Università di Trento, Via Sommarive 14, I-38123 Povo (Trento), Italy}
\email{luigi.demasi@unitn.it}

\author{Nick Edelen}
\address{\textit{Nick Edelen: } Department of Mathematics, University of Notre Dame, Notre Dame, IN, 46556, USA}
\email{nedelen@nd.edu}

\author{Carlo Gasparetto}
\address{\textit{Carlo Gasparetto: } Dipartimento di Matematica, Università di Pisa, Largo Bruno Pontecorvo 5, 56127 Pisa, Italy}
\email{carlo.gasparetto@dm.unipi.it}

\author{Chao Li}
\address{\textit{Chao Li: } Courant Institute, New York University, 251 Mercer St, New York, NY 10012  }
\email{chaoli@nyu.edu}

\begin{abstract}
	We prove $\varepsilon$-regularity theorems for varifolds with capillary boundary condition in a Riemannian manifold. These varifolds were first introduced by Kagaya-Tonegawa \cite{KaTo}. We establish a uniform first variation control for all such varifolds (and free-boundary varifolds generally) satisfying a sharp density bound and prove that if a capillary varifold has bounded mean curvature and is close to a capillary half-plane with angle not equal to $\tfrac{\pi}{2}$, then it coincides with a $C^{1,\alpha}$ properly embedded hypersurface.  We apply our theorem to deduce regularity at a generic point along the boundary in the region where the density is strictly less than $1$.
\end{abstract}
\maketitle

\tableofcontents

\section{Introduction}

In this paper, we are interested in the boundary regularity of capillary hypersurfaces, which are surfaces with prescribed mean curvature that meet the container boundary at a prescribed angle.  Such surfaces arise as critical points of the Gauss' free energy: given a Riemannian manifold with boundary $(X^{n+1},g)$, functions $h, \beta \in C^1(X)$, the Gauss' free energy for an open subset $E\subset X$ is defined as
	\begin{equation}\label{eqn:gauss.free.energy}
		\cA(E) = \cH^n (\partial E\cap \mathring X) +\int_{\partial E\cap \partial X} \cos\beta(x) d\cH^n  + \int_E h(x)d\cH^{n+1}.
	\end{equation}
	When $E\subset X$ is stationary for $\cA$, or stationary among domains with a fixed volume, the interface $M = \partial E\cap \mathring X$ is called a capillary (hyper)surface. Such surfaces are the mathematical model for interfaces of incompressible fluids in an equilibrium state. Formally (or when $M$ is sufficiently regular), the Euler-Lagrange equation for $M$ - also known as \textit{Young's law } \cite{Young}-  asserts that the mean curvature of $M$ equals $h$ (plus a constant, if volume-constrained), and that $M$ meets $\partial X$ at an angle $\beta$. The existence, regularity, and geometric properties of capillary surfaces have attracted active investigations for the past two centuries, see, e.g. \cite{Finn} for a beautiful historical survey on the subject. There have also been interesting developments and geometric applications in recent years, see, e.g. \cite{King-Maggi-Stuvard-2022, Jia-Zhang-2024, Wang-Xia-2019, DePhil-Fusco-Morini,Li-Polyhedron,ChaiWang-dihedral}.
    A non-local version of Gauss' energy was introduced in \cite{Maggi-Valdinoci_2017}, see also the survey \cite{Valdinoci-survey} and the references therein.
    For a capillary problem involving the one-phase Bernoulli free boundary problem, see \cite{Ferreri-Tortone-Velichkov-2023}.

 	The primary scope of this paper is to establish Allard type \cite{allard1972}, $\varepsilon$-regularity results for capillary hypersurfaces near their boundary. Such results are fundamental in understanding the regularity of stationary capillary hypersurfaces, which can arise not only in minimization problems but also in min-max constructions.  Our weak notion of a hypersurface is a stationary varifold,  and since our regularity results are local, we model a neighborhood of $p\in \partial X$ with the half-ball $B_1\cap \{x_1\le 0\}$ in $\R^{n+1}$ endowed with a Riemannian metric $g$.
	
	We remark here that an $\varepsilon$-regularity theorem for \emph{minimizing} capillary hypersurfaces was obtained by Taylor in \cite{Taylor} (when $n = 2$), and by De Philippis-Maggi in \cite{DePhilippisMaggi} (in general dimensions).  Let us also remark that in the special case when $\beta \equiv \pi/2$, capillary surfaces are referred to as free-boundary surfaces, and the regularity of these has been extensively studied by \cite{GrJo}, \cite{Gruter}, and others.
 
	Varifolds with prescribed contact angle arise naturally as critical points of the Gauss' free energy \eqref{eqn:gauss.free.energy}.
    Indeed, assume $E\subset \{x_1\le 0\}$ is a smooth, relatively open set such that the first variation of $\cA$ vanishes at $E$ for every smooth compactly supported vector field $X$ that is tangential along $\R^n=\{x_1=0\}$, that is
    \begin{equation}
        \frac{d}{dt}\Big|_{t=0}\cA(\Phi_t(E))=0,
    \end{equation}
    where $\Phi_t(x)=x+tX(x)$.
    Then a simple computation (see section \ref{subsection.capillary.varifolds} for details) shows that the varifold
	\begin{equation}\label{eqn:intro.capillary.varifolds.induced.open.sets}
		V=[\partial E\cap \{x_1<0\}] + \cos \beta [\partial E\cap (\R^n=\{x_1=0\})]
	\end{equation}
	behaves like a `free-boundary' varifold with bounded first variation, see \eqref{eqn:capillary.first.variation.intro} below. Motivated by this, we define a capillary varifold $V$ in the form of $V = V_I + V_B$, corresponding to the interior part (e.g. $V_I = [\partial E\cap \{x_1<0\}]$) and the boundary part (e.g. $V_B =\cos\beta[\partial E\cap \R^n]$), and ask specifically that $V_B=\cos \beta [S]$ for some set $S\in \R^n$. In this weak formulation, the (distributional) condition that $V_I$ meets $\R^n$ with angle $\beta$ and has integrable mean curvature (i.e. the Euler-Lagrange equation for \eqref{eqn:gauss.free.energy}) becomes:
	\begin{align}\label{eqn:capillary.first.variation.intro}
	\delta_{\beta, g} V(X) &:= \int \mdiv_{V, g}(X) d\mu_{V_I} + \int_S \mdiv_{\R^n, g}(X \cos\beta) d\haus^n_g \\
	&= - \int g(H^{tan}, X) d\mu_{V_I}
	\end{align}
    for all $X\in C^1_c(B_1, \R^{n+1})$ tangential along $\{x_1=0\}$,
	for some mean curvature vector field $H^{tan}\in L_{loc}^1(\mu_{V_I})$. Here $\mdiv_{V, g}$ is the $g$-divergence in the tangent space of $V$, and $\haus^n_g$ is the $n$-dimensional Hausdorff measure with respect to $g$.
	
	\begin{definition}\label{def.capillary.varifolds}
		Given a Riemannian metric $g$ in $B_1\subset \R^{n+1}$ and $\beta\in C^1(B_1)$, we say a signed $n$-rectifiable varifold $V$ in $B_1\cap \{x_1\le 0\}$ is a $(\beta,S)$-capillary for some set $S\subset \{x_1=0\}$, if $V=V_I+V_B:= V\llcorner \{x_1<0\} + \cos \beta \cH^n_g \llcorner S$ satisfies \eqref{eqn:capillary.first.variation.intro}, where $V_I$ is integral and $V_B$ has density equal to $\cos \beta$.
	\end{definition}
	
	\begin{example}\label{example.model.capillary.half.plane}
		An archetype example, which serves the model for our regularity theorem, is the half-plane capillary varifold
		\[V^{(\theta)} = [P_\theta^-]+\cos\theta [\{x_1=0,x_{n+1}\le 0\}].\]
		Here $g$ is the Euclidean metric, $\beta=\theta\in (0,\pi)$ is constant, $H=0$, and $P_\theta^{-}$ is the half plane
		\[P_\theta^- = \{x\in \R^{n+1}: x_1\cos \theta + x_{n+1}\sin\theta =0, x_1\le 0\}.\]
	\end{example}
	
	\begin{example}[Swapping the wet region]\label{example.swap.region}
		One way of constructing new capillary varifolds from existing ones is by `swapping the wet region'. Suppose $V = V_I +\cos\beta \cH^n_g\llcorner S$ is $(\beta,S)$-capillary. We define $\tilde V$ by setting
		\[\tilde V = V_I  + \cos (\pi - \beta) \cH_g^n\llcorner (\R^n\setminus S).\]
		It is straightforward to check (using the divergence theorem) that $\tilde V$ is $(\pi -\beta, \R^n\setminus S)$-capillary. Particularly, if $V$ is induced by $E$ as in \eqref{eqn:intro.capillary.varifolds.induced.open.sets}, then $\tilde V$ is induced by $B_1\setminus E$.
	\end{example}
	
	\begin{remark}
		When $\beta$ is a constant, capillary varifolds were first defined and investigated by Kagaya-Tonegawa \cite{KaTo}, where a monotonicity formula for $V$ was established. Note that in this case $V$ is a free-boundary varifold with integrable mean curvature, in the sense that
            \begin{equation}\label{eqn:intro.free.boundary}
                \delta_g V(X):=\int \dive_{V,g}(X) d\mu_V=-\int g(H^{tan},X) d\mu_V
            \end{equation}
            for all $X\in C_c^1(B_1,\R^{n+1})$ tangential to $\R^n$.
	\end{remark}
	
	\begin{remark}
		Varifolds as in Definition \ref{def.capillary.varifolds} naturally arise in the min-max construction of capillary surfaces \cite{LiZhouZhu}, \cite{DeMasiDePhilippis}. Particularly, they appear as varifold limits of the suitably weighted boundary of a min-max sequence of  \eqref{eqn:gauss.free.energy}.
	\end{remark}

	Although well motivated, Definition \ref{def.capillary.varifolds} is very general and allows degenerate phenomena not necessarily seen in the formulation induced by a domain $E$.  For instance, when $V_I$ is the varifold associated to a sufficiently regular hypersurface $M$, \eqref{eqn:capillary.first.variation.intro} asserts that, at every boundary point $M$ either meets $\{ x_1 = 0 \}$ at angle $\beta$ (when the boundary of $S$ nearby coincides with the boundary of $\del M$), or instead $M$ could meet $\{ x_1 = 0 \}$ at angle $\pi/2$ (when $S$ has no boundary nearby).  Moreover, a sequence of $(\beta, S)$-capillary hypersurfaces $M_j$ all meeting $\{x_1=0\}$ at constant angle $\beta \ne \tfrac{\pi}{2}$ could convergence as varifolds to a $(\beta,S)$-capillary hypersurface meeting $\{x_1=0\}$ at an entirely different (constant) angle, see the examples in Remarks \ref{rmk:triple.junction} and \ref{rem:example-collision}.
	
	Another essential challenge with Definition \ref{def.capillary.varifolds} is that the condition that $V$ has bounded $g$-first variation (a consequence of an identity like \eqref{eqn:capillary.first.variation.intro}) does not necessarily guarantee that $V_I$ and $V_B$ each individually have bounded $g$-first variation. In other words, although the composite varifold $V$ will have a distributional notion of boundary and mean curvature, the interior $V_I$ and boundary $V_B$ components of $V$ need not themselves have distributional boundaries/mean curvatures.  Because of this, there is not a general compactness theory for capillary varifolds that respect the splitting of the interior and the boundary parts. See Remark \ref{rem:teaser1-sharp} for such an example.
	
	\subsection{First variation control}
	In light of these issues, we first establish a uniform first variation control for capillary varifolds, which asserts that under a necessary density bound, both $V_I$ and $V_B$ have uniform first variation bounds. Our result holds also for general free-boundary varifolds $V$ in the sense of \eqref{eqn:intro.free.boundary} without assuming that $V\llcorner \{x_1=0\} = 0$ - particularly, when $\beta$ is constant, a $(\beta,S)$-capillary varifold $V$ is free-boundary. As such, we anticipate that this should be widely applicable for the study of free-boundary minimal surfaces. See Theorems \ref{thm:cap-first-var} and \ref{thm:refined} for more detailed conclusions.
	
	\begin{theorem}[First variation control]\label{thm:teaser1}
		Let $g$ be a $C^1$ metric on $B_1 \subset \R^{n+1}$ (for $n \geq 1$), $\beta$ a $C^1$ function on $B_1$, $S \subset \R^n$, and $V$ a $(\beta, S)$-capillary varifold in $(B_1, g)$.  Given $\alpha > 0$, there is a constant $\delta(n, \alpha)>0$ so that if
		\begin{gather}
		\Theta_V(0, 1) + (\cos\beta(0))_- \leq 1-\alpha, \label{eqn:teaser1-hyp1} \\
		\max \{ ||H^{tan}_{V, g}||_{L^\infty(B_1)}, |g - g_{eucl}|_{C^1(B_1)}, |D\beta|_{C^0(B_1)}\} \leq \delta, \label{eqn:teaser1-hyp2}
		\end{gather}
	        (where $\Theta_V (0,1) = \omega_n^{-1} \mu_V(B_1)$ is the normalized mass, and $\lambda_- := - \min\{0, \lambda\}$ denotes the negative part of $\lambda$), 
		
		Then in $B_\delta$: each $V_I = V\llcorner \{x_1<0\}, V_B=V\llcorner\{x_1=0\}$ individually has finite $g$-first-variation; both the interior mass measure $\mu_{V_I}$ and the boundary measure $\sigma_{V, g}$ are Ahlfors-regular; the set $\spt \sigma_{V, g} \cap B_\delta$ is $(n-1)$-rectifiable with uniform $(n-1)$-packing estimates. 
	\end{theorem}
	
	\begin{remark}
		It is instructive to keep in mind that when $V = V^{(\theta)}$ is the model capillary half-plane varifold defined in Example \ref{example.model.capillary.half.plane}, then
		\[
		\Theta_{V^{(\theta)}}(0, 1) + (\cos\theta)_- = (1+|\cos\theta|)/2 < 1.
		\]
		So one may interpret condition \eqref{eqn:teaser1-hyp1} as asking that the total ``unsigned'' mass of $V$ is less than $1-\alpha$, although we also note that \eqref{eqn:teaser1-hyp1} is in fact a stronger requirement than a bound on the unsigned mass of $V$.
		
		We also remark that  \eqref{eqn:teaser1-hyp1} is compatible with swapping $S$ with $\R^n \setminus S$ and $\beta$ with $\pi - \beta$ as described in Example \ref{example.swap.region}.  For, if $\cos\beta(0) < 0$ and we let $\tilde V = V_I + \cos(\pi - \beta) [\R^n \setminus S]$, then $\cos(\pi - \beta(0)) \geq 0$ and $\tilde V$ satisfies
		\[
		\Theta_{\tilde V}(0, 1) \leq \Theta_V(0, 1) + (\cos\beta(0))_- + c(n)\delta \leq 1 - \alpha + c(n)\delta,
		\]
		which can be made less than (e.g.) $1-\alpha/2$ provided that $\delta(n, \alpha)$ is sufficiently small.
	\end{remark}

	With only notational modifications, the same statement also holds for general free-boundary varifolds.
        \begin{theorem}\label{thm:teaser1.free.bdry}
            Let $V$ be a free-boundary varifold in $(B_1,g)$ in the sense of \eqref{eqn:intro.free.boundary}, and assume that $V\llcorner \{x_1<0\}$ is integral, $V\llcorner\{x_1=0\} = \cos\beta \cH^n_g\llcorner S$ with $S \subset \R^n$, and \eqref{eqn:teaser1-hyp1}, \eqref{eqn:teaser1-hyp2} hold. Then the conclusions of Theorem \ref{thm:teaser1} hold.
        \end{theorem}

	\begin{remark}
		If $\beta\ne \tfrac{\pi}{2}$ in $B_\delta$, then Theorem \ref{thm:teaser1} implies that $S$ is a set of finite perimeter in $B_\delta \cap \{ x_1 = 0\}$.  See Remark \ref{rem:S-perimeter}.
	\end{remark}
	
	\begin{remark}\label{rem:teaser1-sharp}
		The density upper bound \eqref{eqn:teaser1-hyp1} in Theorem \ref{thm:teaser1} is sharp, as illustrated by the examples of \cite{GiLa}, \cite{stern-karpukhin}:  in \cite{GiLa}, Girouard-Lagac\'e constructed properly embedded free-boundary minimal surfaces in $B_1 \subset \R^3$, which were shown by Karpukhin-Stern \cite{stern-karpukhin} to converge as varifolds to $[S^2=\partial B_1]$, in the Hausdorff distance to $S^2$, and whose boundary measures converge as Radon measures to $2[S^2]$ (therefore $V_I$ along the sequence limits entirely to the boundary).  In particular the Ahlfors-regularity and packing estimates of the boundary measures $\sigma_{V, g}$ diverge as $\alpha \to 0$. 
	\end{remark}
	
	An important consequence of Theorem \ref{thm:teaser1} is a boundary maximum principle for capillary varifolds (Theorem \ref{thm:boundary-max}), which plays a key role in the proof of the main regularity result.
	
	\subsection{Allard type regularity theorems}

	Our second main theorem is a $C^{1,\alpha}$-regularity result for capillary varifolds near a capillary half-plane, analogous to Allard's regularity theorem.  See Theorem \ref{thm:main} for a more quantitative version.
	\begin{theorem}[Allard type regularity]\label{thm:teaser2}
		Let $g$ be a $C^1$ metric on $B_1 \subset \R^{n+1}$ (for $n \geq 1$), $\beta$ a $C^1$ function on $B_1$, $S \subset \R^n$, and $V$ a $(\beta, S)$-capillary varifold in $(B_1, g)$.  Given $\theta \in (0, \pi) \setminus \{ \pi/2 \}$ and $\eps > 0$, there are constants $\gamma(n, \theta)$, $\delta(n, \theta, \eps)$ so that if
		\begin{gather}
			\Theta_V(0, 1) + (\cos\beta(0))_- \leq 3/4 + |\cos\theta|/4  < 1 \label{eqn:teaser2-hyp1} \\
			0 \in \spt V_I, \quad \spt V_I \subset \{ x : d(x, P_{\theta}) < \delta \} \label{eqn:teaser2-hyp2} \\
			\max \{ ||H^{tan}_{V, g}||_{L^\infty(B_1)}, |g - g_{eucl}|_{C^1(B_1)}, |D\beta|_{C^0(B_1)} \} \leq \delta.  \label{eqn:teaser2-hyp3} 
		\end{gather}
		Then $\spt V_I \cap B_\eps \neq \emptyset$ and $\spt V_I \cap B_{\gamma}$ is a $C^{1, 1/2}$ hypersurface, which can be written as a graph over $P_{\theta}$ with $C^{1,1/2}$-norm $\leq \eps$, and which at any $\xi \in \spt V_I \cap B_{\gamma} \cap \{ x_1 = 0 \}$ meets $\{ x_1 = 0 \}$ at $g$-angle $\beta(\xi)$.
	\end{theorem}
	
	\begin{remark}
		There is nothing special about the Holder exponent $1/2$ here -- any number $\alpha \in (0, 1)$ would work (of course with $\delta$ then depending on $\alpha$ also).
	\end{remark}

        \begin{remark}
            While we were finishing the paper, Wang \cite{WangGaoming} obtained a similar regularity result as in Theorem \ref{thm:teaser2} for all angles $\theta\in (0,\pi)$, under the additional assumption that the capillary varifold $V$ has uniformly bounded first variation on both $V_I$ and $V_B$ (the conclusion from Theorem \ref{thm:teaser1}). Wang's proof is an extension of Simon's technique for cylindrical tangent cones \cite{Simon-cylindrical}. In comparison, ours is based on an $L^\infty$-excess decay estimate using a viscosity approach, which is more ``linear.''
        \end{remark}

	\begin{remark}\label{rem:no_pi/2}
		Although we have established first variation control for all angles in $(0, \pi)$, currently we can only prove $C^{1,\alpha}$-regularity for angles in $ (0, \pi) \setminus \{\pi/2\}$.  When $\beta = \pi/2$ there seem to be some non-trivial subtleties.  One reason is that in the regions where $\cos\beta = 0$ (which may contain an open subset of $\{x_1=0\}$), the portion $S$ of $V$ on the container boundary could ``flip'' from lying below $V_I$ to lying above $V_I$, and thereby flip the expected angle of contact from $\beta$ to $\pi-\beta$ (w.r.t. a fixed orientation).  Another possible behavior is that $\cos\beta$ transitions from $0$ to non-zero but the surface contact angle remains $\pi/2$.  By only assuming $V$ is varifold-close to $V^{(\pi/2)}$ it is unclear how to distinguish these behaviors.  It seems that an Allard type result for $V$ near $V^{(\pi/2)}$ may require additional assumptions, and is an interesting problem to investigate in future.
	\end{remark}
	
	Another, possibly more natural, formulation of our regularity theorem is in terms of varifold distance or convergence.  Let us define $\D$ to be a metric on the space of (signed) varifolds compatible with varifold convergence, see \eqref{eqn:D-def} and the nearby discussion for more details.  Then we have
	\begin{theorem}[Allard type regularity, second version]\label{thm:teaser3}
		In Theorem \ref{thm:teaser2}, in place of \eqref{eqn:teaser2-hyp1}, \eqref{eqn:teaser2-hyp2}, \eqref{eqn:teaser2-hyp3}, one could instead assume that
		\begin{gather}
			\D(V, V^{(\theta)}) \leq \delta, \\
			\max \{ ||H^{tan}_{V, g}||_{L^\infty(B_1)}, |g - g_{eucl}|_{C^1(B_1)}, |\beta - \theta|_{C^1(B_1)} \} \leq \delta.
		\end{gather}
	\end{theorem}
	
	\begin{remark}\label{rmk:triple.junction}
		The formulation of Theorem \ref{thm:teaser3} fails if we only assume $|D\beta|$ is small.  Consider the following example.  Define the rays in $\R^2$
		\begin{gather}
			l_1 = \{ x_1 = 0, x_2 \leq 0 \}, \quad l_2 = \{ -x_1 + \sqrt{3} x_2 = 0, x_1 \geq 0 \}, \\
			l_3 = \{ x_1 + \sqrt{3} x_2 = 0, x_1 \leq 0 \},
		\end{gather}
		which meet at $120^\circ$.  For any $s > 0$ define
		\begin{gather}
			V_{I, s} = \left( [l_1 - se_1] + [l_2 - se_1] + [l_3 - se_1] \right) \llcorner \{ x_1 < 0 \}, \\
			S_s = \{ x_1 = 0, x_2 < s/\sqrt{3} \}.
		\end{gather}
		Then each $V_s := V_{I, s} - (1/2)[S_s]$ is a $(2\pi/3, S_s)$-capillary varifold, and $V_s \to V^{(\pi/3)}$ as $s \to 0$, but the capillary angle remains $2\pi/3 \neq \pi/3$ for all $s$ and none of the $V_{I, s}$ are graphical over $P_{\pi/3}$.  In general examples like this can only occur when $\cos\beta(0) + 1 \approx \cos\theta$.
	\end{remark}

	We shall highlight two consequences of Theorems \ref{thm:teaser1}, \ref{thm:teaser2}, but first let us introduce a little extra notation.  Let us say $x \in \spt V_I \cap \R^n$ is regular if nearby $\spt V_I$ is a $C^{1,\alpha}$ hypersurface meeting $\R^n$ at $g$-angle $\beta$ or $\pi/2$.  We define the pointwise density function
	\[
	\Theta_{V, g}(x) := \lim_{r \to 0} \frac{\mu_V(B_r^g(x))}{\omega_n r^n}.
	\]
	In Section \ref{sec:mono} we show that $\Theta_{V, g}(x)$ always exists, is an upper-semi-continuous function of $x \in \{x_1 = 0\}$, and coincides with $\theta_{V, g}(x)$ at $|\mu_V|$-a.e. point $x \in B_1$.
	
	The first corollary recasts Theorem \ref{thm:teaser1} in terms of tangent cones.
	\begin{cor}[Flat capillary tangent cones implies regularity]\label{cor:teaser1}
		Let $g$ be a $C^1$ metric on $B_1$, $\beta$ a $C^1$ function on $B_1 \subset \R^{n+1}$, $S \subset \R^n$, and $V$ a $(\beta, S)$-capillary varifold in $(B_1, g)$ with $||H^{tan}_{V, g}||_{L^\infty(B_1)} < \infty$.  If at any $x \in \R^n$ some tangent cone of $V$ at $x$ is (up to rotation) $V^{(\beta(x))}$ and $\beta(x) \in (0, \pi) \setminus \{\pi/2\}$, then $x$ is a regular point of $V_I$.  (By Remark \ref{rem:tangent-cones} tangent cones for $V$ at $x$ always exists.)
	\end{cor}
	
	\begin{remark}
		Because one can always add a multiple of $[\R^n]$ to any free-boundary varifold, a point having density $(1 + \cos\theta)/2 \equiv \Theta_{V^{(\theta)}}(0)$ does not necessarily mean a tangent cone at that point will be $V^{(\theta)}$ (c.f. Lemma \ref{lem:tangent-cones}). Another such example is half of the Simons' cone, which is a free boundary varifold with density $\sqrt{2}/2$; thus, letting $\theta \coloneqq \arccos(\sqrt{2}-1)$, it is a conical $(\theta,\emptyset)$-capillary varifold with the same density of $V^{(\theta)}$.
	\end{remark}

	In the (relatively open) set of points in $\R^n$ where we have good density control, we have good first variation and boundary control, and from ideas of \cite{demasi} at $\haus^{n-1}$-a.e. boundary point that tangent cones to $V$ looks like $V^{(\beta(x))}$, $V^{(\pi/2)}$, or $V^{(\pi/2)} + \cos\beta(x)[\R^n]$ (see Theorem \ref{thm:refined}).  From this we obtain regularity at generic boundary points.
		\begin{theorem}[Regularity at generic boundary points]\label{thm:teaser4}
		Let $g$ be a $C^1$ metric on $B_1 \subset \R^{n+1}$, $\beta$ a $C^1$ function on $B_1$, $S \subset \R^n$, and $V$ a $(\beta, S)$-capillary varifold in $(B_1, g)$ with $||H^{tan}_{V, g}||_{L^\infty(B_1)} < \infty$.  Then the regular set of $\spt V_I \cap \R^n$ is relatively open and dense in
		\[
		\spt V_I \cap \R^n \cap \{ \Theta_{V, g}(x) + (\cos\beta(x))_- < 1 \},
		\]
		(which itself is relatively open in $\spt V_I \cap \R^n$), and the regular set has $\haus^{n-1}_g$-full measure (or equivalently $\sigma_{V, g}$-full measure) in
		\[
		\spt V_I \cap \R^n \cap \Big\{ \Theta_{V, g}(x) + (\cos\beta(x))_- < \min \{ 1, 1/2 + |\cos\beta(x)| \} \Big\} \cap \{ \beta(x) \neq \pi/2 \} .
		\]
	\end{theorem}

	\begin{remark}\label{rem:example-collision}
	Probably the regular set has $\haus^{n-1}_g$-full measure in $\spt V_I \cap \R^n \cap \{\Theta_{V, g}(x) + (\cos\beta(x))_- < 1 \}$ also, but we encounter two issues when trying to prove this.  The first is that we cannot yet prove an Allard-type regularity when $\beta(x) = \pi/2$ but $\beta$ is not identically $\pi/2$ nearby (c.f. Remark \ref{rem:no_pi/2}).  The second is that when $\beta \neq \pi/2$ we cannot rule out there being many points whose tangent cones are $V^{(\pi/2)} + \cos\beta(x)[\R^n]$, but which themselves are limits of capillary points, i.e. we cannot rule out that $\overline{\del^* S} \setminus \del^*S$ has positive $(n-1)$-measure in $\{ \beta \neq \pi/2\}$.
	
	The following examples illustrate this kind of subtlety.  Define $\bY \subset \R^3$ by
	\begin{gather}
	\bY = \Big(\{ x_2 = \sqrt{3} x_1, x_1 > 0 \} \cup \{ x_2 = - \sqrt{3} x_1, x_1 > 0 \} \cup \{x_2 = 0, x_1 < 0 \} \Big) \times \R.
	\end{gather}
 As a $2$-varifold in $(\R^3,g_{eucl})$, $\bY$ is stationary. Now let $A \subset \R$ be a fat Cantor set, i.e. a closed set of positive $\haus^1$-measure which is nowhere dense and Let $h : \R \to \R$ be a smooth function which is positive on $\R \setminus A$ and vanishes on $A$ with small derivative. Let us define the diffeomorphism of $\R^3$ given by $\Psi(x_1,x_2,x_3)=(x_1+h(x_3),x_2,x_3)$: then $\Psi(\bY)$ is stationary in $(\R^3,\Psi_\sharp g_{eucl})$ and the varifold $V=V_I+\beta S$ defined as 
	\[
	V_I =\Psi(\bY) \cap \{x_1<0\}, \quad S = \{ (x_1, x_2, y) : x_1 = 0, |x_2| > h(y) \sqrt{3} \}.
	\]
 is a $(\beta,S)$-capillary varifold in $(\{x_1<0\},\Psi_\sharp g_{eucl})$, where $\beta$ is small perturbation of the constant $\pi/6$, depending on how small the derivative of $h$ is.
Moreover $\del^*S$ are all regular capillary points (with tangent cones being rotations of $V^{(\beta(x))}$), while $\overline{\del^*S} \setminus \del^*S = \{0^2\}\times A = \sing V_I \cap \R^n$ has positive measure, and at every point $(0, y) \in \{0^2\}\times A$ the tangent cone to $V$ is $V^{(\pi/2)} + (\sqrt{3}/2)[\R^2]$.
	
	Although the singularities of $V$ have very ``non-optimal'' density $=1/2 + \sqrt{3}/2$, these examples are somewhat devious because $V_I$ can be made to look arbitrarily close to a free-boundary half-plane $[P_{\pi/2}]$ with multiplicity-one.  We expect this pathology cannot occur when $\Theta_V(0, 1) < 1$.

    On the other hand the bound $\Theta_V(x) < 1$ is optimal, in that Theorem \ref{thm:teaser4} can fail at points where $\Theta_V(x) \ge 1$.  For an example, take any $\beta \in (0, \pi/2)$ and consider the (conical, stationary) $(\beta, \R^2)$-capillary varifold in $(\R^3\cap\{x_1<0\}, g_{eucl})$ given by
	\[
	V = [\{x_2 = (\cot\beta) x_1, x_1<0\}] + [\{ x_2 = -(\cot\beta) x_1 , x_1<0\}].
	\]
	Then $\Theta_V(0, 1) = \Theta_V(0) = 1$ and $\{ x_1 = x_2 = 0 \} = \sing V_I$.
	\end{remark}

	\subsection{Idea of Proof and plan of the paper}
    Section \ref{sec:prelim} is dedicated to introducing some notation and the definition of capillary varifolds. We also recall some basic results (basic first variation control, monotonicity, Allard's theorems on compactness and regularity) which apply to capillary varifolds after some simple adaptation of classical results for free-boundary varifolds.
    
 \vspace{3mm}
    \textbf{First variation control.}
    The crux of the paper is to establish the first variation control of Theorem \ref{thm:teaser1}.
    The key step in the proof of Theorem \ref{thm:teaser1} is the Ahlfors-regularity of the boundary measure $\sigma_{V,g}$ which is deduced from a dichotomy result (Theorem \ref{thm:sigma-dichotomy}): roughly speaking, for a capillary  varifold satisfying the assumptions of Theorem \ref{thm:teaser1}, at any point $\xi\in\{x_1=0\}$, either $V_I$ vanishes in a small ball $B_{cr}(\xi)$ or $\sigma_{V,g}(B_r(\xi))\ge cr^{n-1}$.
    The proof of this dichotomy rests on a boundary rigidity result (Theorem \ref{thm:rigidity}), which can be seen as a \enquote{baby version} of Allard's theorem: if a stationary varifold in $\{x_1 < 0 \}$ has density $< 1$, and if the excess at scale $1$ with respect to the plane $\{x_1=0\}$ is small enough, then the varifold must vanish in a smaller ball.

    We prove Theorem \ref{thm:teaser1} in Section \ref{sec:first_var_control}. In particular, Subsection \ref{sec:rigidity} is dedicated to proving the boundary rigidity result and in Subsection \ref{sec:sigma-ahlfors} we prove the dichotomy result.
    Subsection \ref{sec:conv} is dedicated to proving an improved compactness and convergence result for capillary varifolds with small enough total mass (Lemma \ref{lem:conv}), while in Subsection \ref{subsec:refined_bdry} we prove a more refined structure for the boundary measure of capillary varifolds (Theorem \ref{thm:refined}). 
    
    \vspace{3mm} 
    \textbf{Regularity results.}
    The proof of Theorem \ref{thm:teaser2} is based on an $L^\infty$-excess decay estimate using a viscosity approach modeled on \cite{savin07, DePhil-Gas-Schu, desilva:fbreg}. 
    With techniques similar to ours (although without the complications arising from the lack of first variation control), the regularity for the one-phase free boundary problem with capillarity was addressed in the recent paper \cite{Ferreri-Tortone-Velichkov-2023}.

    In Section \ref{sec:max} we take the first step towards this result by proving a boundary maximum principle (Theorem \ref{thm:boundary-max}) for capillary varifolds.
    Roughly speaking, we prove that any smooth surface touching the support of a capillary varifold at a boundary point must do so without violating the expected condition that the varifolds attaches to $\{x_1=0\}$ with angle $\beta$. 
    However, the discussion after Definition \ref{def.capillary.varifolds} shows that, in general, one should not expect that a $(\beta,S)$ capillary varifold induced by a smooth surface attaches to $\{x_1=0\}$ with angle $\beta$ (since, for instance, $[\{x_{n+1}=0,x_1\le0\}]$ is a $(\beta,\emptyset)$-capillary varifold for any $\beta\in(0,\pi)$).
    Despite this, by exploiting Theorem \ref{thm:cap-first-var}, we may prove a soft classification of capillary tangent cones (Lemma \ref{lem:tangent-cones}) that allows us to exclude any pathological behavior if the varifold is close enough to the model capillary cone $V^{(\theta)}$.

    The key part of the proof of Theorem \ref{thm:teaser2} is a weak decay of oscillations (Theorem \ref{thm:decay-r}, which we prove in Section \ref{sec:harnack}), whose proof is based on ideas developed in \cite{desilva:fbreg} for the regularity of the one-phase free boundary. In some ways, capillary minimal surfaces behave more like solutions to the one-phase Bernoulli problem than minimal surfaces with free-boundary (compare also with \cite{ChoEdeLi2024}).

    With the above result at hand, the proof of the $\eps$-regularity results stated above are extensions of already existing results (see for instance \cite{savin07,DePhil-Gas-Schu}), which we carry out in Sections \ref{sec:decay} and \ref{sec:proofs-of-main}.

     

	\subsection{Acknowledgements}
	
	We thank Guido De Philippis, Luca Spolaor, and Bozhidar Velichkov for helpful conversations. 
 
    L.D.M was supported by  the \textsc{STARS - StG} project ``\textsc{QuASAR} - Question About Structure And Regularity of currents" and by PRIN project 2022PJ9EFL ``Geometric Measure Theory: Structure of Singular Measures, Regularity Theory and Applications in the Calculus of Variations" and partially supported by \textsc{INDAM-GNAMPA}. N.E. was supported by NSF grant DMS-2204301. C.G. was supported by the European Research Council (\textsc{ERC}), under the European Union's Horizon 2020 research and innovation program, through the project \textsc{ERC VAREG} - Variational approach to the regularity of the free boundaries (grant agreement No. 853404) and partially supported by \textsc{INDAM-GNAMPA}. C.L. was supported by NSF
    grant DMS-2202343 and a Simons Junior Faculty Fellowship.

\section{Preliminaries}\label{sec:prelim}

We work in $\R^{n+1}$, for $n \geq 1$. We will identify $\R^n \equiv \{0^1\}\times \R^n $ and $\R^{n-1} \equiv \{ 0^1 \} \times \R^{n-1} \times \{0^1\}$ as subspaces of $\R^{n+1}$.  Given an $n$-plane $P$ we always denote by $\pi_P$ the Euclidean-orthogonal projection to $P$, and $\pi_P^\perp$ the projection to $P^\perp$.  Given two $n$-planes $P$ and $Q$, with a small abuse of notation we let $|P-Q|\coloneqq |\pi_P-\pi_Q|$ in the operator norm. We write $\eta_{x, r}(y) := (y-x)/r$ for the translation/dilation map.

 Define
\[
P_\theta = \{ \cos\theta x_1 + \sin\theta x_{n+1} = 0 \}, \quad P^-_\theta = P_\theta \cap \{ x_1 \leq 0 \}
\]
to be the plane (and half-plane) meeting $\{ x_1 = 0 \}$ at angle $\theta$.  Write $\nu_\theta = \cos\theta e_1 + \sin\theta e_{n+1}$ for its unit normal.  We highlight $P_{\pi/2} = \{ x_{n+1} = 0 \}$.

It will be useful to define the squashed balls
\[
B^\theta_r(x) = \{ y \in \hor : (\sin \theta)^{-2} (y_1 - x_1)^2 + (y_2 - x_2)^2 + \ldots + (y_n - x_n)^2 < r^2 \},
\]
the reason being is that the (Euclidean) orthogonal projection $\pi_{\hor}(B_r(0) \cap P_\theta) = B^\theta_r(0)$.  Define the vertical cylinders $C^\theta_r(x) = B^\theta_r(x) \times \R \subset \R^{n+1}$.

Given an $n$-plane $P$ and a set $S$, define the oscillation function
\begin{equation}
\osc_P(S, B_r(x)) = \frac{1}{2} \sup \{ |\pi_P^\perp(y - z)| : y, z \in S \cap B_r(x) \} .
\end{equation}
If $A$ is any invertible linear map, and $z = z^T + z^\perp \in P + P^\perp$, we have
\begin{equation}
|\pi_{AP}^\perp(Az)| = |\pi_{AP}^\perp(Az^\perp)| \leq |Az^\perp| \leq |A| |\pi_P^\perp(z)| ,
\end{equation}
(where $|A|$ denotes the operator norm) and therefore we have
\begin{equation}\label{eqn:osc-A}
\osc_{AP}(AS, B_r) \leq |A|\osc_P(S, B_{|A^{-1}|r}) .
\end{equation} 

We will often endow (subset of) $\R^{n+1}$ with a $C^1$ metric $g$.  In this case we write $B_r^g(x)$ for the $g$-geodesic ball of radius $r$ centered at $x$, $\haus^k_g$ for the $k$-dimensional Hausdorff measure w.r.t. $g$, $\pi_{P, g}$ for the $g$-orthogonal projection onto a plane $P$, $\nabla^g$ for the $g$-Levi-Civita connection, etc.  However we emphasize that \emph{unless $g$ is explicitly stated} we take distances, angles, volumes, projections, etc. with respect to the Euclidean metric.
We define the {$g$-unit normal to $\{x_1=0\}$ pointing in the positive $e_1$ direction} as
\begin{equation}\label{eq:unit_normal_definition}
    \nu_g = \frac{\pi^\perp_{\R^n,g}e_1}{|\pi^\perp_{\R^n,g}e_1|}.
\end{equation}

\subsection{Slanted graphs}\label{subsection:slanted_graphs}

For geometric convenience, we will be working with (small) graphs over $P_\theta$, but which push in the $e_{n+1}$ direction and are defined as functions in $x' \equiv (x_1, \ldots, x_n) \in \hor$.  Specifically, if $u(x_1, \ldots, x_n) : U \subset \hor \to \R$ satisfies $|u|_{C^2} \leq 1$, then we will be considering the surface
\[
G_\theta(u) := \{ (x', -\cot\theta x_1 + u(x')) : x' \in U \}.
\]
It will also be convenient to consider $G_\theta(u)$ as sitting inside $\R^{n+1}$ endowed with a general metric $g$.

By direct computation, the mean curvature scalar of $G_\theta(u)$ with respect to an ambient metric $g$ and in the $e_{n+1}$ direction, as a function of $x' \in U$, is
\[
H_\theta(u, g) = \cL_\theta(u) + R_1(\theta, u) + R_{2}(\theta, u, g),
\]
where $\cL_\theta$ is the linear operator
\[
\cL_\theta(u) = \Delta u - \cos^2 \theta D_1^2 u,
\]
$R_{1}$ is a smooth function in $\theta, Du, D^2 u$ which is quadratic in $Du, D^2u$, and $R_{2}$ is a smooth function of $\theta, g, Dg, u, Du, D^2 u$ which vanishes when $g = g_{eucl}$.
In other words:
\[
R_{1} = R_{1,1} \star Du \star Du + R_{1,2} \star Du \star D^2u, \quad R_2 = R_{2, 1} \star (g - g_{eucl}) + R_{2, 2} \star Dg
\]
for smooth functions $R_{1, i}(\theta, Du, D^2u)$, $R_{2, i}(\theta, g, Dg, u, Du, D^2 u)$.

A further computation shows that the unit normal $\nu_{\theta}(u, g)$ of $G_\theta(u)$, with respect to an ambient metric $g$ and in the $e_{n+1}$ direction, satisfies
\[
\nu_{\theta}(u, g) \cdot e_1 = \cos\theta + D_1 u (\cos^2 \theta \sin\theta - 1) + \tilde R_{1}(u) + \tilde R_{2}(u, g)
\]
where $\tilde R_{1}$ is a smooth function of $\theta, Du$ which is quadratic in $Du$; and $\tilde R_{2}$ is a smooth function of $\theta, g, u, Du$ which vanishes when $g = g_{eucl}$:
\[
\tilde R_1 = \tilde R_{1,1} \star Du \star Du, \quad \tilde R_2 = \tilde R_{2,1} \star (g - g_{eucl}),
\]
for smooth functions $\tilde R_{1,1}(\theta, Du)$, $\tilde R_{2, 1}(\theta, g, u, Du)$.

\subsection{(Signed) varifolds in an ambient metric}

Let us recall that if $g$ is a $C^1$ metric on $U \subset \R^{n+1}$, then an \textbf{$n$-rectifiable varifold $V$ in $(U, g)$} is a positive Radon measure on the Grassmannian bundle $U \times \mathrm{Gr}_n(n+1)$ of the form
\begin{equation}\label{eqn:varifold}
V(\phi(x, S)) = \int_{M_V} \phi(x, T_x M_V) \theta_{V, g}(x) \dif\haus^n_g(x) \quad \forall \phi \in C^0_c(U \times \mathrm{Gr}_n( n+1)),
\end{equation}
where: $M_V$ is an $n$-rectifiable set in $U$, $\theta_{V, g}$ is a non-negative $\haus^n_g \llcorner M_V$-measurable function, and $\haus^n_g$ is the $n$-dimensional Hausdorff measure with respect to $g$.  
We say $V$ is \textbf{integral} if, in addition to being rectifiable, $\theta_{V, g} \in \N$ at $\mu_V$-a.e.  If $g = g_{eucl}$ we will often write $\theta_V$ in place of $\theta_{V, g_{eucl}}$.

Associated to $V$ is the mass measure $\mu_V$, which is a Radon measure on $U$ defined by $\mu_V := \pi_\sharp V$, where $\pi : U \times \mathrm{Gr}_n(n+1) \to U$ is the projection map.  Equivalently, 
\begin{equation}\label{eqn:varifold-mass}
\mu_V(\phi) = \int \phi(x) \theta_{V, g}(x) d\haus^n_g(x) \quad \forall \phi \in C^0_c(U).
\end{equation}
If $M$ is a $n$-dimensional Lipschitz submanifold of $U$, we write $[M]_g = \haus^n_g \llcorner M \oplus \delta_{T_x M}$ for the (multiplicity-one) varifold obtained by integrating $M$ against $\haus^n_g$, and when $g = g_{eucl}$ we will for shorthand write $[M] := [M]_{g_{eucl}} \equiv \haus^n \llcorner M \oplus \delta_{T_x M}$.

Given a proper $C^1$ map $f : (U, g) \to (U', g')$, the pushforward $f_\sharp V$ is defined by
\[
(f_\sharp V)(\phi(x, S)) = \int_{M_V} \phi(f(x), Df|_x(T_x M_V)) Jf|_{(x, T_xM_V)} \theta_{V,g}(x)    d\haus^n_g
\]
where $Jf|_{x, S} = \det( g'(Df|_x(e_i), Df|_x(e_j)))^{1/2}$ is the Jacobian of $f$, the $\{ e_i \}$ being a $g$-orthonormal basis of $S \subset \R^{n+1}$.  The area formula implies 
\[
(f_\sharp V)(\phi(x, S)) = \int_{f(M_V)} \phi(x, T_x f(M_V)) \theta_{V,g}(f^{-1}(x)) d \haus^n_{g'}(x).
\]
And we remark that if $f$ is an isometry, then
\[
\mu_{f_\sharp V}(A) = \mu_V(f^{-1}(A)) \equiv (f_\sharp \mu_V)(A) \quad \forall A \subset U'.
\]

Given an $n$-plane $P$ and a varifold $V$, we define
\begin{align}
\osc_P(V, B_r(x)) &:= \osc_P(\spt V, B_r(x)),
\end{align}
where $\spt V \subset \R^{n+1}$ is the support of the mass measure $\mu_V$.

\vspace{3mm} 

It will useful to talk about varifolds with possibly negative density on some regions.  Let us say $V$ is a \textbf{signed $n$-rectifiable varifold in $(U, g)$} if we can write $V = V_+ - V_-$ for $n$-rectifiable varifolds $V_+, V_-$ in $(U, g)$ satisfying $V_+ \perp V_-$.  Equivalently, $V$ can be thought of as a signed Radon measure on the Grassmannian bundle of the same form as \eqref{eqn:varifold}, but now $\theta_{V, g} \in L^1_{loc}( \haus^n_g \llcorner M_V)$ is an $\R$-valued function.

Note that if $V$ is a signed varifold, then the pushforward measure $\mu_V = \pi_\sharp V$ is now a \emph{signed} mass measure, having the same form as \eqref{eqn:varifold-mass}.  Let us write $|\mu_V|$ for the \emph{unsigned} or total mass measure:
\[
|\mu_V|(\phi) = \int_{M_V} \phi |\theta_{V, g}| d\haus^n_g.
\]
We can define the pushforward $f_\sharp V := f_\sharp V_+ - f_\sharp V_-$.

We say signed Radon measures $\mu_i \to \mu$ in an open set $U$ if $\mu_i(\phi) \to \mu(\phi)$ for all $\phi \in C^0_c(U)$, and correspondingly we say signed varifolds $V_i \to V$ if they converge as signed Radon measures.  Note that it it not necessarily the case that $V_{+, i} \to V_+$ and $V_{-, i} \to V_i$.

\vspace{3mm}

For $V, W$ signed $n$-varifolds in $B_1$, define the varifold metric 
\begin{equation}\label{eqn:D-def}
\D(V, W) = \sum_{\alpha \in \N} 2^{-\alpha} \frac{ |V(\phi_\alpha) - W(\phi_\alpha)|}{1+|V(\phi_\alpha) - W(\phi_\alpha)|}
\end{equation}
where $\{\phi_\alpha\}_\alpha$ is a dense subset of $C^0_c(B_1) \times \mathrm{Gr}_n(n+1)$.  Let us further arrange the $\{\phi_\alpha\}_\alpha$ to satisfy the two additional properties: for every $r = 1-1/k$ (for $k \in \N$), some subset of $\{ \phi_\alpha \}_\alpha$ forms a dense subset of $C^0_c(B_r) \times \mathrm{Gr}_n(n+1)$; and for every such $r$ there is a $\phi_\alpha \in C^0_c(B_{\sqrt{r}})$ satisfying $\phi_\alpha \equiv 1$ on $B_{r}$ and $|\phi_\alpha| \leq 1$. 

Then $\D$ is a metric on the space of signed $n$-varifolds in $B_1$ which is compatible with varifold convergence in the following sense: If $V_i, V$ are signed $n$-rectifiable varifolds with negative parts uniformly bounded by  $\Lambda$ for some $\Lambda \geq 0$, then $\D(V_i, V) \to 0$ if and only if $V_i \to V$ as varifolds in $B_1$.

To see this, let $V_i, V$ be as above, and suppose $\D(V_i, V) \to 0$.   Fix $\phi \in C^0_c(B_1) \times\mathrm{Gr}_n(n+1)$, $\eps > 0$, and choose an $r = 1-1/k$ ($k \in \N$) so that $\spt \phi \subset B_r$.  First observe that, taking $\phi_\alpha \in C^1_c(B_{\sqrt{r}})$ which $\equiv 1$ on $B_r$ we have
\[
|\mu_{V_i}|(B_r) \leq V_i(\phi_\alpha) + \Lambda \leq 4^\alpha \D(V_i, V) + V(\phi_\alpha) + \Lambda \leq 1+2\Lambda + |\mu_V|(B_{\sqrt{r}}).
\]
Now choose a $\phi_\alpha \in C^0_c(B_r)$ with $|\phi_\alpha - \phi|_{C^0} \leq \eps$, and then get for $i \gg 1$:
\begin{align}
|V_i(\phi) - V(\phi)| &\leq \eps |\mu_{V_i}|(B_r) + \eps |\mu_V|(B_r) + |V_i(\phi_\alpha) - V(\phi_\alpha)| \\
&\leq (1+3\Lambda + |\mu_V|(B_{\sqrt{r}}))\eps + 4^\alpha \D(V_i, V) \\
&\leq (2 + 3\Lambda + |\mu_V|(B_{\sqrt{r}})) \eps.
\end{align}
This shows $V_i(\phi) \to V(\phi)$.  On the other hand, if we assume $V_i \to V$, then it's easy to see $\D(V_i, V) \to 0$ also.

\vspace{3mm}

If $V$ is a signed $n$-rectifiable varifold in $(U, g)$, and $X$ is a compactly supported $C^1$ vector field in $U$, then $\{id+tX:(U,g)\to(U,g)\}_{t\in(-\eps,\eps)}$ is a family of diffeomorphisms of $(U,g)$ into itself and we may define the \textbf{$g$-first variation of $V$} as
\begin{align}
\delta_g V(X)&\coloneqq \frac{d}{dt}\Big|_{t = 0} \mu_{(id + t X)_\sharp V}(U) = \int_{M_V} \dive_{V, g}(X) \theta_{V, g}(x) d\haus^n_g(x)
\end{align}
where $\dive_{V, g}(X)|_x = \sum_i g(e_i, \nabla^g_{e_i} X)$ is the tangential $g$-divergence of $X$, with $e_i$ being a $g$-orthonormal basis of $T_x M_V$ and $\nabla^g$ being the Levi-Civita connection with respect to $g$. 
If $g = g_{eucl}$ we will often simply write $\delta V \equiv \delta_{g_{eucl}} V$.  $V$ is said to have \textbf{locally-finite $g$-first-variation} if 
\[
|\delta_g V(X)| \leq C(W) |X|_{C^0(W)} \quad \forall X \in C^1_c(W, \R^{n+1}), \forall W \subset\subset U.
\]
We say $V$ is \textbf{stationary} if $\delta_g V(X) = 0$ for all $X \in C^1_c(U, \R^{n+1})$.

In general, if $V$ has locally finite $g$-first-variation then $V$ will not also have locally-finite $g_{eucl}$-first-variation.  However, in many variational arguments it will be useful to treat $g$ as a perturbation of $g_{eucl}$, and we record the inequality
\begin{align}
&\big|\mdiv_{V, g}(X)|_x -\mdiv_{V,g_{eucl}}(X)|_x\big|\\
&\qquad= \bigg|\sum_{i, j} g^{ij} g(\nabla^g_{\tau_i} X, \tau_j) -\mdiv_{V,g_{eucl}}(X)\bigg|\nonumber \\
&\qquad\le  c(n, |g|_{C^0}) (|Dg| |X| + |g - g_{eucl}| |D_V X|) \label{eqn:div}
\end{align}
where $\{ \tau_i\}$ is any $g_{eucl}$-orthonormal basis for $T_x M_V$, $g^{ij}$ is the inverse of $g(\tau_i, \tau_j)$, and $|D_V X|^2 := \sum_i |D_{\tau_i} X|^2$.

By the Riesz representation theorem, if $V$ has locally-finite $g$-first variation in $U$ then $||\delta_g V||$ is a Radon measure in the sense that
\[
\delta_g V(X) = \int g(\nu_{V, g}, X) \dif||\delta_g V||
\]
for some $||\delta_g V||$-measurable, $g$-unit vector field $\nu_{V, g}$.  By the Radon-Nikodyn theorem, we can further decompose $\delta_g V$ by finding $H_{V, g} \in L^1_{loc}(U, \R^{n+1}; |\mu_V|)$, a Radon measure $\sigma_{V, g} \perp |\mu_V|$, and a $\sigma_{V, g}$-measurable $g$-unit vector field $\eta_{V, g}$ so that
\begin{equation}\label{eq:definition_of_H_and_sigma}
\delta_g V(X) = - \int g(H_{V, g}, X) \dif\mu_V + \int g(\eta_{V, g}, X) \dif\sigma_{V, g}.
\end{equation}
Here $H_{V, g}$ is called the (generalized) $g$-mean curvature, and $\sigma_{V, g}$ the (generalized) $g$-boundary measure, with $g$-conormal $\eta_{V, g}$.  Any stationary $V$ has $H_{V, g} = 0$ and $\sigma_{V, g} = 0$.

The first variation transforms nicely under isometries and scaling.  If $f : (U, g) \to (U', g')$ is an isometry, and $V' = f_\sharp V$, then
\begin{equation}\label{eqn:delta-isometry}
\delta_{g'} V'(X) = \delta_g V (Df^{-1} (X \circ f)) .
\end{equation}
In particular, if $\delta_g V$ is locally finite then $\delta_{g'} V'$ is also locally-finite, and $H_{V', g'}|_x = Df H_{V, g}|_{f^{-1}(x)}$, $\eta_{V', g'}|_x = Df \eta_{V, g}|_{f^{-1}(x)}$, $\sigma_{V', g'}(\phi) = \sigma_{V, g}(\phi \circ f^{-1})$.

If instead $f : (U, g) \to ( R U, g'(x) = g(x/R))$ has the form $f(x) = Rx$ (for $R > 0$ a constant), so that $f^* g' = R g$, then \eqref{eqn:delta-isometry} still holds, but now we have $\mu_{V'}(\phi) = R^n \mu_{V} (\phi \circ f^{-1})$, $H_{V', g'}|_x = R^{-1} H_{V, g}|_{x/R}$, $\eta_{V', g'}|_x = \eta_{V, g}|_{x/R}$, and $\sigma_{V', g'}(\phi) = R^{n-1} \sigma_{V, g}(\phi \circ f^{-1})$.

\vspace{3mm}

If $V$ is a signed $n$-rectifiable varifold in $(B_1, g)$ supported in $\{ x_1 \leq 0 \}$, we say \textbf{$V$ has $g$-free-boundary in $\{ x_1 = 0 \}$} if there is an $H_{V, g}^{tan} \in L^1_{loc}(B_1, \R^{n+1}; |\mu_V|)$ so that $H_{V,g}^{tan}(x)$ is tangential to $\{x_1=0\}$ for $|\mu_V|$-a.e. $x\in\{x_1=0\}$ and
\[
\delta_g V(X) = -\int g(H_{V, g}^{tan}, X) d\mu_V \quad \forall X \in C^1_c(B_1, \R^{n+1}) \text{ tangential to $\{ x_1 = 0 \}$}.
\]
Equivalently, we ask that when restricted to tangential vector fields, then $\delta_g V$ is locally-finite and $\delta_g V \ll |\mu_V|$ as Radon measures.  We say a $g$-free-boundary $V$ is \textbf{stationary} if $\delta_g V(X) = 0$ for all $X \in C^1_c(B_1, \R^{n+1})$ tangential to $\{ x_1 = 0 \}$.

By a relatively straightforward modification of \cite{GrJo} (see also \cite{demasi}), for a $g$-free-boundary varifolds $V$, with the additional assumption that $\mu_V\llcorner\{x_1>0\}$ is a positive measure, the $g$-first variation $\delta_g V$ will be locally-finite in $B_1$ for all vector fields, and will admit the decomposition
\[
\delta_g V(X) = -\int g(H_{V, g}^{tan}, X) d\mu_V - \int_{\{x_1 = 0\}} g(H_{\R^n, g}, X) d\mu_V + \int g(\nu_g, Y) d\sigma_{V, g}
\]
where now $H_{\R^n, g}$ is the mean curvature vector of $\R^n \subset (\R^{n+1}, g)$, $\sigma_{V, g}$ is supported in $\R^n$, and $\nu_g$ is defined in \eqref{eq:unit_normal_definition}.  In other words, the ``total'' mean curvature of $V$ is $H_{V, g} := H^{tan}_{V, g} + H_{\R^n, g}$, and the generalized conormal is $g$-normal to the barrier.  For the reader's convenience we provide a proof in the Appendix, Theorem \ref{thm:fb-first-var}.  Note that a stationary $g$-free-boundary varifold only has $H^{tan}_{V, g} = 0$ -- the normal mean curvature $H_{\R^n, g}$ and boundary measure $\sigma_{V, g}$ might not vanish.

In this paper we will be dealing almost exclusively with varifolds in $B_1$ and supported in $\{ x_1 \leq 0 \}$. A standard convention for us will be to write $V_I := V \llcorner \{ x_1 < 0 \}$ for the interior piece, and $V_B := V \llcorner \{ x_1 = 0 \}$ for the boundary piece, where here and in the following we use the small abuse of notation $V\llcorner U \coloneqq V\llcorner (U\times\Gr_{n}(n+1))$.  Let us also remark that when $g = g_{eucl}$, we will often drop the notational dependence on $g$, e.g. by writing $H_V$ and $\sigma_V$ in place of $H_{V, g}$, $\sigma_{V, g}$.

\subsection{Capillary varifolds}\label{subsection.capillary.varifolds}

Our primary objects of interest are varifolds which (distributionally) meet the barrier $\{x_1 = 0 \}$ at a prescribed, potentially changing angle $\beta$.  As pointed out in the introduction, a natural way to impose this condition is through the Euler-Lagrange equation of the Gauss free energy.  

\begin{definition}[Capillary varifold]
Let $g$ be a $C^1$ metric on $B_R$, $\beta$ a $C^1$ function on $B_R \cap \{x_1 = 0 \}$, and $S \subset \{ x_1 = 0 \}$. 
We say a $V$ is a $(\beta, S)$-capillary varifold in $(B_R, g)$ if $V$ is a signed rectifiable $n$-varifold in $(B_R, g)$ of the form $V = V \llcorner \{ x_1 < 0 \} + \cos\beta [S]_g =: V_I + V_B$, for $V_I$ integral in $\{ x_1 < 0 \}$, and which satisfies the variational identity\footnote{One could also allow a priori tangential mean curvature on $V_B$.  In this case our regularity theory would essentially be unchanged, and in fact we would get that a posteriori the tangential mean curvature vanishes along $\{ x_1 = 0 \}$.  However as critical points of the Gauss free energy \eqref{eqn:gauss.free.energy} one will only get a priori tangential mean curvature on $V_I$.} 
\begin{align}
\delta_{\beta, g} V(X) &:= \int \mdiv_{V, g}(X) d\mu_{V_I} + \int_S \mdiv_{\R^n, g}(X \cos\beta) d\haus^n_g \\
&= - \int g(H^{tan}_{V, g}, X) d\mu_{V_I} \quad \forall X \in C^1_c(B_R, \R^{n+1}) \text{ tangential to $\{ x_1 = 0 \}$},
\end{align}
for some $H^{tan}_{V, g} \in L^1_{loc}(\mu_{V_I})$.  We say a $(\beta, S)$-capillary varifold is stationary if $H^{tan}_{V, g} = 0$.
\end{definition}


(Note that in our definition $V$ is the \emph{total} varifold composed of both the interior piece $V_I$ and the weighted portion of the boundary $V_B = \cos\beta[S]_g$.)

For any $\theta \in \R$, the ``model'' capillary (stationary) varifold $V^{(\theta)}$ is the varifold corresponding to a half-plane $P_\theta^-$ meeting $\R^n$ at angle $\theta$:
\[
V^{(\theta)} := [P_\theta^-] + \cos\theta [x_1 = 0, x_{n+1} < 0 ]
\]

A useful fact that follows trivially from the divergence theorem is that if $V = V_I + \cos\beta[S]_g$ is a $(\beta, S)$-capillary varifold, then $V = V_I + \cos(\pi-\beta)[\R^n \setminus S]_g$ is a $(\pi - \beta, \R^n \setminus S)$-capillary varifold.

Let us also observe that
\begin{align}
\delta_{\beta, g} V(X) &= \delta_g V_I(X) + \delta_g [S]_g (\cos\beta X) \label{eqn:beta-g-vs-g} \\
&= \delta_{g} V(X) + \int_S g(X, \nabla^g_{\R^n} \sin\beta) d\haus^n_g \label{eqn:beta-g-vs-g-2} \\
&=:  \delta_{g} V(X) + \int g(X, H_{\beta, g}) d\mu_{V_B} \label{eqn:beta-g-vs-g-3}
\end{align}
for $H_{\beta, g} = \frac{\nabla^g_{\R^n} \sin\beta}{\cos\beta}$ wherever $\cos\beta \neq 0$ and $H_{\beta, g} = 0$ wherever $\cos\beta = 0$.  Since trivially
\begin{equation}\label{eqn:H_beta-bound}
||H_{\beta, g}||_{L^1(|\mu_{V_B}|)} \leq \int_S \frac{|\nabla^g_{\R^n} \sin\beta|}{|\cos\beta|} |\cos\beta|d\haus^n_g \leq c(n, |g|_{C^0(B_1)}) |D\beta|_{C^0(B_1)} < \infty, 
\end{equation}
being a capillary varifold in is equivalent to having $g$-free-boundary in $\{ x_1 = 0 \}$, with certain prescribed tangential mean curvature on the barrier $\{ x_1 = 0 \}$.  In fact when $\beta$ is constant then any capillary varifold is a free-boundary varifold with only interior tangential mean curvature.  Conversely, let us remark that if $V_I = V_I \llcorner \{ x_1 < 0 \}$ has free-boundary in $\{ x_1 = 0 \}$, then $V_I$ is a $(\beta, \emptyset)$-capillary varifold for any $\beta$.

Being $g$-free-boundary varifolds, capillary varifolds have good (total) first variation control, which we state explicitly below.  Note that $\sigma_{V, g}$ is the boundary measure for $V$ in the usual sense, since $H_{\beta, g} \in L^1(|\mu_V|)$.
\begin{theorem}\label{thm:fb-cap-first-var}
Let $g$ be a $C^1$ metric on $B_1$, $\beta$ a $C^1$ function on $B_1$, and $V$ a $(\beta, S)$-capillary varifold in $(B_1, g)$.  Then $\delta_{\beta, g} V$, $\delta_g V$ are locally locally-finite in $B_1$, and there exists a Radon measure $\sigma_{V, g} \perp |\mu_V|$ on $B_1$ supported in $\{ x_1 = 0 \}$ so that for all $Y \in C^1_c(B_1, \R^{n+1})$ we have
\begin{align}
\delta_{\beta, g} V(Y) &= - \int g(H^{tan}_{V, g}, Y) d\mu_{V_I} - \int g(H_{\R^n, g}, Y) d\mu_{V_B} + \int g(\nu_g, Y) d\sigma_{V, g}, \\
\delta_g V(Y) &= -\int g(H^{tan}_{V, g}, Y) d\mu_{V_I} - \int g(H_{\R^n, g} + H_{\beta, g}, Y) d\mu_{V_B} + \int g(\nu_g, Y) d\sigma_{V, g}, 
\end{align}
where $\nu_g$ is defined in \eqref{eq:unit_normal_definition} and $H_{\R^n, g}$ is the $g$-mean curvature vector of $\R^n \subset (\R^{n+1}, g)$.

For any $\gamma < 1$ we have the boundary measure bound
\[
\sigma_{V, g}(B_\gamma) \leq c ||H^{tan}_{V, g}||_{L^1(B_1, \mu_{V_I})} + c \mu_{V_I}(B_1),
\]
and in particular, for $Y \in C^1_c(B_\gamma, \R^{n+1})$ we have the bounds
\begin{align}
\delta_{\beta, g} V(Y) &\leq c\left(  ||H^{tan}_{V, g}||_{L^1(B_1, \mu_{V_I})} + |\mu_V|(B_1) \right) |Y|_{C^0}, \\\
\delta_g V(Y) &\leq c\left(  ||H^{tan}_{V, g}||_{L^1(B_1, \mu_{V_I})} +  |\mu_V|(B_1) +  |D\beta|_{C^0(B_1)} \right) |Y|_{C^0}.
\end{align}
In all the above equations $c = c(n, \gamma, |g|_{C^1(B_1)})$.
\end{theorem}

\begin{proof}
Follows directly from Theorem \ref{thm:fb-first-var}, using the formulation \eqref{eqn:beta-g-vs-g-3} of $V$ as a varifold with $g$-free-boundary, and the bound \eqref{eqn:H_beta-bound}.
\end{proof}

From the identities \eqref{eqn:beta-g-vs-g} and \eqref{eqn:delta-isometry}, capillary varifolds transform nicely under isometries and scaling.  If $f : (B_1, g) \to (B_1, g')$ is an orientation-preserving isometry mapping $f(\R^n) = \R^n$, and $V$ is a $(\beta, S)$-capillary varifold in $(B_1, g)$, then $f_\sharp V$ is a $(\beta \circ f^{-1}, f(S))$-capillary varifold in $(B_1, g')$.  And as one would expect, if we write $V' = f_\sharp V$, $\beta' = \beta \circ f^{-1}$, $S' = f(S)$, then $H^{tan}_{V', g'} = Df H^{tan}_{V, g}$, $\eta_{V', g'} = Df \eta_{V, g}$, and $\sigma_{V', g'}(\phi) = \sigma_{V, g}(\phi \circ f^{-1})$.

Similarly, if instead $f(x) = Rx : (B_1, g) \to (B_R, g'(x) = g(x/R))$ is scaling, then $V'$ is a $(\beta', S')$-capillary varifold in $(B_R, g')$, but now $\mu_{V'}(\phi) = R^n \mu_V(\phi \circ f^{-1})$, $H^{tan}_{V', g'}|_x = R^{-1} H^{tan}_{V, g}|_{x/R}$, $\eta_{V', g'}|_x = \eta_{V, g}|_{x/R}$, and $\sigma_{V', g'}(\phi) = R^{n-1}\sigma_{V, g}(\phi \circ f^{-1})$.

\vspace{3mm}

Capillary varifolds have good compactness (unlike general $n$-rectifiable signed varifolds), with the only caveat that the limit may not itself be a capillary varifold.
\begin{lemma}\label{lem:soft-conv}
Let $g_i$ be a sequence of $C^1$ metrics on $B_1$, $\beta_i$ a sequence of $C^1$ functions on $B_1$, and $V_i$ a sequence of $(\beta_i, S_i)$-capillary varifolds in $(B_1, g_i)$.  Suppose that
\begin{gather}
\Lambda_i := \max \{ ||H_{V_i, g_i}||_{L^\infty(B_1)}, |g_i - g_{eucl}|_{C^1(B_1)}, |D\beta_i|_{C^0(B_1)} \} \to 0, \\
\sup_i \mu_{V_{I, i}}(B_1) < \infty.
\end{gather}

Then (after passing to a subsequence) there is a $\beta_0 \in \R$ and an $n$-rectifiable signed varifold $V$ in $(B_1, g_{eucl})$, which is stationary with free-boundary in $\{ x_1 = 0 \}$, satisfying
\begin{gather}
\theta_V(x) \geq -(\cos\beta_0)_- \text{ $|\mu_V|$-a.e. $x_1 = 0$}, \quad \text{ $V_I$ is integral}, \label{eqn:soft-conv-concl1} \\
\Theta_V(0, 1) \leq \liminf_i \Theta_{V_i}(0, 1),
\end{gather}
so that: $V_i \to V$ as (signed) varifolds in $B_1$, $\delta_{\beta_i, g} V_i \to \delta V$ as Radon measures, and $\sigma_{V_i, g_i} \to \sigma_V$ as Radon measures.  Additionally, for every $x \in B_1$ and a.e. $r \in (0, 1-|x|)$ we have $\mu_{V_i}(B_r(x)) \to \mu_V(B_r(x))$.
\end{lemma}

\begin{remark}
Note $V$ may not be a capillary varifold, i.e. $V \llcorner \{ x_1 = 0 \}$ may not have the form $\cos \beta [S]$ for some $S \subset \R^n$.  See Remark \ref{rem:teaser1-sharp}.
\end{remark}

\begin{remark}\label{rem:soft-conv-density}
Though we will not require it, we remark that by a minor modification of the proof of integrality (\cite[Chapter 8]{Sim}) one can in fact show
\[
\theta_V(x) \in (\N_0 + \cos\beta_0) \cup \N  \text{ $|\mu_V|$-a.e. $x_1 = 0$}.
\]
This was also observed by \cite{WangGaoming}.
\end{remark}

\begin{proof}[Proof of Lemma \ref{lem:soft-conv}]
Passing to a subsequence, we can assume $\cos\beta_i \to \cos\beta_0$ in $C^1(B_1)$.  Let us define the varifolds $W_i = V_i + (1+(\cos\beta_0)_-)[\R^n]$.  Then each $W_i$ is $n$-rectifiable with $\theta_{W_i, g_i}(x) \geq 1 - \eps_i$ at $\mu_{W_i}$-a.e. $x$ (for numbers $\eps_i \to 0$), and $|\mu_{V_i}| \leq \mu_{W_i}$, and we have the mass bounds
\[
\sup_i \mu_{W_i}(B_1) \leq \sup_i \mu_{V_{I, i}}(B_1) + c(n) < \infty,
\]
and from Theorem \ref{thm:fb-cap-first-var} for any $\gamma < 1$ and $X \in C^1_c(B_\gamma, \R^{n+1})$ we have the first variation control
\[
|\delta_{g_i} W_i(X)| = |\delta_{g_i} V_i(X)| \leq c(1 + \Lambda_i) (\sup_i \mu_{W_i}(B_1)) |X|_{C^0},
\]
where $c = c(n, \gamma)$.

Passing to a further subsequence we can therefore apply the standard compactness theory for rectifiable $n$-varifolds (Theorem \ref{thm:allard-compact}) to deduce $W_i \to W$ for some $n$-rectifiable varifold $W$ having $\theta_W(x) \geq 1$ for $\mu_W$-a.e. $x$.  Define $V = W - (1+(\cos\beta_0)_-)[\R^n]$, so that $V$ is a signed $n$-rectifiable varifold in $B_1$ and $V_i \to V$ as varifolds.  The lower bound $\theta_V(x) \geq -(\cos\beta_0)_-$ follows from the lower bound on $W$; integrality of $V_I$ follows from integrality of the $V_{I, i} \equiv W_{I, i}$; and we have the bound
\begin{align}
\Theta_V(0, 1) &= \Theta_W(0, 1) - 1 - (\cos\beta_0)_- \\
&\leq \liminf_i \Theta_{W_i}(0, 1) - 1 - (\cos\beta_0)_- = \liminf_i \Theta_{V_i}(0, 1).
\end{align}
Since $\mu_{W_i}(B_r(x)) \to \mu_W(B_r(x))$ for every $x$ and a.e. $r \in (1-|x|)$, and $\haus^n(\del B_r(x) \cap \R^n) = 0$ always, we get the corresponding statement for $\mu_{V_i}$ and $\mu_V$ also.

The convergence $\delta_{g_i} V_i \to \delta V$ and $\delta_{\beta_i, g_i} V_i \to \delta V$ follows from varifold convergence, the identities \eqref{eqn:div}, \eqref{eqn:beta-g-vs-g-2}, and our assumptions on $g_i, \beta_i$.  Since for any $X \in C^1_c(B_1, \R^{n+1})$ tangential to $\{ x_1 = 0 \}$ we have
\[
|\delta_{\beta_i, g_i} V_i(X)| \leq \Lambda_i \Big(\sup_i |\mu_{V_i}|(B_1)\Big) |X|_{C^0} \to 0,
\]
we have $\delta V(X) = 0$, so $V$ is stationary with free-boundary in $\{ x_1 = 0 \}$.

Lastly, we verify the convergence $\sigma_{V_i, g_i} \to \sigma_V$.  For any given metric $g$ on $B_1$, let $\nu_g$ denote the $g$-unit normal of $\{x_1=0\}\subset (B_1,g)$ pointing in the $e_1$ direction.  From Theorem \ref{thm:fb-cap-first-var} we can write
\[
\delta_{\beta_i, g_i} V_i(X) = -\int g_i(H_{V_i, g_i}, X) d\mu_{V_i} + \int g_i(\nu_{g_i}, X) d\sigma_{V_i}
\]
(where $H_{V_i, g_i} = H^{tan}_{V_i, g_i} + H_{\R^n, g_i}$), while from Theorem \ref{thm:fb-first-var} we can write
\[
\delta V(X) = \int e_1 \cdot X d\mu_V .
\]
Now for $\phi \in C^1_c(B_1)$ fixed, define the vector fields $X_i = \phi \nu_{g_i}$, $X = \phi e_1$, and note by our hypotheses that $X_i \to X$ in $C^1$.  It follows by varifold convergence $V_i \to V$ that $\delta_{g_i} V_i(X_i) \to \delta V(X)$.  We then compute
\begin{align}
\left| \int \phi d\sigma_{V_i} - \int \phi d\sigma_V \right|
&= \left| \delta_{g_i} V_i (X_i) - \delta V(X) + \int g_i(H_{V_i, g_i}, X) d\mu_{V_i} \right| \\
&\leq \big|\delta_{g_i} V_i(X_i) - \delta V(X)\big| + c(\phi, n) ||H_{V_i, g_i}||_{L^\infty(B_1)} \\
&\to 0 \qedhere
\end{align}
\end{proof}

\subsection{Monotonicity}\label{sec:mono}

Here we establish a weighted interior monotonicity for (positive) varifolds, and a monotonicity for signed varifolds centered at points in the boundary.  In both cases we require the metric to be Euclidean at the center point, which ultimately is no real restriction because we can always arrange this with a linear transformation of coordinates.

We define the Euclidean density ratio of a signed $n$-varifold $V$ at scale $B_r(x)$ to be
\[
\Theta_V(x, r) := \frac{\mu_V(B_r(x))}{\omega_n r^n},
\]
define the pointwise densities
\[
\Theta_{V, g}(x) := \lim_{r \to 0} \frac{\mu_V(B_r^g(x))}{\omega_n r^n}, \quad \Theta_V(x) := \lim_{r \to 0} \Theta_V(x, r)
\]
wherever they exist.

Note that for a signed $n$-varifold $\Theta_V$ may not be positive, and we do not in general have mass monotonicity $\mu_V(U) \leq \mu_V(U')$ for $U \subset U'$.  However we will only be dealing with signed varifolds with a ``nice'' negative part, and in this case one can treat $\Theta_V$ as ``almost positive,'' which we record in the following Remark.
\begin{remark}[Basic signed density inequalities]\label{rem:theta-change-ball}
Suppose $V$ is a signed $n$-rectifiable varifold in $(B_1, g)$ supported in $\{ x_1 \leq 0 \}$ satisfying
\begin{gather}
\theta_{V, g}(x) \geq 0 \text{ $|\mu_V|$-a.e. $x_1 < 0$}, \quad \theta_{V, g}(x) \geq -a \text{ $|\mu_V|$-a.e. $x_1 = 0$}, \\
|g - g_{eucl}|_{C^0(B_1)}\le\delta, 
\end{gather}
for some $a \geq 0$.  If $B_r(\xi) \subset B_R(x) \subset B_1$, then
\begin{align}
\Theta_V(\xi, r) &\leq (R/r)^n \Theta_V(x, R) + a \frac{\haus^n_g(\ver \cap B_R(x) \setminus B_r(\xi))}{\omega_n r^n} \\
&\leq (R/r)^n \Theta_V(x, R) + (1+c(n)\delta) a ( (R/r)^n - 1).
\end{align}
Furthermore, if $A \leq A'$, then
\begin{align}
A\Theta_V(\xi, r)
&\leq A'  \frac{\mu_V(B_r(\xi)) + a \haus^n_g(B_r(\xi) \cap \ver)}{\omega_n r^n} - A a \frac{\haus^n_g(B_r(\xi) \cap \ver)}{\omega_n r^n} \\
&\leq A' \Theta_V(\xi, r) + (1+c(n)\delta) a (A' - A) .
\end{align}
\end{remark}

We first prove a weighted monotonicity for positive varifolds away from the boundary.  This is essentially standard but is not usually written down for a general metric.
\begin{lemma}[Weighted interior monotonicity]\label{lem:mono}
There is a constant $c(n)$ so that the following holds.  Let $g$ be a $C^1$ metric on $B_1$, $h$ a non-negative $C^1$ function in $B_1$, and $V$ be an $n$-rectifiable varifold in $(B_1 \cap \{ x_1 < 0 \}, g)$ with $\Theta_V(0, 1) < \infty$, locally finite $g$-first variation and $\sigma_{V,g}=0$.  Take $\xi \in B_1 \cap \{ x_1 < 0 \}$, and suppose
\[
g(\xi) = g_{eucl}, \quad \max\{||H_{V, g}||_{L^\infty(B_1)}, |Dg|_{C^0(B_1)} \} \leq \Lambda \leq 1.
\]
Then for every $0 < \tau < \sigma < 1-|\xi|$ (if $h = 0$ on $\{ x_1  = 0\}$), or $0 < \tau < \sigma < \min\{ 1-|\xi|, |\xi_1| \}$ (otherwise), we have the monotonicity
\begin{align}
&\frac{(1+c\Lambda \tau)^{n+1}}{\tau^n} \int_{B_\tau(\xi)} h d\mu_V \leq \frac{(1+c\Lambda \sigma)^{n+1}}{\sigma^n} \int_{B_\sigma(\xi)} h d\mu_V \\
&\quad\quad - \int_{B_\sigma(\xi)\setminus B_\tau(\xi)} h \frac{|(x - \xi)^\perp|^2}{|x - \xi|^{n+2}} d\mu_V + \int_\tau^\sigma \frac{(1+c\Lambda \rho)^{n+1}}{ \rho^{n}} \int_{B_\rho(\xi)} |\nabla_V h| d\mu_V d\rho.\label{eq:weight_monotonicity}
 \end{align}
\end{lemma}

\begin{proof}
Let $\phi$ be any smooth non increasing function which $\equiv 1$ on $(-\infty, 1/2]$ and $\equiv 0$ on $[1, \infty)$.  We wish to plug the vector field $X = h(x) \phi(|x - \xi|/\rho) (x - \xi)$ into the first variation. We observe that if $h=0$ on $\{x_1=0\}$ then $X=0$ on $\{x_1=0\}$ also, while if $h \neq 0$ on $\{ x_1 = 0 \}$ then $X=0$ on $\{x_1=0\}$ by our restriction on $\sigma$.  Either way the assumption on $V$ (and Remark \ref{rem:testing_with_vanish_vf}) implies $\delta_g V(X) = - \int g(H_{V,g}, X) d\mu_V$.
\begin{align}
\int \dive_{V, g_{eucl}}(X) d\mu_V
&\leq \int \dive_{V, g}(X) d\mu_V + c |Dg| |X| + c |g - g_{eucl}| |D_VX| d\mu_V \\
&\leq c\Lambda \int \phi h |x-\xi| d\mu_V \\
&\quad + c \int |Dg| |x - \xi| \left( |\phi'| \frac{|x - \xi|}{\rho} h + \phi |\nabla_V h| |x - \xi| + \phi h \right) d\mu_V \\
&\leq c\Lambda \rho \int \left(|\phi'| \frac{|x - \xi|}{\rho} h + \phi |\nabla_V h| \rho + \phi \right) d\mu_V
\end{align}
for $c = c(n, |g|_{C^0}) = c(n)$.

Now
\[
\dive_{V, g_{eucl}}(X) = \phi' \frac{|x - \xi|}{\rho} h  - \phi' \frac{|(x - \xi)^\perp|^2}{\rho |x - \xi|} h + \phi h n + \phi \nabla_V h \cdot x ,
\]
and so we deduce the inequality
\begin{align}
\int \left(\phi' \frac{|x - \xi|}{\rho} h + n \phi h\right) \dif\mu_V &\leq \int \phi' \frac{|(x - \xi)^\perp|^2}{\rho |x - \xi|} h \dif\mu_V + \rho \int \phi h \dif ||\delta_g V|| \\
&\quad + c \Lambda \rho \int \left(|\phi'| \frac{|x - \xi|}{\rho} h + \phi h\right) \dif\mu_V \\
&\quad + (1+ c\Lambda \rho) \rho \int \phi |\nabla_V h| \dif\mu_V
\end{align}
for $c = c(n)$.  

Therefore if we define
\[
I(\rho) = \int h \phi\left(\frac{|x - \xi|}{\rho}\right) \dif\mu_V, \quad J(\rho) = \int h \phi\left(\frac{|x - \xi|}{\rho}\right) \frac{|(x - \xi)^\perp|^2}{|x - \xi|^2} \dif\mu_V,
\]
and recalling that $\phi' \leq 0$, we have the differential inequality
\[
-(1+c\Lambda \rho) \rho I' + (n-c\Lambda \rho) I \leq -\rho J' + (1+c\Lambda\rho)\rho \int \phi |\nabla_V h| \dif\mu_V
\]
or equivalently
\[
I'(\rho) - \frac{n-c\Lambda \rho}{1+c\Lambda \rho} \frac{1}{\rho} I(\rho) \geq \frac{1}{1+c\Lambda \rho} J' - \int \phi |\nabla_V h| \dif\mu_V.
\]
Here again $c = c(n)$.  One can easily check the integrating factor of the left-hand side is
\[
\frac{(1+c \Lambda \rho)^{n+1}}{\rho^n},
\]
therefore (recalling that $J' \geq 0$) we get
\[
\frac{\mathrm{d}}{\mathrm{d}\rho} \big( (1+c\Lambda \rho)^{n+1} \rho^{-n} I(\rho)\big) \geq \rho^{-n} J'(\rho) - (1+c\Lambda\rho)^{n+1} \rho^{-n} \int \phi |\nabla_V h| \dif\mu_V.
\]
Integrating from $0 < \tau < \sigma < 1-|\xi|$ and taking $\phi \to 1_{(-\infty, 1)}$ gives the required inequality.
\end{proof}

\begin{remark}\label{rem:testing_with_vanish_vf}
In the above result, and in Lemma \ref{lem:w12}, we use the following fact: if $V$ is an $n$-rectifiable varifold in $(B_1 \cap \{ x_1 < 0 \}, g)$ with locally-finite $g$-first-variation, $\sigma_{V, g} = 0$, and $\mu_V(B_1) + ||H_{V, g}||_{L^1(B_1)} < \infty$, then
\begin{equation}\label{eqn:test-vanish}
\int \mdiv_{V, g}(Y) d\mu_V = - \int g(H_{V, g}, Y) d\mu_V
\end{equation}
for all $Y \in C^1_c(B_1, \R^{n+1})$ which vanish on $\{ x_1 = 0 \}$ (as opposed to only $Y\in C^1_c(B_1 \cap \{ x_1 < 0 \}, \R^{n+1})$.  To see this, replace $Y$ with $\gamma(|x_1|/\eps) Y$ in the first variation identity \eqref{eq:definition_of_H_and_sigma}, for $\gamma$ a smooth cutoff which $\equiv 0$ on $(-\infty, 1]$ and $\equiv 1$ on $[2, \infty)$ and $|\gamma'| \leq 2$.  Then noting that
\[
|\gamma'(x_1/\eps)|/\eps |Y(x)| \leq (2/\eps) |Y(x) - Y(\pi_{\R^n}(x))| \leq 2|DY|_{C^0(B_1)},
\]
we can take $\eps \to 0$ and obtain \eqref{eqn:test-vanish} by the dominated convergence theorem.
\end{remark}

We next prove a signed monotonicity at points centered in the boundary.  The below Lemma (and Corollary \ref{cor:usc}) are essentially intrinsic versions of the monotonicity due to \cite{KaTo}.
\begin{lemma}[Signed boundary monotonicity]\label{lem:bd-mono}
There is a constant $c(n)$ so that the following holds.  Let $g$ be a $C^1$ metric on $B_1$, $\beta$ a $C^1$ function on $B_1$.  Let $V$ either be a $(\beta, S)$-capillary varifold in $(B_1, g)$, or let $V$ to be a signed $n$-rectifiable varifold supported in $\{ x_1 \leq 0\}$ with $g$-free-boundary in $\{ x_1 = 0 \}$ satisfying
\[
\theta_{V, g}(x) \geq 0 \text{ for $|\mu_V|$-a.e. $x_1 < 0$}, \quad \theta_{V, g}(x) \geq -1 \text{ for $|\mu_V|$-a.e. $x_1 = 0$}.
\]
Either way, take $\xi \in B_1 \cap \{ x_1 =  0\}$ and suppose that
\begin{gather}
g(\xi) = g_{eucl}, \quad \max \{ ||H_{V, g}^{tan}||_{L^\infty(B_1)}, |Dg|_{C^0(B_1)}, |D\beta|_{C^0(B_1)} \} \leq \Lambda \leq 1.
\end{gather}
Then for every $0 < \tau < \sigma < 1-|\xi|$ we have the monotonicity
\begin{align}
\frac{(1+c\Lambda \tau)^{n+1}}{\tau^n} \mu_V(B_\tau(\xi)) &\leq \frac{(1+c\Lambda \sigma)^{n+1}}{\sigma^n} \mu_V(B_\sigma(\xi)) \\
&\quad\quad - \int_{B_\sigma(\xi)\setminus B_\tau(\xi)} \frac{|(x - \xi)^\perp|^2}{|x - \xi|^{n+2}} d\mu_{V_I}(x) + c\Lambda (\sigma - \tau).
\end{align}
(recall here $\mu_V$ is the \emph{signed} mass measure)
\end{lemma} 

\begin{proof}
The proof is similar to Lemma \ref{lem:mono}.  A key observation is that, since $\theta_{V, g} \geq 0$ $|\mu_V|$-a.e. in $\{x_1 < 0\}$ and $|\theta_{V, g}| \leq \theta_{V, g} + 2$ $|\mu_V|$-a.e. in $\{x_1 = 0\}$, we have the inequality
\begin{equation}\label{eqn:bd-mono}
|\mu_V|(B_\rho(\xi)) \leq \mu_V(B_\rho(\xi)) + c(n) \rho^n
\end{equation}
for every $0 < \rho < 1-|\xi|$.

Take $\phi$ a smooth decreasing function which $\equiv 1$ on $(-\infty, 1/2]$ and $\equiv 0$ on $[1, \infty)$, and note the vector field $X = \phi(|x - \xi|/\rho) (x - \xi)$ lies tangential to $\{ x_1 = 0 \}$.  Therefore, as in Lemma \ref{lem:mono}, we can use \eqref{eqn:bd-mono} and \eqref{eqn:beta-g-vs-g-2} to compute
\begin{align}
\int \dive_{V, g_{eucl}}(X) \dif\mu_V
&\leq c \int (|H_{V, g}^{tan}| + |Dg|) \phi |x-\xi| \dif|\mu_V| \\
&\quad + c \int |Dg| |x - \xi| \left( |\phi'| \frac{|x - \xi|}{\rho} + \phi  \right) \dif|\mu_V|\\
&\quad + \int_{\R^n} |D\beta| \phi |x - \xi| d\haus^n \\
&\leq c\Lambda \rho \int \left(|\phi'| \frac{|x - \xi|}{\rho} + \phi \right) \dif\mu_V + c\Lambda \rho^{1+n}
\end{align}
for $c = c(n, |g|_{C^0}) = c(n)$ (since $\Lambda \leq 1$).

By the same computations as in Lemma \ref{lem:mono} we define
\[
I(\rho) = \int \phi(|x - \xi|/\rho) \dif\mu_V, \quad J(\rho) = \int \phi(|x - \xi|/\rho) \frac{|(x - \xi)^\perp|^2}{|x - \xi|^2} \dif\mu_V,
\]
then we get
\[
-(1+c\Lambda \rho) \rho I' + (n-c\Lambda \rho) I \leq -\rho J' + c \Lambda \rho^{1+n} .
\]
Note that since $\pi_V^\perp(x - \xi) = 0$ $|\mu_V|$-a.e. in $\{x_1 = 0\}$, we have $J' \geq 0$.  Therefore, using the same integrating factor as before and recalling that $\Lambda \leq 1$, we get
\[
\frac{\mathrm{d}}{\mathrm{d}\rho} ( (1+c\Lambda \rho)^{n+1} \rho^{-n} I(\rho)) \geq \rho^{-n} J'(\rho) - c \Lambda
\]
for $c = c(n)$.  Now integrate from $\tau$ to $\sigma$ and take $\phi \to 1_{(-\infty, 1)}$.
\end{proof}

\begin{remark}[Existence of $\Theta_{V, g}(x)$]\label{rem:thetag-vs-theta}
If the assumptions of one of the two above Lemmas hold, then we have that $\Theta_{V, g}(\xi)$ exists and coincides with $\lim_{r \to 0} \Theta_V(\xi, r)$.  To see this, note that since $g(\xi) = g_{eucl}$ we have
\[
B_{(1-c(n)r)r}(\xi) \subset B^g_r(\xi) \subset B_{(1+c(n)r)r}(\xi)
\]
as $r \to 0$, and hence by Remark \ref{rem:theta-change-ball}
\[
\Theta_V(\xi, (1-cr)r) - c r \leq \frac{\mu_V(B^g_r(\xi))}{\omega_n r^n} \leq \Theta_V(\xi, (1+cr)r) + c r.
\]
\end{remark}

\begin{remark}[Existence of capillary tangent cones]\label{rem:tangent-cones}
If $V$ is a $(\beta, S)$-capillary varifold in $(B_1, g)$ and $g(0) = g_{eucl}$, then it follows from the compactness Lemma \ref{lem:soft-conv} and monotonicity of Lemma \ref{lem:bd-mono} that for any sequence $r_i \to 0$, there is a subsequence $r_{i'}$ so that the rescaled varifolds $(\eta_{0, r_{i'}})_\sharp V$ converge to a stationary $n$-rectifiable signed varifold cone $V'$ with free-boundary in $\{ x_1 = 0 \}$, satisfying $\Theta_{V'}(0, 1) = \Theta_{V, g}(0)$.
\end{remark}

We can restate the monotonicities of Lemma \ref{lem:mono}, \ref{lem:bd-mono} in a general metric (not necessarily Euclidean at $\xi$) as follows.
\begin{prop}[Weak monotonicity at general points]\label{prop:mono-weak}
There are constants $c(n) \ll C(n)$ so that the following holds. Let $g$ be a $C^1$ metric on $B_1$, and $\beta$ a $C^1$ function. Assume one of the two following cases:
\begin{description}
    \item[(interior case)] $V$ is as in Lemma \ref{lem:mono}, $\xi\in B_1\cap\{x_1<0\}$ and
    \begin{gather}
    \max \{ ||H_{V, g}||_{L^\infty(B_1)}, |g - g_{eucl}|_{C^1(B_1)} \} \leq \delta \leq 1 ;
    \end{gather}
    \item[(boundary case)] $V$ is as in Lemma \ref{lem:bd-mono}, $\xi\in B_1\cap\{x_1=0\}$ and
    \begin{gather}
    \max \{ ||H^{tan}_{V, g}||_{L^\infty(B_1)}, |g - g_{eucl}|_{C^1(B_1)}, |D\beta|_{C^0(B_1)} \} \leq \delta \leq 1 .
    \end{gather}
\end{description}
If we write $\eps = |g(\xi) - g_{eucl}|$, then we have
\begin{equation}\label{eqn:mono-weak-concl}
\max\{\Theta_{V,g}(\xi), \Theta_V(\xi, r)\} \leq (1+c \delta s + c \eps) \Theta_V(\xi, s) + C (\delta s + \eps) 
\end{equation}
for every $0 < (1+c\eps)r < s < (1-c\eps)\min \{ 1 - |\xi|, |\xi_1| \}$ (in the interior case) or $0 < (1+c\eps)r < s < (1-c\eps)(1-|\xi|)$ (in the boundary case).
\end{prop}

\begin{proof}
By a standard Gram-Schmidt process (combined with a possible rotation) we can choose an affine map $L$ satisfying
\begin{equation}\label{eqn:mono-weak-0}
L\xi = \xi, \quad |DL - \operatorname{Id}| \leq c(n) \eps, \quad (L^* g)(\xi) = g_{eucl}, \quad L\ver = se_1 + \ver
\end{equation}
where $s = 0$ if $\xi_1 = 0$ and $|s| \leq c(n)\eps$ if $\xi_1 < 0$.  Define $V' = (L^{-1})_\sharp V$, $g' = L^* g$, $\beta' = \beta \circ L$.

If we let $R = (1-c(n)\eps)(1-|\xi|)$ for suitable $c(n)$, then in the interior case $V'$ is a $n$-rectifiable varifold in $(B_{R}(\xi) \cap \{ x_1 < s \}, g')$ satisfying
\[
\max \{ ||H_{V', g'}||_{L^\infty(B_R(\xi))}, |g' - g_{eucl}|_{C^1(B_R(\xi))} \} \leq c(n)\delta ;
\]
while, in the boundary case, $V'$ is either a $(\beta', L^{-1}(B))$-capillary varifold in $(B_R, g')$, or $V'$ is a signed $n$-rectifiable varifold supported in $\{ x_1 \leq 0\}$ with $g$-free-boundary in $\{ x_1 = 0\}$, which either way satisfies
\begin{gather}
\theta_{V', g'}(x) \geq 1 \text{ $|\mu_{V'}|$-a.e. $x_1 < 0$}, \quad \theta_{V', g'}(x) \geq -1 \text{ $|\mu_{V'}|$-a.e. $x_1 = 0$}, \\
\max \{ ||H_{V', g'}^{tan}||_{L^\infty(B_R(\xi))}, |g' - g_{eucl}|_{C^1(B_R(\xi))} , |D\beta'|_{C^0(B_R(\xi))} \} \leq c(n)\delta .
\end{gather}
(Recall $\eps \leq \delta$).

Now since $\mu_{V'}(B_r(\xi)) = \mu_V(L B_r(\xi))$, we can apply Lemma \ref{lem:mono} or \ref{lem:bd-mono} to $V'$ to obtain
\begin{equation}\label{eqn:mono-weak-1}
r \mapsto (1+c\delta r)^{n+1} r^{-n} \mu_V(LB_r(\xi)) + c \delta r \quad \text{ is increasing }  
\end{equation}
for $0 < r < (1-c\eps) \min\{ 1-|\xi|, |\xi_1| \}$ (in the interior case) or $0 < r < (1-c\eps) (1-|\xi|)$ (in the boundary case), and $c = c(n)$.

By our assumptions and \eqref{eqn:mono-weak-0} we can find $r < (1+c_1\eps)r < (1-c_1\eps) s < s$, for some $c_1(n)$, so that
\[
B_r(\xi) \subset L B_{(1+c_1\eps)r}(\xi) \subset B_{(1+c_1^2 \eps)r}(\xi) \subset B_{(1-c_1^2\eps)s}(\xi) \subset LB_{(1-c_1\eps)s}(\xi) \subset B_s(\xi).
\]
We can then apply Remark \ref{rem:theta-change-ball} and monotonicity \eqref{eqn:mono-weak-1} to get the required inequality \eqref{eqn:mono-weak-concl}.
\end{proof}

\begin{cor}[Properties of $\Theta_{V, g}(x)$]\label{cor:usc}
Under the same hypotheses as Proposition \ref{prop:mono-weak}, the pointwise density function $\Theta_{V, g}(x)$ exists for all $x$, coincides with $\theta_{V,g}(x)$ at $|\mu_V|$-a.e. $x$, and is upper semi-continuous on $\{ x_1 < 0 \}$ and (separately) on $\{ x_1 = 0 \}$.
\end{cor}

\begin{remark}
One obvious consequence of the above Corollary is:  if $\theta_{V, g}(x) \geq \theta_0$ at $|\mu_V|$-a.e. $x \in U$, for $U = \{ x_1 < 0 \}$ or  $U = \{ x_1 = 0 \}$, then $\Theta_{V, g}(x) \geq \theta_0$ for all $x \in U$.
\end{remark}

\begin{proof}[Proof of Corollary \ref{cor:usc}]
Fix a $\xi \in B_1$.  Since the existence and upper semicontinuity of $\Theta_{V, g}$ is local and coordinate-invariant, there is no loss in choosing an affine linear change of coordinates like in the proof of Proposition \ref{prop:mono-weak}, and thereby assume $g(\xi) = g_{eucl}$.  The existence of $\Theta_{V, g}(\xi)$ then follow from Remark \ref{rem:thetag-vs-theta}.

We prove $\Theta_{V, g}$ is upper-semi-continuous on $\{x_1 = 0 \}$; the case when $\{ x_1 < 0 \}$ is essentially verbatim and is in fact easier, since then the varifold is unsigned.  So $\xi_1 = 0$.  Take $\xi_i \in \ver$ converging to $\xi$, write $\eps_i = |g(\xi_i) - g_{eucl}|$, and for any $r$ write $r_i = r - |\xi_i - \xi|$.  Of course $\eps_i \to 0$ and $r_i \to r$ by our assumptions.

Then for $i \gg 1$, we can use monotonicity \eqref{eqn:mono-weak-concl} and Remark \ref{rem:theta-change-ball} to compute
\begin{align}
\Theta_{V, g}(\xi_i) &\leq (1+ c \delta r_i + c \eps_i)\Theta_V(\xi_i, r_i) + C a (\delta r_i + \eps_i) \\
&\leq (1+c\delta r_i + c \eps_i) ( (r^n/r_i^n) \Theta_V(\xi, r) + c(n) ( r^n/r_i^n - 1)) + C a (\delta r_i + \eps_i).
\end{align}
for $c(n), C(n)$ as in Proposition \ref{prop:mono-weak}.  We deduce that
\[
\limsup_i \Theta_{V, g}(\xi_i) \leq (1+c \delta r) \Theta_V(\xi, r) + C a \delta r.
\]
Now by Remark \ref{rem:thetag-vs-theta} the limit of the right-hand side as $r \to 0$ is $\Theta_{V, g}(\xi)$.

To see that $\Theta_{V, g} = \theta_{V, g}$ at $|\mu_V|$-a.e., let us write by the area formula (and using the notation of the Preliminaries section), 
\[
\mu_V(A) = \int_{A \cap M_V} \theta_{V, g}(x) F_g(x) d\haus^n(x),
\]
where $F_g(x) = (\det g(e_i, e_j))^{1/2}$, for $e_i$ a $g_{eucl}$-orthonormal basis of $T_x M_V$.  Then at $|\mu_V|$-a.e. $x$ (being a Lebesgue point for $\theta_{V, g}$ w.r.t. $\haus^n \llcorner M_V$ and also where $M$ has an approximate tangent plane) we have
\begin{align}
\Theta_{V, g}(x) &= \theta_{V, g}(x) F_g(x) \omega_n^{-1} \haus^n(\{ y \in T_x M_V : g|_x(y, y) < 1 \})  \\
&= \theta_{V, g}(x)
\end{align}
where the second equality follows again by the area formula.
\end{proof}

\subsection{Compactness and Allard regularity}

We will make use of Allard's varifold compactness and regularity theorems.  These are typically stated (and proven) in Euclidean metrics, but they can be adapted to general metrics with fairly little trouble.  We omit the proofs here.
\begin{theorem}[Allard's compactness]\label{thm:allard-compact}
Let $a > 0$, $g_i$ be a sequence of $C^1$ metrics on $B_1$ converging to $g_{eucl}$ in $C^1$, and $V_i$ a sequence of $n$-rectifiable varifolds in $(B_1, g_i)$ satisfying
\begin{gather}
\sup_i \left( \mu_{V_i}(B_r) + ||\delta_{g_i} V_i||(B_r) \right) < \infty  \quad \forall r < 1, \\
\theta_{V_i, g_i}(x) \geq a > 0 \text{ $\mu_{V_i}$-a.e. $x \in B_1$}.
\end{gather}
Then after passing to a subsequence we can find an $n$-rectifiable varifold $V$ in $(B_1, g_{eucl})$ so that $V_i \to V$ as varifolds, $\delta_{g_i} V_i \to \delta V$ as Radon measures, and $\theta_{V}(x) \geq a$ at $\mu_V$-a.e. $x \in B_1$.  Moreover, if the $V_i$ are integral, then $V$ is integral also.
\end{theorem}

\begin{proof}
Straightforward modification of Allard's compactness theorem, see e.g. \cite[Chapter 8]{Sim}.
\end{proof}

\begin{theorem}[Allard's regularity]\label{thm:allard}
There are constants $\delta(n), c(n)$ so that if $V$ is a rectifiable $n$-varifold in $(B_1, g)$ with locally finite $g$-first variation such that $\nu_{V,g}=0$ and which satisfies:
\begin{gather}
\max\{ \osc_{\R^n}(V, B_1), ||H_{V, g}||_{L^\infty(B_1)}, |g - g_{eucl}|_{C^1(B_1)} \} \leq \delta, \label{eqn:allard-hyp1} \\
\theta_{V, g}(x) \in \N \text{ $\mu_V$-a.e. $x \in B_1$}, \quad \Theta_V(0, 1) \leq 3/2, \label{eqn:allard-hyp2}
\end{gather}
then
\begin{gather}
\spt V \cap B_{1/2} = \graph_{\R^n}(u), \label{eqn:allard-concl1} \\
|u|_{C^{1,1/2}} \leq c \max\{ \osc_{\R^n}(V, B_1),  ||H_{V, g}||_{L^\infty(B_1)},  |Dg|_{C^0(B_1)} \}. \label{eqn:allard-concl2}
\end{gather}
\end{theorem}

\begin{proof}
The proof is a straightforward adaption of the proof in Euclidean metrics (as in e.g. \cite[Chapter 5]{Sim}). The only slightly subtle aspect here is that the estimate \eqref{eqn:allard-concl2} depends on $|Dg|_{C^0}$ rather than $|g - g_{eucl}|_{C^1}$.  This is essentially a direct consequence of the fact that if $u : \R^n\cap B_1 \to \R$ is $C^{1, \alpha}$ with $|u|_{C^{1,\alpha}} \leq \Lambda$, and $L$ is a linear function satisfying $|L - \operatorname{Id}| \leq \eps$ for $\eps(n)$ small, then $L(\graph_{\R^n}(u)) \cap B_{1-c(n)\eps} = \graph_{L \R^n}(\tilde u) \cap B_{1-c(n)\eps}$ for $\tilde u$ a $C^{1,\alpha}$ function satisfying $|u|_{C^{1,\alpha}} \leq c(n, \alpha)\Lambda$.

We elaborate.  First note that \eqref{eqn:allard-concl2} certainly holds true if $g(0) = g_{eucl}$, since in this case $|g - g_{eucl}|_{C^1(B_1)} \leq 2 |Dg|_{C^0(B_1)}$.  Now for a general metric $g$ with $|g(0) - g_{eucl}| \leq \delta$, choose a linear map $L$ satisfying
\[
|L - \operatorname{Id}| \leq c(n)\delta, \quad (L^* g)(0) = g_{eucl}, \quad L\R^n = \R^n.
\]
Let $\tilde g = L^* g$, $\tilde V = (L^{-1})_\sharp V$, so that as matrices $\tilde g|_x = L^T g|_{Lx} L$ and as sets $\spt \tilde V = L^{-1} \spt V$.  From \eqref{eqn:osc-A}, ensuring $\delta(n)$ is small, $\tilde V \subset (B_{1/2}, \tilde g)$ satisfies
\begin{align}
&\max\{ \osc_{\R^n}(\tilde V, B_{1/2}) , ||H_{\tilde V, \tilde g}||_{L^\infty(B_{1/2})}, |\tilde g - g_{eucl}|_{C^1(B_{1/2})} \}  \\
&\leq c(n) \max \{ \osc_{\R^n}(V, B_1), ||H_{V, g}||_{L^\infty(B_1)}, |Dg|_{C^0(B_1)} \}, \label{eqn:allard-1}
\end{align}
(having used that $|\tilde g - g_{eucl}|_{C^0(B_{1/2})} \leq |D\tilde g|_{C^0(B_{1/2})}$), and
\[
\theta_{\tilde V, \tilde g}(x) \in \N \text{ $\mu_{\tilde V}$-a.e. $x \in B_{1/2}$}, \quad \Theta_{\tilde V}(0, 1) \leq 3/2.
\]
The bound on $\Theta_{\tilde V}(0, 1/2)$ can be verified as follows: as $\delta \to 0$, by compactness of integral varifolds we have $\Theta_{V}(0, 3/4) \to 1$, so ensuring $\delta$ is sufficiently small we can compute
\[
\Theta_{\tilde V}(0, 1/2) \leq (1+c(n)\delta) \Theta_V(0, 1+c(n)\delta) \leq (1+c\delta)\Theta_V(0, 3/4) \leq 3/2.
\]

We can then apply Allard and \eqref{eqn:allard-1} to obtain a $C^{1,1/2}$ function $\tilde u : B_{1/4}^n \to \R$ so that
\begin{gather}
\spt \tilde V \cap B_{1/4} = \graph_{\R^n}(\tilde u), \\
|\tilde u|_{C^{1,1/2}} \leq c(n)\max \{ \osc_{\R^n}(V, B_1), ||H_{V, g}||_{L^\infty(B_1)}, |Dg|_{C^0(B_1)} \} .
\end{gather}
Now we claim that
\begin{equation}\label{eqn:allard-2}
\spt V \cap B_{1/8} = (L \graph_{\R^n}(\tilde u)) \cap B_{1/8} = \graph_{\R^n}(u) \cap B_{1/8}
\end{equation}
for $u : B_{1/8}^n \to \R$ a $C^{1,1/2}$ function satisfying $|\tilde u|_{C^{1,1/2}} \leq c(n) |u|_{C^{1,1/2}}$.

Of course the first equality in \eqref{eqn:allard-2} is trivial.  Towards the second equality, define $F(x) = \pi_{\R^n}(Lx + u(x) Le_{n+1})$.  Then $F$ is a $C^{1,1/2}$ function $B_{1/4}^n \to \R^n$ satisfying
\[
|F - id| + |DF - Id| + [DF]_{1/2} \leq c(n) \delta,
\]
and so (ensuring $\delta(n)$ is small) there is a $C^{1,1/2}$ inverse $F^{-1} : F(B_{1/4}^n) \supset B_{1/8}^n \to \R^n$ satisfying the same estimates.  Since $L$ is linear and $L\R^n = \R^n$ we can then write
\[
L(x + u(x) e_{n+1}) = F(x) + \pi^\perp_{\R^n}(L e_{n+1}) u(x) =: y + \tilde u(y) e_{n+1}
\]
for $y = F(x)$ and $\tilde u(y) = (e_{n+1} \cdot L e_{n+1}) u(F^{-1}(y))$.  One can easily verify that the function $\tilde u$ is defined on $B_{1/8}^n$, and satisfies the estimates
\[
|\tilde u|_{C^{1, 1/2}(B_{1/8}^n)} \leq c(n) |u|_{C^{1, 1/2}(B_{1/4}^n)},
\]
which proves our claim.  By applying our claim in various smaller balls we get the conclusion of the Theorem.
\end{proof}

\section{Capillary first variation control}\label{sec:first_var_control}

Our main result of this section is the following Theorem \ref{thm:cap-first-var}, bounding individually the first variation of the interior and boundary pieces $V_I, V_B$ of $V$.  The key to proving first variation control is establishing $(n-1)$-dimensional Ahlfors-regularity of the boundary measure $\sigma_{V, g}$ (Theorem \ref{thm:sigma-dichotomy}).  In turn, the key to proving (lower-)Ahlfors-regularity of $\sigma_{V, g}$ is a boundary rigidity result (Theorem \ref{thm:rigidity}), which says that whenever the tilt-excess of $V$ with respect to $\ver$ is small, then $\spt V \subset \ver$.  This last result can be viewed as a kind of baby Allard theorem in and of itself.

Using Theorem \ref{thm:cap-first-var}, we later prove in Subsection \ref{sec:conv} an improved compactness result (Lemma \ref{lem:conv}) which we will use later in the proof of the regularity results.

\begin{theorem}[First variation control]\label{thm:cap-first-var}
Given $\alpha > 0$, there are $\delta(n, \alpha)$, $\gamma(n, \alpha)$, $c(n, \alpha)$ positive so that the following holds.  Let $g$ be a $C^1$ metric on $B_1$, $\beta$ a $C^1$ function on $B_1$, $S \subset \R^n$, and $V$ a $(\beta, S)$-capillary varifold in $(B_1, g)$. Suppose that
\begin{gather}\label{eqn:cap-first-var-hyp}
\Theta_V(0, 1) + (\cos\beta(0))_- \leq 1- \alpha, \\
\max \{ ||H^{tan}_{V, g}||_{L^\infty(B_1)}, |g - g_{eucl}|_{C^1(B_1)}, |D\beta|_{C^0(B_1)} \} \leq \delta.
\end{gather}
Then:
\begin{enumerate}
\item \label{lab:cap-1} $V_I := V \llcorner \{ x_1 < 0 \}$ and $V_B := V \llcorner \{ x_1 = 0\}$ both have locally-finite $g$-first variation in $B_\gamma$, and we have the inequalities
\begin{gather}
|\delta_g V_I| \leq c \delta \mu_{V_I} + c \sigma_{V, g} \leq c^2, \label{eqn:cap-first-var-concl1} \\
|\delta_g V_B| \leq c \delta |\mu_{V_B}| + c |H_{\beta, g}| |\mu_{V_B}| + c \sigma_{V, g} \leq c \delta \haus^n_g \llcorner S + c \sigma_{V, g} \leq c^2, \label{eqn:cap-first-var-concl1.2} 
\end{gather}
which imply
$x \mapsto \cos\beta(x) 1_{S}(x)$ is a function of locally-bounded variation (w.r.t. either $g$ or $g_{eucl}$) in $B_\gamma \cap \R^n$, with total variation $\leq c^2$,
and
\begin{equation}\label{eqn:cap-first-var-concl1.5}
\int \phi \dif |\delta_g V_I|  \leq c \int \big(|\nabla_V \phi| + \phi\big) \dif\mu_{V_I} \quad \forall \phi \in C^1_c(B_\gamma) \text{ non-negative};
\end{equation}

\item \label{lab:cap-2} we have the Ahlfors-regularity
\begin{equation}\label{eqn:cap-first-var-concl2}
c^{-1} r^{n-1} \leq \sigma_{V, g}(B_r(x)) \leq c r^{n-1} \quad \forall x \in \spt \sigma_{V, g} \cap B_\gamma,\,\, 0 < r < 1/4,
\end{equation}
and
\begin{equation}\label{eqn:cap-first-var-concl2.5}
c^{-1} \haus^{n-1}_g \llcorner \spt \sigma_{V, g} \leq \sigma_{V, g} \leq c \haus^{n-1}_g \llcorner \spt\sigma_{V,g} \text{ on } B_\gamma, 
\end{equation}
and for any set $W \subset B_\gamma$ we have the Minkowski estimate
\begin{equation}\label{eqn:cap-first-var-concl2.6}
r^{-2} |B_r(\spt \sigma_{V, g} \cap W)| \leq c \sigma_{V, g}(B_r(W) \cap B_{1/2}) \leq c^2 \quad \forall 0 < r <1/4;
\end{equation}

\item \label{lab:cap-3} we have the Ahlfors-regularity
\begin{equation}\label{eqn:cap-first-var-concl3}
c^{-1} r^n \leq \mu_{V_I}(B_r(x)) \leq c r^n \quad \forall x \in \spt V_I \cap B_\gamma,\,\, 0 < r < 1/4;
\end{equation}

\item \label{lab:cap-4} if $\sigma_{V_I, g}, \sigma_{V_B, g}$ are the generalized $g$-boundary measures associated to $V_I, V_B$, and $\eta_{V_I, g}, \eta_{V_B, g}$ the generalized $g$-conormals, and $\nu_g$ is the outward $g$-unit normal of $\ver \subset (\R^{n+1}, g)$, then in $B_\gamma$ we can write
\begin{equation}\label{eqn:cap-first-var-concl4}
\nu_g d\sigma_{V, g} = \eta_{V_I, g} d\sigma_{V_I, g} + \eta_{V_B, g} d\sigma_{V_B, g}, \quad \sigma_{V_I, g}, \sigma_{V_B, g} \leq c \sigma_{V, g}, 
\end{equation}
with $g( \eta_{V_I, g}, \nu_g) \geq 1/c$ for $\sigma_{V_I, g}$-a.e. and $g(\eta_{V_B, g}, \nu_g) = 0$ for $\sigma_{V_B, g}$-a.e. ;


\item \label{lab:cap-5} we have
\begin{align}\label{eqn:cap-first-var-concl6}
\spt \sigma_{V, g} \cap B_\gamma = \spt V_I \cap \{ x_1 = 0 \} \cap B_\gamma = \spt \sigma_{V_I, g} \cap B_\gamma 
\end{align}
and
\begin{equation}\label{eqn:cap-first-var-concl6.5}
\del(\spt V_B) \cap B_\gamma \subset \spt \sigma_{V, g} \cup \del  \{ \cos\beta \neq 0 \} ,
\end{equation}
where $\del(\spt V_B)$ and $\del\{\cos\beta\neq0\}$ are the topological boundaries of $\spt V_B$ and $\{\cos\beta\neq0\}$ inside $\ver$, respectively.
\end{enumerate}
\end{theorem} 

\begin{remark}\label{rem:S-perimeter}
If $|\cos\beta| \geq b > 0$ on some open set $U \subset B_\gamma \cap \R^n$, then Theorem \ref{thm:cap-first-var} implies $S$ is a set of finite perimeter in $U$, with respect to either the metric $g$ or $g_{eucl}$.  This is because
\begin{align}
\int_S \mdiv_{g_{eucl}}(X) d\haus^n 
&= \int_S \mdiv_g(X/\sqrt{g}) d\haus^n_g \\
&= \delta_g V_B\left(\frac{X}{\sqrt{g} \cos\beta}\right) - \int_S X \cdot D(\sin\beta)/\cos\beta d\haus^n \\
&\leq c(n, \alpha, b)|X|_{C^0}
\end{align}
for any $X \in C^1_c(U, \R^n)$, where $\sqrt{g} = \sqrt{\det(g_{ij})}$ is the volume form of $g|_{\R^n}$.  So we also have both the $g$- and $g_{eucl}$-perimeters of $S$ in $U$ are bounded by $c(n, \alpha, b)$.
\end{remark}


\begin{remark}\label{rem:sigma-lower}
The hypotheses of Theorem \ref{thm:cap-first-var} imply
\begin{equation}\label{eqn:sigma-lower-rem}
|\mu_V|(B_1) \leq \omega_n \Theta_V(0, 1) + ( (\cos\beta(0))_- + \delta)\haus^n_g(B_1 \cap \ver) \leq \omega_n(1-\alpha/2)
\end{equation}
taking $\delta(n, \alpha)$ small.  However we note that a condition like $\Theta_V(0, 1) + (\cos\beta(0))_- \leq 1- \alpha$ is in general stronger than an inequality like \eqref{eqn:sigma-lower-rem}.
\end{remark}

\begin{remark}
The constant $c$ necessarily $\to \infty$ as $\alpha \to 0$, and Theorem \ref{thm:cap-first-var} can fail if $\alpha = 0$.  See Remark \ref{rem:teaser1-sharp}.  
\end{remark}



\begin{proof}[Proof of Theorem \ref{thm:cap-first-var} given Sections \ref{sec:rigidity}, \ref{sec:sigma-ahlfors}]
Let us first note that by taking $\delta(\alpha)$ small, we can find an $a \in [0, 1]$ so that
\[
\inf_{B_1} \cos\beta(x) \geq -a, \quad \Theta_V(0, 1) + a \leq 1 - \alpha/2.
\]

We start with \hyperref[lab:cap-2]{Item 2}.  Ensure $\gamma$ is no larger than the $\gamma(n, \alpha/2)$ from Theorem \ref{thm:sigma-dichotomy}, then from this same theorem we have Ahlfors-regularity \eqref{eqn:cap-first-var-concl2}.  By standard covering techniques, \eqref{eqn:cap-first-var-concl2} implies
\[
\sigma_{V, g}(A) /c(n, \alpha) \leq \haus^{n-1}_g(\spt\sigma_{V,g} \cap A) \leq c(n, \alpha)\sigma_{V, g}(A)
\]
for all Borel $A \subset B_\gamma$, which implies \eqref{eqn:cap-first-var-concl2.5}.

Given $W \subset B_\gamma$, choose a maximal $r/2$-net $\{x_i\}_i$ in $W \cap \spt \sigma_{V, g}$, so that the balls $\{ B_r(x_i)\}_i$ cover $W \cap \spt \sigma_{V, g}$ and the $r/2$-balls are disjoint.  For $r \leq 1/4$, we can compute using \eqref{eqn:cap-first-var-concl2} and Lemma \ref{lem:sigma-lower}:
\begin{align}
r^{-2} |B_r(\spt \sigma_{V, g} \cap W)| 
&\leq r^{-2} \sum_i |B_{2r}(x_i)| \\
&\leq c \sum_i (r/2)^{n-1} \\
&\leq c \sum_i \sigma_{V, g}(B_{r/2}(x_i)) \\
&\leq c \sigma_{V, g}(B_r(W) \cap B_{1/2}) \\
&\leq c^2,
\end{align}
which gives \eqref{eqn:cap-first-var-concl2.6}.

To continue we shall also require the first equality of \eqref{eqn:cap-first-var-concl6}, i.e. that
\begin{equation}\label{eqn:cap-first-var-0.5}
\sigma_{V, g} \cap B_\gamma = \spt V_I \cap \R^n \cap B_\gamma.
\end{equation}
The inclusion $\spt \sigma_{V, g} \subset \spt V_I \cap \R^n$ follows directly from Remark \ref{rem:no-V_I}.  On the other hand, if $\xi \in B_\gamma \setminus \spt \sigma_{V, g}$, then by Theorem \ref{thm:sigma-dichotomy} we must have $V_I \llcorner B_r(\xi) = 0$ for some $r > 0$, and hence $\spt V_I \cap \ver \subset \spt \sigma_{V, g}$ in $B_\gamma$.

We can now prove \hyperref[lab:cap-1]{Item 1}.  By \eqref{eqn:cap-first-var-0.5} and Theorem \ref{thm:fb-cap-first-var} we have that
\[
\delta_g V_B(X) = -\int g(H_{\beta, g} + H_{\R^n, g} , X) d\mu_{V_B} \quad \forall X \in C^1_c(B_\gamma \setminus  \spt\sigma_{V, g}).
\]

Fix an open set $W \subset B_\gamma $, and $0 < \tau < 1/4$.  By e.g. mollifying the Euclidean distance function we can find a $\phi_\tau \in C^1_c(\R^{n+1}, \R)$ satisfying
\begin{gather}
\phi_\tau \equiv 1 \text{ on } B_{\tau/2}(\spt \sigma_{V, g} \cap W), \quad \phi_\tau \equiv 0 \text{ outside } B_\tau(\spt \sigma_{V, g} \cap W), \\
|D\phi_\tau| \leq c(n)/\tau.
\end{gather}
By the monotonicity of Lemma \ref{lem:bd-mono}, for every $\xi \in B_\gamma \cap \ver$ we have $\Theta_V(\xi, r) \leq c(n)$, and hence
\begin{equation}\label{eqn:cap-first-var-1}
|\mu_V|(B_r(\xi)) \leq \omega_n \Theta_V(\xi, r) r^n + c(n) a r^n \leq c(n) r^n.
\end{equation}
Therefore, by choosing a maximal $\tau/2$-net $\{ x_i \}_i$ in $\spt \sigma_{V, g} \cap W$, we can bound
\begin{align}
\tau^{-1} |\mu_{V_B}|(B_\tau(\spt \sigma_{V, g} \cap W)) &\leq \tau^{-1} \sum_i c(n) \tau^n \\
&\leq c(n) \tau^{-2} |B_\tau(\spt \sigma_{V, g} \cap W)| \\
&\leq c(n, \alpha) \sigma_{V, g}(B_\tau(W)). \label{eqn:cap-first-var-2}
\end{align}

Take $X \in C^1_c(W, \R^{n+1})$, and then decompose
\begin{align}
\delta_g V_B(X) 
&= \delta_g V_B( (1-\phi_\tau) X) + \delta_g V_B (\phi_\tau X) \\
&= -\int  (1-\phi_\tau) g(H_{\beta, g} + H_{R^n, g}, X) d\mu_{V_B} + \int \phi_\tau \mdiv_{\ver, g}(X) d\mu_{V_B} \\
&\quad \quad + \int g(\nabla^g_V \phi_\tau, X) d\mu_{V_B}.
\end{align}
Using \eqref{eqn:cap-first-var-2} we can estimate the last term by
\begin{align}
\int g(\nabla^g_V \phi_\tau, X) d\mu_{V_B}
&\leq c(n) \int |D\phi_\tau| |X| d|\mu_{V_B}| \\
&\leq c(n) |X|_{C^0} \tau^{-1} |\mu_{V_B}|(B_\tau(\spt \sigma_{V, g} \cap W)) \\
&\leq c(n, \alpha) |X|_{C^0} \sigma_{V, g}(B_\tau(W)).
\end{align}
Therefore, taking $\tau \to 0$ and noting that $\haus^n(\spt \sigma_{V, g} \cap B_\gamma) = 0$,we get
\begin{align}
\delta_g V_B(X) &\leq c(n) |X|_{C^0} ||H_{\beta, g} + H_{\R^n, g}||_{L^1(W; |\mu_{V_B}|)} + c(n, \alpha) |X|_{C^0} \sigma(\overline{W})
\end{align}

So $V_B$ has locally-finite $g$-first variation in $B_\gamma$, and for every open $W \subset B_\gamma$ we have $|\delta_g V_B|(W) \leq c(n) \delta |\mu_{V_B}|(W) + c(n, \alpha) \sigma_{V, g}(\overline{W})$.  Since $|\delta_g V_B|$, $|\mu_{V_B}|$, $|H_{\beta, g}\mu_{V_B}| \leq c \delta \haus^n \llcorner S$, $\sigma_{V, g}$ are Radon measures in $B_\gamma$, we deduce
\begin{align}
|\delta_g V_B|(W) &\leq c(n)\int_W (\delta + |H_{\beta, g}|) d|\mu_{V_B}| + c(n, \alpha) \sigma_{V, g}(W) \\
&\leq c(n) \delta \haus^n_g(S \cap W) + c(n, \alpha) \sigma_{V, g}(W) \quad \text{ for all Borel $W \subset B_\gamma$.}.
\end{align}
The fact that $x\mapsto\cos\beta(x)1_S$ is a function of locally bounded $g$-variation in $B_\gamma\cap\R^n$ is just a restatement of the bound $|\delta_g V_B| \leq c^2$.  The equivalence between $g$- and $g_{eucl}$-variation follows from the identity
\[
\int_S \mdiv_{g_{eucl}}(X) \cos\beta d\haus^n = \int_S \mdiv_g(X/\sqrt{g}) \cos\beta d\haus^n_g = \delta_g V_B(X/\sqrt{g})
\]
where $\sqrt{g}$ is the volume element of $g|_{\R^n}.$

Similarly, if $X \in C^1_c(B_\gamma \setminus \spt \sigma_{V, g})$ then from \eqref{eqn:cap-first-var-0.5} we have
\[
\delta_g V_I(X) = - \int g(H^{tan}_{V, g}, X) d\mu_{V_I},
\]
and therefore a verbatim argument (with $H^{tan}_{V, g}$ in place of $H_{\beta, g} + H_{\R^n, g}$) shows that $\delta_g V_I$ is locally-finite in $B_\gamma$, with instead the inequality
\[
|\delta_g V_I| \leq c(n) \delta \mu_{V_I} + c(n, \alpha)\sigma_{V, g}.
\]
\eqref{eqn:cap-first-var-concl1.5} then follows from the above and Lemma \ref{lem:sigma-upper}.  This proves Item 1.

\hyperref[lab:cap-3]{Item 3}: Take $x \in \spt V_I \cap B_\gamma$, and let $\xi = \pi_{\ver}(x) \in B_\gamma \cap \ver$.  For any $0 < r < 1/4$ we break into two cases: if $r \geq |x_1|$, we use \eqref{eqn:cap-first-var-1} to estimate
\[
\mu_{V_I}(B_r(x)) \leq |\mu_V|(B_{r+|x_1|}(\xi)) \leq c (r+|x_1|)^n \leq c r^n;
\]
while if $r < |x_1|$ we use the interior monotonicity of Lemma \ref{lem:mono} with \eqref{eqn:cap-first-var-1} to estimate
\[
\mu_{V_I}(B_r(x)) \leq c r^n |x_1|^{-n} \mu_{V_I}(B_{|x_1|}(x)) \leq c r^n |x_1|^{-n} |\mu_V|(B_{2|x_1|}(\xi)) \leq c r^n.
\]
where $c = c(n)$.

For the lower bound, let $\eta(n, \alpha/2)$ be the constant from Theorem \ref{thm:sigma-dichotomy}, and again break into two cases.  If $|x_1| \leq \eta r/4$, then since $V_I \llcorner B_{\eta r/3} \neq 0$ we can estimate using Theorem \ref{thm:sigma-dichotomy} and Lemma \ref{lem:sigma-upper}
\[
\mu_{V_I}(B_r(x)) \geq \mu_{V_I}(B_{2r/3}(\xi)) \geq r \sigma_{V, g}(B_{r/3}(\xi))/c(n) \geq r^n/c(n, \alpha).
\]
Conversely, if $|x_1| > \eta r/3$, then since $V_I$ is integral we have $\Theta_{V_I, g}(x) \geq 1$, and so by interior monotonicity \eqref{eqn:mono-weak-concl} we can estimate
\[
\mu_{V_I}(B_r(x)) \geq \mu_{V_I}(B_{\min\{r, |x_1|\}}(x)) \geq (\min\{ r, |x_1| \})^n/c(n) \geq \eta^n r^n/c(n).
\]
This proves Item 3.

\hyperref[lab:cap-4]{Item 4}: It follows from inequalities \eqref{eqn:cap-first-var-concl1}, \eqref{eqn:cap-first-var-concl1.2} that $\sigma_{V_I, g}, \sigma_{V_B, g}  \leq c \sigma_{V, g}$.  The first equality of \eqref{eqn:cap-first-var-concl4} then follows from the relation $\delta_g V = \delta_g V_I + \delta_g V_B$.  The orthogonality $g(\nu_{V_B, g}, \nu_g) = 0$ follows from the fact that if $X || \nu_g$, then
\[
\delta_g V_B(X) = - \int g(H_{\ver, g}, X) d\mu_{V_B} 
\]
and so $\delta_g V_B \ll |\mu_{V_B}|$ on the space of $X || \nu_g$.  To see the lower bound $g(\eta_{V_I, g}, \nu_g) \geq 1/c$, first note that from \eqref{eqn:cap-first-var-concl4} and the previous sentence we have $d \sigma_{V, g} = g(\eta_{V_I, g}, \nu_g) d\sigma_{V_I, g}$, and then use the Radon-Nikodyn theorem to deduce that for $\sigma_{V_I, g}$-a.e. $x \in B_\gamma$ we have
\begin{align}
g(\eta_{V_I, g}, \nu_g)|_x
&= \lim_{r \to 0} \frac{1}{\sigma_{V_I, g}\Big(\overline{B_r(x)}\Big)} \int_{\overline{B_r(x)}} g(\eta_{V_I}, \nu_g) d\mu_{V_I, g} \\
&=\lim_{r \to 0} \frac{\sigma_{V, g}\Big(\overline{B_r(x)}\Big) }{\sigma_{V_I, g}\Big(\overline{B_r(x)}\Big)} \\
&\geq 1/c.
\end{align}

\hyperref[lab:cap-5]{Item 5}: We already showed the first equality in \eqref{eqn:cap-first-var-concl6}.  The second inequality in \eqref{eqn:cap-first-var-concl6} follows from \eqref{eqn:cap-first-var-concl4}, which in fact implies
\[
c^{-1}\sigma_{V, g} \leq \sigma_{V_I, g} \leq c \sigma_{V, g}.
\]
To see \eqref{eqn:cap-first-var-concl6.5}, note that if $x \not \in \spt \sigma_{V, g} \cup \del \{ \cos\beta \neq 0 \}$, then in some ball $B_r(x) \cap \ver$ we have $\cos\beta \equiv 0$ or $\cos\beta > 0$ or $\cos\beta < 0$.  In the first case $\spt V_B \cap B_r(x) = \emptyset$, while in the second or third cases we have $\spt V_B \supset B_r(x) \cap \ver$, and hence in any of the three cases $x \not \in \del(\spt V_B)$
\end{proof}

\subsection{Boundary rigidity}\label{sec:rigidity}

In this section we consider (positive) $n$-rectifiable varifolds $V \equiv V_I$ that are only defined in the \emph{open} half-space $\{ x_1 < 0 \}$, with finite mass, locally finite $g$-first variation in $\{x_1<0\}$, bounded generalized mean curvature $H_{V,g}$ and zero generalized boundary $\sigma_{V,g}$ (see \eqref{eq:definition_of_H_and_sigma} and the discussion preceding it).  In particular the results in this section are true for $V$ without any kind of boundary condition.  Our key result is Theorem \ref{thm:rigidity}, which can be viewed as a baby Allard-type theorem saying that if the tilt-excess is sufficiently small, then the interior varifold $V_I$ must vanish in some small ball.
In other words, if $V$ is free-boundary in $\{ x_1 \leq 0 \}$ and the tilt excess (w.r.t. $\{x_1 = 0 \}$) is very small, then in a small ball $V$ must be supported in $\{ x_1 = 0 \}$.

\begin{lemma}[Tilt- vs $L^2$-excess]\label{lem:w12}
Let $g$ be a $C^1$ metric on $B_1$, and $V \equiv V_I$ an $n$-rectifiable varifold in $(B_1 \cap \{ x_1 < 0 \} ,g )$ with $\Theta_V(0, 1) < \infty$, locally finite first variation and $\sigma_{V,g}=0$.  Then for any $\beta \in (0, 1)$ we have the bound
\begin{align}
\int_{B_\beta} |\pi_V - \pi_{\ver}|^2 d\mu_V &\leq c \int_{B_1 \setminus B_\beta} |x_1|^2 d\mu_V \\
&\qquad+ c \max\{ ||H_{V, g}||_{L^\infty(B_1)}, |g - g_{eucl}|_{C^1(B_1)}\}\Theta_V(0, 1).
\end{align}
for $c = c(n, \beta, |g|_{C^0(B_1)})$.
\end{lemma}

\begin{proof}
Choose $\phi \in C^1_c(B_1)$ which $\equiv 1$ on $B_\beta$, and write
$$\Lambda = \max \{ ||H_{V, g}||_{L^\infty(B_1)}, |g - g_{eucl}|_{C^1(B_1)} \}.$$
For the vector field $X = \phi(|x|)^2 x_1 e_1$ we have at $|\mu_V|$-a.e. $x$:
\begin{align}
\dive_{V, g_{eucl}}(X) &= 2\phi \phi' x_1 \frac{x \cdot \pi_V(e_1)}{|x|} + \phi^2 |\pi_V(e_1)|^2  \\
&\geq \frac{1}{2}\phi^2 |\pi_V(e_1)|^2 - 2 (\phi')^2 x_1^2 \\
&= \frac{1}{4}\phi^2 |\pi_V - \pi_{\ver}|^2 - 2(\phi')^2 x_1^2.
\end{align}
On the other hand, recalling \eqref{eqn:div}, we have
\[
\dive_{V, g_{eucl}}(X) \leq \mathrm{div}_{V, g}(X) + c \Lambda 
\]
for $c = c(n, \beta, |g|_{C^0(B_1)})$.

Therefore, noting that $X \equiv 0$ on $\{ x_1 = 0 \}$ and using Remark \ref{rem:testing_with_vanish_vf}, we can plug $X$ into the first variation
to obtain
\begin{align}
\frac{1}{4} \int \phi^2 |\pi_V - \pi_{\ver}|^2 d\mu_V
&\leq \int -g(H_{V, g} , X) d\mu_V + 2\int (\phi')^2 x_1^2 d\mu_V +  c\Lambda \mu_V(B_1)\\
&\leq c \int_{B_1 \setminus B_\beta} x_1^2 d\mu_V + c \Lambda \mu_V(B_1). \label{eqn:w12-1}
\end{align}
The Lemma then follows by our choice of $\phi$.
\end{proof}

\begin{lemma}[Lower tilt excess control]\label{lem:lower-tilt}
Given $\beta \in (0, 1)$, there are constants $c(n), \delta(n, \beta), C(n, \beta)$ so that the following holds.  Let $g$ be a $C^1$ metric on $B_1$, and let $V \equiv V_I$ be an $n$-rectifiable varifold in $(B_1 \cap \{ x_1 < 0 \} , g)$ with locally finite first variation, $\sigma_{V,g}=0$ and satisfying
\[
\max \{ ||H_{V, g}||_{L^\infty(B_1)}, |g - g_{eucl}|_{C^1(B_1)} \} \leq \delta.
\]
Then for any $\xi \in B_\beta \cap \{ x_1 < -\beta/4\}$ we have the inequality
\[
\Theta_{V, g}(\xi) \leq (1+ c(n) \beta) \Theta_V(0, 1) + C(n, \beta) \int_{B_1} |\pi_V - \pi_{\ver}|^2 d\mu_V.
\]
\end{lemma}

\begin{proof}
Without loss of generality, we can assume $\delta \leq \beta$.  Choose an affine linear map $L$ satisfying
\begin{equation}\label{eqn:lower-tilt-1}
L(\xi) = \xi, \quad |DL - Id| \leq c(n) \delta, \quad L(\ver) = s e_1 + \ver, \quad (L^* g)(\xi) = g_{eucl},
\end{equation}
for some $|s| \leq c(n)\delta$.  Provided $\delta(n, \beta)$ is small, we have $L^{-1} B_1 \supset B_{1-\beta}$.  Let $V' = (L^{-1})_\sharp V$, and $g' = L^* g$, and then $V'$ is an $n$-rectifiable varifold in $(B_{1-\beta} \cap \{ x_1 < s \}, g')$ satisfying
\[
||H_{V', g'}||_{L^\infty(B_{1-\beta})} \leq \delta, \quad |g' - g_{eucl}|_{C^1(B_{1-\beta})} \leq c(n) \delta,
\]
and of course $g'(\xi) = g_{eucl}$ by construction.

We can use Lemma \ref{lem:mono} first with $h =h_2\coloneqq 1$ (and $\tau = 0, \sigma = \beta/8$) and then with $h = h_2\coloneqq f(x_1/\beta)$ (with $\tau = \beta/8, \sigma = 1-2\beta$), for $f$ a smooth function chosen so that 
\[
f(t) = 1 \text{ for } |t| \geq 1/8, \quad f(t) = 0 \text{ for } |t| < 1/16, \quad |f'(t)| \leq 100,
\]
to obtain the inequalities
\begin{align}
\omega_n \Theta_{V', g'}(\xi)
&\leq \frac{(1+c\delta \beta/8)^{n+1}}{(\beta/8)^n} \mu_{V'}(B_{\beta/8}(\xi)) \\
&= \frac{(1+c\delta \beta/8)^{n+1}}{(\beta/8)^n} \int_{B_{\beta/8}(\xi)} h_2 \dif\mu_{V'} \\
&\leq \frac{(1+c \delta)^{n+1}}{(1-2\beta)^n} \int_{B_{1-2\beta}(\xi)} h_2 \dif\mu_{V'} \\
&\quad+ \int_{\beta/8}^{1-2\beta} \frac{(1+c\delta)^{n+1}}{(\beta/8)^n} \int_{B_\rho(\xi)} |\nabla_{V'} h_2| \dif\mu_{V'} \dif\rho \\
&\leq (1+c \beta) \mu_{V'}(B_{1-\beta}) + C(n,\beta) \int_{B_{1-\beta}} |\nabla_{V'} x_1| \dif\mu_{V'}, \label{eqn:lower-tilt-2}
\end{align}
for $c = c(n)$ (recall we assume $\delta \leq \beta$).  Here we also assume $\delta(n, \beta)$ is small enough so that $\spt f(x_1/\beta) \subset \{ x_1 < s \}$.

Noting that (from \eqref{eqn:lower-tilt-1})
\[
|\nabla_{V'} x_1|(x) = |\pi_{V'}(e_1)|(x) \leq |\pi_{V'} - \pi_{\ver}|(x) \leq |\pi_V - \pi_{\ver}|(Lx) + c(n) \delta,
\]
and recalling that $L^{-1}B_1 \supset B_{1-\beta}$, we can continue our estimate \eqref{eqn:lower-tilt-2} to get
\begin{align}
\omega_n \Theta_{V, g}(\xi)
&\leq (1+2c\beta) \mu_{V}(B_1) + C(n, \beta) \int_{B_1} |\pi_V - \pi_{\ver}| d\mu_V \\
&\leq (1+3c\beta) \mu_V(B_1) + C(n, \beta) \int_{B_1} |\pi_V - \pi_{\ver}|^2 d\mu_V,
\end{align}
having used the Cauchy inequality $a \leq \beta^{-1} a^2 + \beta$.
\end{proof}

\begin{theorem}[Boundary rigidity]\label{thm:rigidity}
Given $\alpha \in (0, 1)$, there are $\eta(n, \alpha) \ll \gamma(n, \alpha)$ so that the following holds.  Let $g$ be a $C^1$ metric on $B_1$, and $V \equiv V_I$ an $n$-rectifiable varifold in $(B_1 \cap \{ x_1 < 0 \} , g )$ with locally finite first variation and $\sigma_{V,g}=0$, satisfying
\begin{gather}
\theta_{V, g}(x) \geq 1 \text{ $\mu_V$-a.e. $x$}, \quad \Theta_V(x, r) \leq 1-\alpha \text{ $\forall x \in \ver \cap B_{1/2}, r \in (0, 1/2)$},  \label{eqn:rigidity-hyp1} \\
\int_{B_1} |\pi_V - \pi_{\ver}|^2 d\mu_V \leq \eta, \label{eqn:rigidity-hyp2} \\
\max \{ ||H_{V, g}||_{L^\infty(B_1)}, |g - g_{eucl}|_{C^1(B_1)} \} \leq \eta. \label{eqn:rigidity-hyp3}
\end{gather}
Then $V \llcorner B_{\gamma} = 0$.
\end{theorem}

%

\begin{proof}
For notation convenience let us define $E(x, r)$ to be the tilt-excess of $V$ in $B_r(x)$ w.r.t. $\ver$:
\[
E(x, r) = r^{-n} \int_{B_r(x)} |\pi_V - \pi_{\ver}|^2 d\mu_V.
\]
We will take $\gamma = \beta^2/4$ for $\beta(n, \alpha)$ to be chosen below, but small enough so that $100 \beta < 1-\beta^2$.

Fix a $\xi \in B_{\beta^2} \cap \spt V \cap \{ x_1 < 0 \}$, let $x = \pi_{\ver}(\xi)$ and define $r = 2|\xi_1|/\beta$.  Then by our assumptions we can apply Lemma \ref{lem:lower-tilt} to the rescaled varifold $V_{x, r}$ (being an $n$-rectifiable varifold in $(B_1, g_{x, r})$) defined by
\begin{equation}\label{eqn:rigidity-1}
g_{x, r}(y) = g(x + ry), \quad V_{x, r} = (\eta_{x, r})_\sharp V, \quad \eta_{x, r} : (B_r(x), g) \to (B_1, g_{x, r})
\end{equation}
to obtain the inequalities
\begin{align}
1 &\leq \Theta_{V, g}(\xi) \leq (1+c \beta)\Theta_V(x, r) + C E(x, r) \\
&\leq (1+c(n)\beta)(1-\alpha) + C(n, \beta)E(x, r).
\end{align}
Therefore, if we ensure $\beta(n, \alpha)$ is sufficiently small can obtain
\begin{equation}\label{eqn:rigidity-2}
1 \leq c E(x, r) \leq c(n, \beta) E(\xi, 4|\xi_1|/\beta).
\end{equation}

One immediate consequence of \eqref{eqn:rigidity-2} is that
\begin{equation}\label{eqn:rigidity-3}
|\xi_1| \leq c(n, \beta) E(0, 1)^{1/n}.
\end{equation}
Another immediate consequence is that, if we define $r(\xi) = 4|\xi_1|/\beta$, then we have
\begin{gather}
r(\xi)^n \leq c(n, \beta) \int_{B_{r(\xi)}(\xi)} |\pi_V - \pi_{\ver}|^2 d\mu_V.
\end{gather}
Moreover, by \eqref{eqn:rigidity-hyp1}, it holds
\begin{gather}
\mu_V(B_{5r(\xi)}(\xi)) \leq \mu_V(B_{10r(\xi)}(x)) \leq (1-\alpha) \omega_n 10^n r(\xi)^n.
\end{gather}
(Recall that $x \in B_{\beta^2}$ and $10r(\xi) \leq 40 \beta < \frac{1 - \beta^2}{2}$.)

Now from the set $\{ B_{r(\xi)}(\xi) : \xi \in \spt V \cap B_{\beta^2} \cap \{ x_1 < 0 \} \}$ choose a Vitali subcollection $\{ B_{r_i}(\xi_i) \}_i$ so that the $r_i$-balls are disjoint and the $5r_i$-balls cover the aforementioned set.  Then we can estimate
\begin{align}
\mu_V(B_{\beta^2}) 
&\leq \sum_i \mu_V(B_{5r_i}(\xi_i)) \\
&\leq \sum_i c(n) r_i^n \\
&\leq \sum_i c(n, \beta)\int_{B_{r_i}(\xi_i)} |\pi_V - \pi_{\ver}|^2 d\mu_V \\
&\leq c(n, \beta) E(0, 1).
\end{align}

For any $x \in B_{\beta^2/2} \cap \ver$ and $R \leq 1/2$, the rescaled varifold $V_{x, R}$ and metric $g_{x, R}$ (as in \eqref{eqn:rigidity-1}) also satisfy \eqref{eqn:rigidity-hyp1}, and so applying the above reasoning at scale $B_{R}(x)$ we get
\begin{gather}
\spt V \cap B_{\beta^2 R}(x) \subset \{ R^{-1} |x_1| < c E(x, R)^{1/n} \},  \label{eqn:rigidity-4} \\
\text{ and } \quad \Theta_V(x, \beta^2 R) \leq c E(x, R) . \label{eqn:rigidity-5}
\end{gather}
for $c = c(n, \beta)$.

We claim that, provided $\eta(n, \beta)$ is chosen sufficiently small, we have
\begin{equation}\label{eqn:rigidity-6}
E(x, R) \leq \eta^{1/4} \quad \forall x \in B_{\beta^2/2} \cap \ver, R \leq 1/2.
\end{equation}
By our assumption \eqref{eqn:rigidity-hyp2} we clearly have
\begin{equation}
    E(x,1/2)= {2^n} \int_{B_{1/2}(x)} |\pi_V - \pi_{\ver}|^2 d\mu_V
    \leq
    {2^n} \int_{B_{1}} |\pi_V - \pi_{\ver}|^2 d\mu_V
    \leq
    2^n\eta
    \leq
    \eta^{1/4}
\end{equation}
for $x$ as above, provided that $\eta(n)$ was chosen sufficiently small.  Now for any $x, R$ as in \eqref{eqn:rigidity-6}, we compute using Lemma \ref{lem:w12}, \eqref{eqn:rigidity-4}, \eqref{eqn:rigidity-5}, \eqref{eqn:rigidity-hyp3}:
\begin{align}
E(x, \beta^2 R/2)
&\leq c(n, \beta) R^{-n-2} \int_{B_{\beta^2R}(x)} |x_1|^2 d\mu_V + c(n) \eta \Theta_V(x, \beta^2 R) \\
&\leq c(n, \beta) E(x, R)^{1+2/n} + c(n) \eta.
\end{align}
Therefore if $R_i = (\beta^2/2)^i (1/2)$, taking $\eta(n, \beta)$ small we can prove by induction that
\[
E(x, R_{i+1}) \leq c(n, \beta) \eta^{1/n} E(x, R_i) + c(n) \eta \leq \eta^{1/2} \quad \forall i = 0, 1, \ldots.
\]
Since $E(x, R) \leq 2^n \beta^{-2n} E(x, R_i)$ for some $i$, we deduce \eqref{eqn:rigidity-6} holds, proving our claim.

Now suppose, towards a contradiction, there is a $\xi \in \spt V \cap B_{\beta^2/4} \cap \{ x_1 < 0 \}$.  Let $x = \pi_{\ver}(\xi)$ and $R = 2|\xi_1|/\beta^2 \leq 1/2$, so that $\xi \in B_{\beta^2 R}(x)$.  Then from \eqref{eqn:rigidity-4}, \eqref{eqn:rigidity-6} we get
\[
\beta^2/2 = R^{-1} |\xi_1| \leq c(n, \beta) E(x, R)^{1/n} \leq c(n, \beta) \eta^{1/4n}
\]
which is a contradiction for $\eta(n, \beta)$ chosen sufficiently small.  This proves our theorem.
\end{proof}

\subsection{Ahlfors-regularity of boundary measure} \label{sec:sigma-ahlfors}

Here we establish upper and lower mass control on the boundary measure $\sigma_{V, g}$ of varifolds satisfying the assumptions of Theorem \ref{thm:cap-first-var}.  The upper bound is a trivial trace-type inequality.  The lower bound uses a trace identity coupled with the boundary rigidity of Theorem \ref{thm:rigidity}, to get a dichotomy which says: either $\sigma_{V, g}(B_1)$ has a definite lower bound, or $V_I \llcorner B_\eta = 0$ for some small (a priori) $\eta$.  One catch is that this dichotomy strictly speaking requires $V$ to look very close to conical, i.e. has a small density drop.  By the monotonicity formula $V$ has small density drop at ``most'' scales and so at every scale we can obtain a slightly weaker dichotomy of the same flavor.

\begin{lemma}[Lower bounds on $\sigma_{V,g}$ for nearly-conical $V$]\label{lem:sigma-lower}
Given $\alpha > 0$,
there are constants $\delta(n, \alpha) \ll \eta(n, \alpha)$ so that the following holds.  Let $g$ be a $C^1$ metric on $B_1$, $\beta$ a $C^1$ function, and $V$ a $(\beta, S)$-capillary varifold in in $(B_1, g)$.  Suppose
\begin{gather} 
\Theta_V(0, 1) + \big(\inf_{B_1}\cos\beta(x)\big)_- \leq 1 - \alpha, \quad \int_{B_1 \setminus B_{1/2}} \frac{|\pi_V^\perp(x)|^2}{|x|^{n+2}} d\mu_V \leq \delta, \label{eqn:sigma-lower-hyp2} \\
\max \{ ||H_{V, g}^{tan}||_{L^\infty(B_1)}, |g - g_{eucl}|_{C^1(B_1)}, |D\beta|_{C^0(B_1)} \} \leq \delta. \label{eqn:sigma-lower-hyp3}
\end{gather}

Then we have the dichotomy: either $\sigma_{V, g}(B_1) \geq \eta$, and/or $V_I \llcorner B_\eta = 0$.  In particular, if $0 \in \spt \sigma_{V, g}$ then $\sigma_{V, g}(B_1) \geq \eta$.
\end{lemma}

\begin{remark}
The dichotomies of Lemma \ref{lem:sigma-lower} and Theorem \ref{thm:sigma-dichotomy} are not necessarily exclusive -- both options could occur (but at least one has to).
\end{remark}

\begin{proof}[Proof of Lemma \ref{lem:sigma-lower}]
We claim that, provided $\gamma(n, \alpha)$, $\delta(n, \alpha)$ are chosen sufficiently small, we have
\[
\Theta_{V_I}(\xi, r) \leq 1-\alpha/2 \quad \forall \xi \in B_{\gamma} \cap \ver, \quad \forall 0 < r < 1/2.
\]

To see this, pick $\xi \in B_\gamma \cap \ver$, and then with $R = (1-c(n)\delta)(1-\gamma)$ for a suitable $c(n)$, we can use monotonicity \eqref{eqn:mono-weak-concl} and Remark \ref{rem:theta-change-ball} to compute for $\gamma(n, \alpha), \delta(n, \alpha)$ small
\begin{align}
\Theta_{V_I}(\xi, r) 
&\leq \Theta_V(\xi, r) + (1+c\delta) a \\
&\leq (1+c \delta) \Theta_V(\xi, R) + (1+c\delta) a \\
&\leq (1+c\delta + c\gamma)(\Theta_V(0, 1) + a) + (c\delta + c\gamma) a \\
&\leq 1 - \alpha/2,
\end{align}
where in the above $a\coloneqq(\inf_{B_1}\cos\beta)_-$ and $c = c(n, \alpha)$.  This proves our claim.

We now proceed with the proof of the Lemma.  Take $\phi(t)$ any smooth decreasing function which $\equiv 1$ on $(-\infty, 1/2]$ and $\equiv 0$ on $[1, \infty)$ and $|\phi'| \leq 4$.  Then by testing the first $g$-variation with the vector field $X = \phi(|x|) e_1$ we can compute
\begin{align}
2\sigma_{V, g}(\phi)
&\geq \int \phi(|x|) g(e_1, \nu_g) \dif\sigma_{V, g} \\
&\geq \int \dive_{V, g_{eucl}}(\phi(|x|) e_1) \dif\mu_V - c \delta |\mu_V|(B_1) \\
&= \int \frac{e_1 \cdot \pi_V(x)}{|x|} \phi' \dif\mu_V - c\delta |\mu_V|(B_1) \\
&\geq \int \frac{x_1}{|x|} \phi' \dif\mu_{V_I} - \int_{B_1 \setminus B_{1/2}} \frac{|\pi_V^\perp(x)|}{|x|^{1+n/2}} \dif\mu_{V_I} - c\delta |\mu_V|(B_1)  \\
&\geq \frac{1}{c} \int x_1^2 \phi'^2 \dif\mu_{V_I} - \left( \int_{B_1\setminus B_{1/2}} \frac{|\pi^\perp_V(x)|^2}{|x|^{n+2}} \dif\mu_{V_I} \right)^{1/2} - c \delta|\mu_V|(B_1) \\
&\geq \frac{1}{c} \int_{B_{1/2}} |\pi_V - \pi_{\ver}|^2 \dif\mu_{V_I} - \left( \int_{B_1 \setminus B_{1/2}} \frac{|\pi^\perp_V(x)|^2}{|x|^{n+2}} \dif\mu_{V_I} \right)^{1/2} - c \delta
\end{align}
where $c = c(n)$.
Let us elaborate.
In the first line we chose $\delta(n)$ small enough so that $g(\nu_g, e_1) \geq 1/2$.
In the second line we used \eqref{eqn:div} and \eqref{eqn:sigma-lower-hyp3}.
In the fourth line we used that  $\pi^\perp_V(x) = 0$ for $|\mu_V|$-a.e. $x_1 = 0$, and $\spt \phi'(|x|) \subset B_1 \setminus B_{1/2}$.  In the fifth line we used that $x_1 \phi' \geq 0$ on $\spt V$ and Holder's inequality ($\mu_{V_I}$ being a non-negative Radon measure).  In the last line we used \eqref{eqn:w12-1} from (the proof of) Lemma \ref{lem:w12}, and Remark \ref{rem:sigma-lower}.

Let $\eta_1, \gamma_1$ be the constants from Theorem \ref{thm:rigidity}, with $\alpha/2$ in place of $\alpha$.  From our Claim we can apply Theorem \ref{thm:rigidity} at scale $B_\gamma$, and therefore if $V_I \llcorner B_{\gamma_1 \gamma} \neq 0$ we have the lower bound $\int_{B_{\gamma}} |\pi_V - \pi_{\ver}|^2 \dif\mu_{V_I} \geq \gamma^n \eta_1$.  Ensuring $\delta(n, \gamma, \eta_1) \equiv \delta(n, \alpha)$ is small we then get the lower bound
\[
\sigma_{V, g}(B_1) \geq \frac{\gamma^n \eta_1}{4c(n)}.
\]
This gives us our dichotomy, and the last assertion follows from Remark \ref{rem:no-V_I}
\end{proof}

\begin{lemma}[Upper bounds on $\sigma_{V,g}$]\label{lem:sigma-upper}
Let $g$ be a $C^1$ metric on $B_1$, $\beta$ a $C^1$ function, and $V$ a $(\beta, S)$-capillary varifold in $(B_1, g)$ satisfying
\[
\max \{ ||H_{V, g}^{tan}||_{L^\infty(B_1)}, |g - g_{eucl}|_{C^1(B_1)} , |D\beta|_{C^0(B_1)}\} \leq \Lambda < \infty.
\]
Then for any $\phi \in C^1_c(B_1)$ non-negative we have
\[
\sigma_{V, g}(\phi) \leq \int\big( |\nabla^g_V \phi| + c(n) \Lambda \phi \big)\dif\mu_{V_I} .
\]
In particular, we have $\sigma_{V, g}(B_{1/2}) \leq c(n)(1+\Lambda) \mu_{V_I}(B_1)$.
\end{lemma}

\begin{proof}
Let $\nu_g = g^{1j} e_j/\sqrt{g^{11}}$ be a choice of unit normal to $\ver \subset (\R^{n+1}, g)$ pointing in the positive $e_1$ direction.  Recalling that for $|\mu_V|$-a.e. $x \in \ver$,  we have $T_x V = \ver$, and (from Theorem \ref{thm:cap-first-var}) $g(H_{V, g}, \nu_g) = g(H_{\ver, g}, \nu_g) = -\mdiv_{\ver, g}(\nu_g)$, we can compute
\begin{align}
\sigma_{V, g}(\phi)
&= \int \big(\mdiv_{V, g}(\phi \nu_g) + \phi g(H_{V, g}, \nu_g)\big) \dif\mu_V \\
&= \int \big(\mdiv_{V, g}(\phi \nu_g) + \phi g(H_{V, g}^{tan}, \nu_g)\big) \dif\mu_{V_I} \\
&\quad + \int \phi\big(\mdiv_{\ver, g}(\nu_g) + g(H_{\ver, g}, \nu_g)\big) \dif\mu_{V_B} \\
&= \int \big(\mdiv_{V, g}(\phi \nu_g) + \phi g(H_{V, g}^{tan}, \nu_g)\big)  \dif\mu_{V_I} \\
&\leq \int \big(|\nabla^g_V \phi| + c(n) \Lambda \phi\big) \dif \mu_{V_I}. \qedhere
\end{align}
\end{proof}

\begin{theorem}[Mass dichotomy for $\sigma_{V, g}$ and $V_I$]\label{thm:sigma-dichotomy}
Given $\alpha > 0$,
there are $\eta(n, \alpha) \ll \delta(n, \alpha) \ll \gamma(n, \alpha) < 1/2$, $c(n, \alpha)$ so that the following holds.  Let $g$ be a $C^1$ metric on $B_1$, $\beta$ a $C^1$ function, and $V$ a $(\beta, S)$-capillary varifold in $(B_1, g)$.  Suppose
\begin{gather}
\Theta_V(0, 1) + (\inf_{B_1}\cos\beta(x))_-\leq 1-\alpha, \\
\max \{ ||H^{tan}_{V, g}||_{L^\infty(B_1)}, |g - g_{eucl}|_{C^1(B_1)}, |D\beta|_{C^0(B_1)} \} \leq \delta.
\end{gather}

Then for every $\xi \in B_\gamma \cap \ver$ and every $0 < r < 1/4$, we have the dichotomy:
\begin{align}
&\text{either} \quad r^{n-1}/c \leq \sigma_{V, g}(B_{r}(\xi)) \leq c r^{n-1}, \label{eqn:sigma-dichotomy-concl1} \\
&\text{and/or} \quad V_I \llcorner B_{\eta r}(\xi) = 0; \label{eqn:sigma-dichotomy-concl2}
\end{align}
the upper bound in \eqref{eqn:sigma-dichotomy-concl1} will always hold.  In particular, the bounds \eqref{eqn:sigma-dichotomy-concl1} hold whenever $\xi \in \spt \sigma_{V, g} \cap B_\gamma$.
\end{theorem}

\begin{proof}
Fix $\xi \in B_\gamma \cap \ver$.  Choose an affine linear map $L$ satisfying
\begin{equation}\label{eqn:sigma-ahlfors-0.1}
L(\xi) = \xi, \quad |DL- \operatorname{Id}| \leq c(n)\delta, \quad L(\ver) = \ver, \quad (L^* g)(\xi) = g_{eucl} .
\end{equation}
Define $V' = (L^{-1})_\sharp V$, $g' = L^* g$, $\beta' = \beta \circ L$.  Then with $R = 1-c(n)\delta$ for some suitable constant $c(n)$, $V'$ is a $(\beta', L^{-1}(B))$-capillary varifold in $(B_R, g')$.  Provided $\delta(n, \alpha)$ is chosen sufficiently small, we can use Remark \ref{rem:theta-change-ball} to show $(V', g')$ satisfies
\begin{gather}
\inf_{B_R} \cos\beta'(x) \geq -a, \label{eqn:sigma-ahlfors-.2} \\
\Theta_{V'}(0, R) + a \leq (1-c(n)\delta)^{-n} (\Theta_V(0, 1) + c(n) \delta a) + a \leq 1 - 3\alpha/4, \label{eqn:sigma-lower-.3} \\
\max \{ ||H_{V', g'}||_{L^\infty(B_R)}, |g' - g_{eucl}|_{C^1(B_R)}, |D\beta'|_{C^0(B_R)} \} \leq c(n)\delta, \label{eqn:sigma-ahlfors-.4}
\end{gather}
where in the above and in the following $a\coloneqq(\inf_{B_1}\cos\beta)_-$.

Let $c_0(n)$ be the constant from Lemma \ref{lem:bd-mono}, and define
\[
h(r) = (1+c_0 \delta r)^{n+1}\Theta_{V'}(\xi, r) + \omega_n^{-1} c_0 \delta r.
\]
Then $h(r)$ is increasing with $r$, and
\begin{equation}\label{eqn:sigma-ahlfors-1}
h(r) - h(s) \geq \frac{1}{\omega_n} \int_{B_r(\xi) \setminus B_s(\xi)} \frac{|\pi_V^\perp(x - \xi)|^2}{|x - \xi|^{n+2}} \dif\mu_{V_I'}(x), 
\end{equation}
and (using Remark \ref{rem:theta-change-ball} and \eqref{eqn:sigma-lower-.3})
\begin{equation}\label{eqn:sigma-ahlfors-2}
h(R - \gamma) - h(0_+) \leq (1+ c\delta + c\gamma)\Theta_V(0, R) + c\delta + c\gamma + a \leq 1,
\end{equation}
for $c = c(n)$.

By a similar computation to Lemma \ref{lem:sigma-lower}, we can estimate for any $r \leq R-\gamma$:
\begin{equation}\label{eqn:sigma-ahlfors-3}
\Theta_V(\xi, r) + a \leq h(r) + (1+c\delta) a \leq h(R - \gamma) + (1+c\delta) a \leq 1 - \alpha/2
\end{equation}
and (hence)
\begin{equation}\label{eqn:sigma-ahlfors-4}
|\mu_V|(B_r(\xi)) \leq \omega_n (\Theta_V(\xi, r) + (1+c\delta )a)r^n \leq \omega_n r^n, 
\end{equation}
where again $c = c(n)$.  Of course in \eqref{eqn:sigma-ahlfors-2}, \eqref{eqn:sigma-ahlfors-3}, \eqref{eqn:sigma-ahlfors-4} we take $\delta(n, \alpha), \gamma(n, \alpha)$ small.


Let $K = 2^{2/\delta}$.  Then for every $r \leq (R-\gamma)/2$, by monotonicity of $h(r)$ and \eqref{eqn:sigma-ahlfors-2} we can find $r' \in [r/K, r]$ so that $h(r') - h(r'/2) \leq \delta$.  From \eqref{eqn:sigma-ahlfors-1}, \eqref{eqn:sigma-ahlfors-3}, \eqref{eqn:sigma-ahlfors-4}, we can therefore apply Lemma \ref{lem:sigma-lower} at scale $B_{r'}(\xi)$ and Lemma \ref{lem:sigma-upper} at scale $B_{2r}(\xi)$ to get:
\begin{align}
&\text{either} \quad r^{n-1}/c(n, \alpha) \leq \sigma_{V', g'}(B_{r}(\xi)) \leq c(n, \alpha) r^{n-1} \\
&\text{and/or} \quad V_I' \llcorner B_{\delta r/K}(\xi) = 0.
\end{align}
(Of course the upper bounds will always hold.)  Our dichotomy then follows by recalling $\sigma_{V, g}(B_r(\xi)) = \sigma_{V', g'}(L B_r(\xi))$ and the bounds \eqref{eqn:sigma-ahlfors-0.1} for $L$.  The very last assertion follows from Remark \ref{rem:no-V_I}.
\end{proof}

\subsection{Improved convergence}\label{sec:conv}

A key consequence Theorem \ref{thm:cap-first-var} is an improved varifold convergence theorem, which effectively states that the interior and boundary pieces converge nicely and independently from each other.

\begin{lemma}[Improved compactness/convergence]\label{lem:conv}
For any $\alpha > 0$, there is a $\gamma(n, \alpha) > 0$ with the following property.  Let $g_i$ be $C^1$ metrics on $B_1$, $\beta_i$ be $C^1$ functions, and $V_i$ be $(\beta_i, S_i)$-capillary varifolds in $(B_1, g_i)$ satisfying
\begin{gather}
\Theta_{V_i}(0, 1) + (\cos\beta_i(0))_- \leq 1 - \alpha,  \label{eqn:conv-hyp1} \\
\max \{ ||H^{tan}_{V_i, g_i}||_{L^\infty(B_1)}, |g_i - g_{eucl}|_{C^1(B_1)}, |D\beta_i|_{C^0(B_1)} \} \to 0. \label{eqn:conv-hyp2}
\end{gather}
Then after passing to a subsequence, there is a $n$-rectifiable varifold $V$ in $(B_1, g_{eucl})$ which is supported in $\{x_1 \leq 0 \}$ and is stationary (i.e. $H_V\equiv0$) with free-boundary in $\{ x_1 = 0 \}$, so that $V_i \to V$ as varifolds in $B_1$.

\emph{Additionally}, $V$ is a $(\beta_0, S_0)$-capillary varifold in $(B_\gamma, g_{eucl})$, for $\beta_0 \in \R$ chosen so that $\cos\beta_i \to \cos\beta_0$, and in $B_\gamma$ we have the following improved convergence:
\begin{enumerate}
\item[A.] $V_{I, i} \to V_I$ and $V_{B, i} \to V_B$ as varifolds; \label{item:conv-a}
\item[B.] $\spt V_{I, i} \to \spt V_I$ in the Hausdorff distance; \label{item:conv-b}
\item[C.] $\sigma_{V_i,g_i} \to \sigma_{V}$ as Radon measures, and $\spt \sigma_{V_i,g_i} \to \spt \sigma_{V}$ in the Hausdorff distance. \label{item:conv-c}
\end{enumerate}
Moreover, we highlight two special cases:
\begin{enumerate}
\item\label{item:limit_is_cone}if $V = V^{(\beta_0)}$, then conclusions A-B-C above in fact hold in any $B_r \Subset B_1$;

\item \label{item:extra_outside_spt_sigma} if $\cos\beta_0 \neq 0$, then for any $\eps > 0$ and $i \gg 1$ (depending on $\eps$), on every connected component $S'$ of $\ver \cap B_{\gamma-\eps} \setminus B_\eps(\spt \sigma_{V})$ we have either
\begin{equation}
\mu_{V_i}  \llcorner S' = \cos\beta_i [S']_{g_i} \quad \text{or}\quad \mu_{V_i} \llcorner S' = 0.
\end{equation}
\end{enumerate}
\end{lemma}

\begin{remark}
Conclusion A (and the fact that $V \llcorner B_\gamma$ is a capillary varifold) could be also deduced from \eqref{eqn:soft-conv-concl1}, the monotonicity Lemma \ref{lem:bd-mono}, (and Remark \ref{rem:soft-conv-density}), while convergence $\sigma_{V_i, g_i} \to \sigma_V$ is true in the more general setting of Lemma \ref{lem:soft-conv}.  The real extra information here is the Hausdorff convergence $\spt V_{I,i} \to \spt V_I$ and $\spt \sigma_{V_i, g_i} \to \spt \sigma_V$.
\end{remark}

\begin{proof}
Without loss of generality we can choose $\beta_0 \in \R$ so that $\cos\beta_i \to \cos\beta_0$ in $C^1(B_1)$.  The existence of $V$ and convergence $V_i \to V$, $\sigma_{V_i, g_i} \to \sigma_V$ is from Lemma \ref{lem:soft-conv}.  By our assumptions and from Lemma \ref{lem:soft-conv} we also have $\Theta_V(0, 1) + (\cos\beta_0)_- \leq 1 - \alpha$.

From Theorem \ref{thm:cap-first-var}, the first variations of $V_{I, i}$ and $V_{B, i}$ are uniformly bounded in $B_\gamma$.  Therefore, passing to a further subsequence, we can assume there are varifolds $V', V''$ in $B_\gamma$ so that in $B_\gamma$ we have $V_{I, i} \to V'$, $V_{B, i} \to V''$, and (hence) $V = V' + V''$.  Since each $V_{I, i}$ is integral, then $V'$ is integral also.

If $\cos\beta_0 = 0$ then $\mu_{V_{B, i}}(B_1) \to 0$ so $V'' = 0 \equiv \cos\beta_0[S_0]$ for any choice of $S_0$. If $\cos\beta_0 \neq 0$, then by Theorem \ref{thm:cap-first-var} the $S_i$ are sets of uniformly bounded perimeter in $\R^n \cap B_\gamma$, and hence we can find a set of finite perimeter $S_0 \subset \R^n \cap B_\gamma$ so that (passing to a further subsequence) $S_i \to S_0$ in $L^1(\R^n \cap B_\gamma)$.  In particular since $\cos\beta_i \to \cos\beta_0$ in $C^1(B_1)$ we deduce $V'' = \cos\beta_0 [S_0]$.

By monotonicity Lemma \ref{lem:bd-mono} and Remark \ref{rem:theta-change-ball}, after shrinking $\gamma(n, \alpha)$ as necessary, for any $x \in B_\gamma \cap \R^n$ we have
\[
\Theta_V(x) + (\cos\beta_0)_- \leq (1+c(n)\gamma) \Theta_V(0, 1) + (\cos\beta_0)_- + c(n)\gamma \leq 1 - \alpha/2 .
\]
On the other hand, since $V'$ is integral and $V'' = \cos\beta_0[S_0]$ we have $\Theta_V(x) \geq 1 - (\cos\beta_0)_-$ for $\mu_{V'}$-a.e. $x \in \R^n \cap B_\gamma$.  Therefore we must have $\mu_{V'}(\R^n \cap B_\gamma) = 0$, which implies $V_I = V'$, $V_B = V''$, and that $V$ is a $(\beta_0, S_0)$-capillary varifold in $B_\gamma$.

The Hausdorff convergence $\spt V_{I, i} \to \spt V_I$ and $\spt \sigma_{V_i,g} \to \spt \sigma_V$ in $B_\gamma$ follows directly from Radon measure convergence and the lower Ahlfors estimates of Theorem \ref{thm:cap-first-var}.  This proves the first part of the Lemma.

If $V = V^{(\beta_0)}$, we observe by the monotonicity formula, the conical nature of $V^{(\beta_0)}$, and our assumption \eqref{eqn:conv-hyp1}, that for any fixed $r < 1$, we have
\[
\Theta_{V}(x, (1-r)/2) + (\cos\beta_0)_- \leq \Theta_{V}(0, 1) + (\cos\beta_0)_- = (1+|\cos\beta_0|)/2 < 1
\]
Hence $\Theta_{V_i}(x, (1-r)/2) \leq 3/4 + |\cos\beta_0|/4 < 1$ for $x \in B_r$ and $i \gg 1$.  We can then apply the first part of the Lemma in every ball $B_{\gamma (1-r)/2}(x)$, $x \in B_{r}$, to deduce special case \ref{item:limit_is_cone}.

To see special case \ref{item:extra_outside_spt_sigma}, note that since $\{ \cos\beta_i = 0 \} = \emptyset$ for $i \gg 1$, and since $\spt \sigma_{V_i, g_i} \equiv \spt V_{I, i} \cap \R^n \to \spt \sigma_V$ in $B_\gamma$, we have
\[
\overline{\del^*S_i} \cap B_\gamma \subset \del\spt V_{B, i} \cap B_\gamma \subset \spt\sigma_{V_i, g_i} \subset B_\eps(\spt \sigma_V) 
\]
for all $i$ large.
\end{proof}

We also make explicit one other special case, phrased in terms of varifold closeness rather than convergence.
\begin{lemma}\label{lem:D-to-M}
Given $\theta \in (0, \pi)$ and $\eps > 0$, there is a $\delta(n, \theta, \eps) > 0$ so that the following holds. Let $g$ be a $C^1$ metric on $B_1$, $\beta$ a $C^1$ function on $B_1$, and $V$ be a $(\beta, S)$-capillary varifold in $(B_1, g)$ satisfying:
\begin{gather}
		\D(V,V^{(\theta)})<\delta,\\
		\max\{\|H^{tan}_{V,g}\|_{L^\infty(B_1)}, |g-g_{eucl}|_{C^1(B_1)}, |\beta-\theta|_{C^1(B_1)}\}\le \delta.
	\end{gather}
	Then
	\begin{gather}\label{eqn:lemma.D.to.M}
		\mu_V \llcorner B_{1-\eps} \cap \{ x_1 = 0, x_{n+1} < -\eps\} = \haus^n_g \llcorner \cos\beta , \\
		\mu_V \llcorner B_{1-\eps} \cap  \{ x_1 = 0, x_{n+1} > \eps\} = 0.
	\end{gather}
\end{lemma}

\begin{proof}
Proof by contradiction: if we had a sequence $V_i \to V^{(\theta)}$ with $\beta_i \to \theta$ in $C^1(B_1)$, then for $i \gg 1$ we would have $\Theta_V(0, 1-\eps/2) + (\cos\beta_i(0))_- < 3/4 + \cos\theta/4 < 1$ (note we are \emph{assuming} $\beta \to \theta$).  But then we can use Lemma \ref{lem:conv} to can conclude $V_i$ satisfies \eqref{eqn:lemma.D.to.M} for $i \gg 1$.
\end{proof}


Lastly, we will require the following Lemma, which says that the varifold distance to $V^{(\theta)}$ is (implicitly) comparable to the oscillation distance to $P_\theta$, under suitable assumptions.
\begin{lemma}[$L^\infty$-vs-varifold distance]\label{lem:D-vs-osc}
Given $\theta_0 > 0$ and $\eps > 0$ there is a $\delta(n, \eps, \theta_0)$ so that the following holds.  Let $\theta \in (\theta_0, \pi/2-\theta_0) \cup (\pi/2+\theta_0, \pi - \theta_0)$, $g$ be a $C^1$ metric on $B_1$, $\beta$ a $C^1$ function, and $V$ an $(\beta, S)$-capillary varifold in $(B_1, g)$ satisfying
\begin{gather}
0 \in \spt V_I , \quad \Theta_V(0, 1) + (\cos\beta(0))_- \leq 3/4 + |\cos\theta_0|/4, \label{eqn:D-vs-osc.hyp2} \\
\max\{ ||H^{tan}_{V, g}||_{B_1}, |g - g_{eucl}|_{C^1(B_1)}, |D\beta|_{C^0(B_1)} \} \leq \delta. \label{eqn:D-vs-osc.hyp3}
\end{gather}
Then
\begin{align}\label{eqn:D-vs-osc.concl1}
\osc_{P_{\theta}}(V_I, B_1) \leq \delta \implies \D(V, V^{(\theta)}) \leq \eps,
\end{align}
and conversely
\begin{equation}\label{eqn:D-vs-osc.concl2}
\D(V, V^{(\theta)}) \leq \delta \implies \osc_{P_{\theta}}(V_I, B_{3/4}) \leq \eps.
\end{equation}
In either scenario we can also conclude that $|\beta - \theta|_{C^0(B_1)} \leq \eps$ (after translating $\beta$ by a multiple of $\pi$ if necessary).
\end{lemma}

\begin{proof}
Suppose, towards a contradiction, there are $\delta_i \to 0$, $\theta_i \to \theta \in [\theta_0, \pi/2-\theta_0] \cup [\pi/2+\theta_0, \pi - \theta_0]$, and $V_i$, $g_i$, $\beta_i$ satisfying \eqref{eqn:D-vs-osc.hyp2}, \eqref{eqn:D-vs-osc.hyp3}, and
\begin{gather}\label{eqn:D-vs-osc-1}
\osc_{P_{\theta_i}}(V_{I, i}, B_1) \leq \delta_i,
\end{gather}
but for which $\inf_i \D(V_i, V^{(\theta_i)}) > 0$.  Note that since $0 \in \spt V_{I, i}$, \eqref{eqn:D-vs-osc-1} implies $\spt V_{I, i} \cap B_1 \subset B_{\delta_i}(P_{\theta_i})$.

We can assume $\cos\beta_i \to \cos\theta'$ in $C^1(B_1)$ for some $\theta' \in [0, \pi]$.  By Lemma \ref{lem:soft-conv} we can also assume the $V_i \to V$ for some $n$-rectifiable $V$ in $(B_1, g_{eucl})$ which is supported in $\{ x_1 \leq 0 \}$, and is stationary with free-boundary in $\{ x_1 = 0 \}$.  $V_I$ is integral and by our contradiction hypothesis \eqref{eqn:D-vs-osc-1} $\spt V_I \subset P_\theta$, and therefore the constancy theorem implies $V_I = k[P_\theta^-]$ for some $k \in \{0, 1, 2, \ldots \}$.

By Lemma \ref{lem:conv} we can assume $\spt V_{I, i} \to \spt V_I$ in the Hausdorff distance $B_\gamma$, and since $0 \in \spt V_{I, i}$ we have $0 \in \spt V_i$ also, so $k \geq 1$.  On the other hand, since $\theta_V(x) \geq -(\cos\theta')_-$ at $|\mu_V|$-a.e. $x_1 = 0$, the upper bound $\Theta_V(0, 1) + (\cos\theta')_- < 1$ implies $k \leq 1$.  So in fact $V_I = [P_\theta^-]$.

From the above paragraph and the constancy theorem we have that $\theta_{V}$ is locally-constant on $\ver \setminus \{ x_1 = x_{n+1} = 0\}$.  On the other hand, Lemma \ref{lem:conv} implies $\theta_{V}(x) = \cos\theta'$ at $|\mu_V|$-a.e. $x \in \ver \setminus \{x_1 = x_{n+1} = 0\}$.  So
\[
V = [P_\theta^-] + k_-[\{ x_1 = 0, x_{n+1} < 0] + k_+ [x_1 = 0, x_{n+1} > 0]
\]
for $k_\pm \in \{0, \cos\theta'\}$.  Stationarity implies we must have $k_- = \cos\theta = \cos\theta'$ and $k_+ = 0$.  So $V = V^{(\theta)}$, and hence $\D(V_i, V^{(\theta)}) \to 0$, which is a contradiction.  This proves \eqref{eqn:D-vs-osc.concl1}.

The converse \eqref{eqn:D-vs-osc.concl2} follows directly from Lemma \ref{lem:conv}:  if $V_i \to V^{(\theta)}$ as varifolds in $B_1$, then in $B_{3/4}$ we have $\spt V_{I, i} \to P_\theta^-$, and $V_{B, i} \to \cos\theta[x_1 = 0, x_{n+1} < 0]$.  These imply $\osc_{P_\theta}(V_{I, i}, B_{3/4}) \to 0$ and 
\[
\cos\beta_i [S]_{g_i} \to \cos\theta[S], \quad S \supset B_{1/2} \cap \R^n \cap \{ x_{n+1} < -1/4 \},
\]

from which we get $\cos\beta_i \to \cos\theta$.
\end{proof}

\subsection{Refined boundary structure}\label{subsec:refined_bdry}

Though we will not require the results of this section for our maximum principle or Allard-type regularity, the following ``refined'' boundary structure will be useful in the almost-everywhere regularity of Theorem \ref{thm:teaser3}.  The ideas behind Theorem \ref{thm:refined} build on those in \cite{demasi}, who proved $(n-1)$-rectifiability of the boundary measure for free-boundary varifolds.

\begin{theorem}\label{thm:refined}
Given $\alpha > 0$, there are $\delta(n, \alpha)$, $\gamma(n, \alpha)$ so that the following holds.  Let $g$ be a $C^1$ metric on $B_1$, $\beta$ a $C^1$ function on $B_1$, $S \subset \R^n$, and $V$ a $(\beta, S)$-capillary varifold in $(B_1, g)$.  Suppose that
\begin{gather}
\Theta_V(0, 1) + (\cos\beta(0))_- \leq 1-\alpha, \\
\max \{ ||H^{tan}_{V, g}||_{L^\infty(B_1)}, |g - g_{eucl}|_{C^1(B_1)}, |D\beta|_{C^0(B_1)} \} \leq \delta.
\end{gather}
Define $BS := B_\gamma \cap \del^*S \cap \{ \beta \neq \pi/2\}$ (c.f. Remark \ref{rem:S-perimeter}).

Then in $B_\gamma$ it holds that: $\spt\sigma_{V, g} = \spt\sigma_{V_I, g}$ is $(n-1)$-rectifiable; for $\sigma_{V, g}$-a.e. $x$, every tangent cone of $V$ at $x$ is some rotation of $V^{(\beta(x))}$ (if $x \in BS$) or either $V^{(\pi/2)}$ or $V^{(\pi/2)} + \cos\beta(x)[\R^n]$ (if $x \in \spt\sigma_{V,g} \setminus BS$); and in $B_\gamma$ we have
\begin{equation}\label{eqn:refined-concl1}
\eta_{V_I, g} d\sigma_{V_I, g} = \eta_{V_I, g} d\haus^{n-1}_g \llcorner BS + \nu_g d\haus^{n-1}_g \llcorner (\spt \sigma_{V, g} \setminus BS), 
\end{equation}
such that $g(\eta_{V_I} , \nu_g) = \cos\beta(x)$ for $\haus^{n-1}_g$-a.e. $x \in BS$.  Note that a direct consequence of \eqref{eqn:refined-concl1} and \eqref{eqn:cap-first-var-concl4} is that, in $B_\gamma$:
\begin{equation}\label{eqn:refined-concl2}
d\sigma_{V, g} = \cos\beta(x) d\haus^{n-1}_g \llcorner BS + d\haus^{n-1}_g \llcorner (\spt \sigma_{V,g} \setminus BS).
\end{equation}
\end{theorem}

\begin{remark}
We think it likely that Theorem \ref{thm:refined} holds with $\spt \sigma_{V_B}$ in place of $BS$.  The issue is ruling out the possibility that $(\overline{\del^*S} \setminus \del^* S) \cap \{ \beta \neq \pi/2 \}$ has positive $\haus^{n-1}_g$-measure.  See Remark \ref{rem:example-collision} for more discussion.
\end{remark}

\begin{proof}[Proof of Theorem \ref{thm:refined}]

Ensuring $\gamma(n, \alpha)$, $\delta(n, \alpha)$ are sufficiently small, by monotonicity \eqref{eqn:mono-weak-concl} and Remark \ref{rem:theta-change-ball} we can assume
\begin{equation}\label{eqn:refined-1}
\Theta_V(x, r) + (\cos\beta(x))_- \leq 1 - \alpha/4 \quad \forall x \in B_\gamma \cap \R^n, \quad \forall r \in (0, 1-|x|).
\end{equation}
Ensure also that $\gamma$ is smaller than the $\gamma(n, \alpha/2)$ from Theorem \ref{thm:cap-first-var}.  Let us recall that for some $c = c(n, \alpha)$ we have from Theorem \ref{thm:cap-first-var}.
\begin{equation}\label{eqn:refined-2}
c^{-1} \haus^{n-1}_g \leq \sigma_{V,g} \leq c \haus^{n-1}_g, \quad c^{-1} \sigma_{V_I,g} \leq \sigma_{V,g} \leq c \sigma_{V_I,g} \quad \text{ on } B_\gamma.
\end{equation}



Choose an $x \in \spt V_I \cap \R^n \cap B_\gamma \equiv \spt \sigma_V \cap B_\gamma$, and suppose $g(x) = g_{eucl}$.  From \eqref{eqn:refined-1}, Lemma \ref{lem:conv}, and monotonicity Lemma \ref{lem:bd-mono}, for any $r_i \to 0$, up to a subsequence the dilations $V_i := (\eta_{x, r_i})_\sharp V \to V'$ (where each $V_i$ lives in the metric $g_i = g \circ \eta_{x, r_i}^{-1}$) for $V'$ a conical, $(\beta' \equiv \beta(x), S')$-capillary varifold in $\R^{n+1}$ satisfying
\begin{equation}\label{eqn:refined-3}
\Theta_{V'}(0, 1) = \Theta_{V,g}(x) \leq 1 - \alpha/4 - (\cos\beta')_-.
\end{equation}
Moreover, we have $V_{I, i} \to V_I'$, $V_{B, i} \to V_B'$, and $\sigma_{V_i, g_i} \to \sigma_{V'}$.

Since $\Theta_{V,g}$ is upper-semi-continuous on $B_1 \cap \R^n$, for $\sigma_{V, g}$-a.e. $x$ we can suppose $x$ is a Lebesgue point of $\Theta_{V, g}|_{\spt \sigma_{V, g} \cap B_1}$ w.r.t. $\sigma_{V, g}$.  Fix such an $x$.  Then we claim we must have $\Theta_{V'}(y) = \Theta_{V'}(0)$ for all $y \in \spt \sigma_{V'}$.  Otherwise, if $y \in \spt \sigma_{V'}$ satisfies $\Theta_{V'}(y) < \Theta_{V'}(0)$, then for some $\eps > 0$ we have $\Theta_{V_i, g_i}(z) < \Theta_{V'}(0) - \eps$ for all $z \in B_\eps(y)$ and all $i \gg 1$.  From Theorem \ref{thm:cap-first-var} with $R = |y| + 1$ we have
\[
\eps^{n-1}/c(n,\alpha) \leq \sigma_{V_i, g_i}(B_\eps(y)) \leq \sigma_{V_i, g_i}(\overline{B_R}) \leq c(n, \alpha) R^{n-1}
\]
for all $i \gg 1$.  We then deduce the contradiction
\begin{align}
\fint_{\overline{B_{Rr_i}(x)}} |\Theta_{V, g} - \Theta_{V, g}(x)| d\sigma_{V, g}
&= \fint_{\overline{B_R}} |\Theta_{V_i, g_i} - \Theta_{V_i, g_i}(0)| d\sigma_{V_i, g_i} \\
&\geq \frac{\sigma_{V_i, g_i}(\overline{B_\eps(y))}}{\sigma_{V_i, g_i}(\overline{B_R})} \\
&\geq 1/c(n, \alpha, \eps, y) .
\end{align}
This proves our claim.

Since $\spt\sigma_{V'} \neq \emptyset$ and $\haus^{n-1}(\spt\sigma_{V'} \cap B_1) > 0$, we deduce by monotonicity that $V'$ has at least $(n-1)$-dimensions of translational symmetry.  Using \eqref{eqn:refined-3}, that $V'_B = \cos\beta' [S'] \equiv \cos\beta(x) [S']$, and that $V'_I$ is integral, we deduce that (up to rotation)
\begin{gather}
V' = V^{(\beta')} \quad \text{ or } \quad V' = V^{(\pi/2)} \quad \text{ or } \quad V' = V^{(\pi/2)} + \cos\beta'[\R^n].
\end{gather}
Note that the nature of $V'$ (though not, a priori, the rotation) is uniquely determined by $\Theta_{V,g}(x)$ being one of $1/2, (1+\cos\beta(x))/2, 1/2 + \cos\beta(x)$.

Observe also that we have $\sigma_{V_I'} = \haus^{n-1} \llcorner \R^{n-1}$ and either $\sigma_{V'_B} = 0$ or $\sigma_{V_B'} = \cos\beta' \haus^{n-1} \llcorner \R^{n-1}$.  So at $\sigma_{V, g}$-a.e. $x$, every tangent measure of $\sigma_{V, g}$ is a multiple and rotation of $\haus^{n-1} \llcorner \R^{n-1}$.  Along with the Ahlfors-regularity \eqref{eqn:cap-first-var-concl2}, it follows from the Marstrand–Mattila rectifiability criterion (\cite[Theorem 5.1]{demasi}) that $\sigma_{V, g}$ is $(n-1)$-rectifiable, and hence (by \eqref{eqn:refined-2}) $\spt \sigma_{V, g}$ is $(n-1)$-rectifiable. 

By standard density estimates (\cite[Chapter 1]{Sim}), for $\haus^{n-1}_g$-a.e. $x \in B_\gamma \cap \{ \beta \neq \pi/2 \} \setminus \del^* S$ we have
\[
\mu_{V'_B}(B_1) \leq \limsup_{r \to 0} \frac{\haus^{n-1}_g(\del^*S \cap B_r(x))}{r^{n-1}} = 0,
\]
which implies $V' = V^{(\pi/2)}$ or $V' = V^{(\pi/2)} + \cos\beta' [\R^n]$.   On the other hand, from De Giorgi's structure theorem, for $\haus^{n-1}_g$-a.e. $x \in B_\gamma \cap \{ \beta \neq \pi/2 \} \cap \del^* S$ we have
\[
\mu_{V'_B}(B_1) = \lim_{r \to 0} \frac{\cos\beta(x) \haus^n_g(S \cap B_r(x))}{r^n} = \cos\beta(x) \omega_n/2,
\]
which implies $V' = V^{(\beta(x))}$.  Of course we have $V' = V^{(\pi/2)}$ or $V' = V^{(\pi/2)} + \cos\beta(x)[\R^n]$ for $\sigma_{V, g}$-a.e. $x \in B_\gamma \cap \{ \beta = \pi/2 \}$.

For $\sigma_{V_I, g}$-a.e. $x \in B_1$ (and hence by \eqref{eqn:refined-2} for $\sigma_{V, g}$-a.e. $x \in B_1$), we have $\eta_{V_I, g}(x) = \lim_{r \to 0} \fint_{\overline{B_r(x)}} \eta_{V_I, g} d\sigma_{V_I, g}$.  Fix an $x$ satisfying both this and being a Lebesgue point for $\Theta_{V, g}|_{\spt \sigma_{V, g}}$ w.r.t. $\sigma_{V, g}$.  Taking the tangent cone as above, we can assume the rescaled measures $\sigma_{V_{I, i}, g_i} \to \lambda$ as Radon measures for some $\lambda \geq \sigma_{V_I'}$.  We then compute for a.e. $R$:
\begin{align}
\lambda(\overline{B_R}) 
&= \lim_{i \to \infty} \sigma_{V_{I, i}, g_i}(\overline{B_R}) \\
&= \lim_{i \to \infty} \int_{\overline{B_R}} \eta_{V_I, g}(x) \cdot \eta_{V_{I, i}, g_i}(y) d\sigma_{V_{I, i}, g_i}(y) \\
&= \int_{\overline{B_R}} \eta_{V_I, g}(x) \cdot \eta_{V_I'}(y) d\sigma_{V_I'}(y) \\
&\leq \sigma_{V_I'}(\overline{B_R}).
\end{align}
Therefore $\sigma_{V_{I, i}, g_i} \to \sigma_{V_I'}$ as Radon measures, and $\eta_{V_I'} \equiv \eta_{V_I, g}(x)$.  In particular, we have $\lim_{r \to 0} r^{-n+1} \sigma_{V_I}(B_r(x)) = \omega_{n-1}$ for $\sigma_V$-a.e. $x \in B_\gamma$.

From our characterization of tangent cone we deduce that $\eta_{V_I, g}(x) = \nu_g$ for $\sigma_{V, g}$-a.e. $x \in B_\gamma \setminus (\{ \beta \neq \pi/2 \} \cap \del^* S)$, and $g(\eta_{V_I, g}(x) , \nu_g) = \cos\beta(x)$ for $\sigma_{V, g}$-a.e. $x \in B_\gamma \cap \{ \beta \neq \pi/2 \} \cap \del^* S$.  
Finally, using that $\spt \sigma_{V, g}$ is rectifiable we deduce that
\[
\lim_{r \to 0} \frac{\sigma_{V_I, g}(\overline{B_r(x)})}{\haus^{n-1}_g(\spt \sigma_{V} \cap \overline{B_r(x)})} = \lim_{r \to 0} \frac{\omega_{n-1} r^{n-1}}{\haus^{n-1}_g(\spt \sigma_{V, g} \cap \overline{B_r(x)})} \frac{\sigma_{V_I, g}(\overline{B_r(x)})}{\omega_{n-1} r^{n-1}} = 1
\]
for $\sigma_{V, g}$-a.e. $x \in B_\gamma$.  The characterization \eqref{eqn:refined-concl1} now follows by the Radon-Nikodyn theorem.
\end{proof}

\section{Boundary maximum principle}\label{sec:max}
Our goal of this section is to establish the following boundary maximum principle (Theorem \ref{thm:boundary-max}).  Our strategy is to first prove the maximum principle at the level of the tangent cone (Lemma \ref{lem:touching}), by showing that any blow-up cannot be contained in a wedge of too small or large an angle.  To ensure the tangent cone has genuine capillary boundary data (as opposed to secretly being free-boundary), we employ an argument (Lemma \ref{lem:closeness}) which shows that closeness to the model cone $V^{(\theta)}$ propagates to all scales on the boundary \emph{provided} one allows $V^{(\theta)}$ to rotate with the scale.  The argument is very soft and only uses monotonicity and a fairly basic classification of cones (Lemma \ref{lem:tangent-cones}).

In Sections \ref{sec:max}, \ref{sec:harnack}, \ref{sec:decay} \textbf{we will always be working with capillary angles $\in (0, \pi/2)$}, and in particular our capillary varifolds will always have positive density. Otherwise, if $\beta>\pi/2$ locally and $V=V_I + \cos\beta [S]_g$, then we work with the swapped varifold $\tilde V=V_I + \cos(\pi-\beta)[\R^n\setminus S]_g$ instead.

Before stating Theorem \ref{thm:boundary-max}, we recall the notations $G_\theta(u)$ and $\nu_\theta(u,g)$ introduced in Subsection \ref{subsection:slanted_graphs} and we introduce the following (straightforward) terminology: we say that $G_\theta(u)$ lies above (resp. below) $\spt \mu_{V_I}$ in some set $E$ if $\spt \mu_{V_I}\cap E\subset\{x\colon x_{n+1}\le -x_1\cot \theta + u(x')\mbox{ (resp. $\ge$)}\}$ and that it touches $\spt \mu_{V_I}$ at $z$ if $z\in\spt \mu_{V_I}\cap G_\theta(u)$.

\begin{theorem}[Boundary maximum principle]\label{thm:boundary-max}
Given $\theta \in (0, \pi/2)$, there is a $\delta(n, \theta)$ so that the following holds.  Let $g$ be a $C^1$ metric on $B_1$, $\beta$ be a $C^1$ function on $B_1$, and $V$ a $(\beta, S)$-capillary varifold in $(B_1, g)$ satisfying
\begin{gather}
\D(V, V^{(\theta)}) \leq \delta, \label{eqn:mp-hyp1}\\
\max \{ ||H^{tan}_{V, g}||_{L^\infty(B_1)}, |g - g_{eucl}|_{C^1(B_1)}, |\beta - \theta|_{C^1(B_1)} \} \leq \delta. \label{eqn:mp-hyp3}
\end{gather}
Let $u : B_{1/2}^\theta \to \R$ be a $C^1$ function, and $z\in \spt \mu_{V_I}\cap \{x_1=0\} \cap B_{3/4}$. Denote by $G_\theta(u)$ the slanted graph of $u$ and $\nu_g$ the $g$-outward unit normal vector field along $\{x_1=0\}$.
\begin{enumerate}
	\item If $G_\theta(u)$ lies above $\spt \mu_{V_I}$ in $C_{1/2}^\theta \cap B_{3/4}$ and touches $\spt \mu_{V_I}$ at $z$, then
	\begin{equation}
		g(\nu_\theta(u, g)|_z, \nu_{g}|_z ) \geq \cos\beta(z).
	\end{equation}
	\item If $G_\theta(u)$ lies below $\spt \mu_{V_I}$ in $C_{1/2}^\theta \cap B_{3/4}$ and touches $\spt \mu_{V_I}$ at $z$, then
	\begin{equation}
		g( \nu_\theta(u, g)|_z, \nu_{g}|_z ) \leq \cos\beta(z).
	\end{equation}
\end{enumerate}
\end{theorem}

\begin{lemma}[Classification of tangent cones]\label{lem:tangent-cones}
	Let $\theta \in (0, \pi/2)$ and let $V$ be a conical, stationary $(\theta, S)$-capillary varifold in $\R^{n+1}$ satisfying
	\begin{gather}
	\Theta_V(0) \leq (1+\cos\theta)/2,  \label{eqn:cones-hyp1} \\
	0 < \mu_V(B_1 \cap \ver) < \cos\theta \omega_n, \label{eqn:cones-hyp3}
	\end{gather}
	then $V = q(V^{(\theta)})$ for some $q \in O(\ver)$.
\end{lemma}

\begin{remark}
This Lemma fails without the assumption \eqref{eqn:cones-hyp3}.  For a counterexample, take $S = \emptyset$ and $V$ to be half the Simons' cone. Then \eqref{eqn:cones-hyp1} holds for some $\theta\in(0,\pi/2)$ and $\mu_V(B_1\cap \R^n)=0$.
\end{remark}

\begin{proof}[Proof of Lemma \ref{lem:tangent-cones}]
	We prove this by induction on $n$.  If $n = 1$ then $V$ consists of finitely-many rays extending from the origin and the conclusion follows directly by our assumptions.  Suppose $n > 1$ and that the Theorem holds for $n-1$.
	
	By Theorem \ref{thm:cap-first-var} $S$ is a set of locally-finite perimeter, and $\del^*S \subset \spt \sigma_V$, and (by \eqref{eqn:cones-hyp3}) $\del^* S \neq \emptyset$.  Therefore for every $x \in \del^* S$, after a rotation as necessary the rescaled sets $(S - x)/r$ converge in $L^1_{loc}(\ver)$ to $\{ x_1 = 0, x_{n+1} < 0 \}$.
	Then by Lemma \ref{lem:conv} and our upper density assumption \eqref{eqn:cones-hyp1} any tangent cone $V'$ to $V$ at $x$ is a $(\theta, S')$-capillary varifold which satisfies $V' \llcorner \{ x_1 = 0 \} = \cos\theta \haus^n \llcorner \{x_1 = 0 , x_{n+1} < 0 \}$.  Moreover, if $x \neq 0$, then $V'$ must have a line of translational-invariance, and in particular we can write $V' = [\R]\times V''$ for $V''$ satisfying the same hypotheses \eqref{eqn:cones-hyp1}, \eqref{eqn:cones-hyp3} except with $n-1$ in place of $n$.  By inductive hypothesis $\Theta_{V''}(x) = (1+\cos\theta)/2$.
	
	By upper semicontinuity, we deduce that $\Theta_V \geq (1+\cos\theta)/2$ on $\overline{\del^* S}$, but since $\haus^{n-1}(\del^* S) > 0$ we have by standard splitting that $V$ has $(n-1)$-dimensions of translational-invariance, and hence $V$ has the required form.
\end{proof}

\begin{lemma}[Propagation of closeness to all scales]\label{lem:closeness}
Given $\theta\in(0,\pi/2)$, $\eps > 0$, and $\sigma \in (0, 1)$, there is $\delta(n, \theta, \eps, \sigma) > 0$ so that the following holds.  Let $g$ be a $C^1$ metric on $B_1$, $\beta$ a $C^1$ function on $B_1$, and let $V$ be a $(\beta, S)$-capillary varifold in $(B_1, g)$ supported in $\{ x_1 \leq 0 \}$ satisfying
\begin{gather}
\D(V, V^{(\theta)}) \leq \delta, \label{eqn:propogation.1}\\
\max\{ ||H^{tan}_{V, g}||_{L^\infty(B_1)} , |g - g_{eucl}|_{C^1(B_1)}, |\beta - \theta|_{C^1(B_1)}| \}\leq \delta. \label{eqn:propogation.3}
\end{gather}
Then for all $x \in B_\sigma \cap \spt \mu_{V_I} \cap \{ x_1 = 0 \}$ and $r \in (0, 1-\sigma)$, we have
\begin{gather}\label{eqn:closeness-concl}
\inf_{q \in O(\{ x_1 = 0 \})} \D( (V - x)/r, q(V^{(\theta)})) \leq \eps,
\end{gather}
and the same inequality \eqref{eqn:closeness-concl} holds with $V^{(\beta(x))}$ in place of $V^{(\theta)}$.
\end{lemma}

\begin{proof}
We  prove \eqref{eqn:closeness-concl} for $V^{(\theta)}$. The case for $V^{(\beta(x))}$ follows immediately from \eqref{eqn:propogation.3} and the observation that $\D(V^{(\beta(x))}, V^{(\theta)})\le \eps'$ whenever $\delta(n, \eps', \theta)$ is sufficiently small.  Also we note there is no loss in assuming $\eps(n, \theta)$ is as small as we like.

Suppose the statement fails. Then there exist sequences $\delta_i\to 0$, $r_i\in (0,1-\sigma)$, varifolds $V_i$ in $(B_1, g_i)$ and $y_i\in B_{\sigma}\cap \spt \mu_{V_{i,I}}\cap \{x_1=0\}$, such that the assumptions \eqref{eqn:propogation.1}, \eqref{eqn:propogation.3} hold with $V_i,  g_i, \delta_i$ in place of $V, g, \delta$, while
\[\D((V_i-y_i)/r_i,q(V^{(\theta)}))>\eps\]
for all $q\in O(n)$.  Note that we must have that $r_i\to 0$, since \eqref{eqn:propogation.1} implies that $V\to V^{(\theta)}$ in any fixed ball.

Note also that, by Theorem \ref{thm:cap-first-var} applied in any ball $B_{(1-\sigma)/2}(x), x\in B_\sigma\cap\ver$, it holds $$\spt\sigma_{V_i,g_i} \cap B_\sigma = \spt \mu_{V_{i,I}}\cap\ver \cap B_\sigma$$.
Moreover, \eqref{eqn:propogation.1}, \eqref{eqn:propogation.3} and Lemma \ref{lem:conv} imply that
$\Theta_{V_i}(x, (1-\sigma)/2) \leq (1+\cos\theta)/2 + o(1)$ for all $x \in B_\sigma \cap \spt\sigma_{V_i,g_i}\cap \ver$ as $i \to \infty$, and so by the monotonicity formula \eqref{eqn:mono-weak-concl} and \eqref{eqn:propogation.3} we have
\begin{equation}\label{eqn:propagation_density1}
    \Theta_{V_i}(y_i, r) \leq (1+\cos\theta)/2 + o(1)\mbox{ for all }r \in (0, (1-\sigma)/2). 
\end{equation}

For sufficiently large $i$, without loss of generality we can assume that each $r_i$ is taken so that, for every $r\in(r_i, (1-\sigma)/2)$, we have that
\begin{equation}\label{eq:taking_the_first_ri}
    \D((V_i-y_i)/r, q_{i,r}(V^{(\theta)}))\le \eps
\end{equation}
for some $q_{i,r}\in O(n)$.  Consider the rescaled varifolds $V_i' = (V_i - y_i)/r_i$, and note $H^{tan}_{V_i', g_i'} \to 0$ in $L^\infty_{loc}$, $g_i' \to g_{eucl}$ in $C^1_{loc}$, and $\beta_i \to \theta$ in $C^1_{loc}$.  By Lemma \ref{lem:conv}, up to a subsequence we get convergence $V_i' \to V'$ to some $(\theta, S')$-capillary varifold $V'$ in $\R^{n+1}$.  We also have $V_{i,I}'\to V_I'$ and $V_{i,B}'\to V_B' = \cos\theta[S']$ (for $S'$ a Caccioppoli set in $\{x_1=0\}$) as varifolds, $\sigma_{V_i', g_i'}\to \sigma_{V'}$ as Radon measures, and $0\in \spt \sigma_{V'}$.
Then \eqref{eqn:propagation_density1} implies that the density ratio satisfies 
\begin{equation}\label{eqn:propogation-smallness-1}
	\Theta_{V'}(0,\infty)\le \frac{1+\cos\theta}{2}.
\end{equation}

This, by monotonicity, implies $\Theta_{V'}(x) \le \frac{1+\cos\theta}{2}$ for any $x \in \R^n \cap B_1$.
Passing to further subsequence as necessary, by \eqref{eq:taking_the_first_ri} we can assume that there are Euclidean isometries $q_r\in O(n)$ for all $r\in (1,\infty)$ so that 
\[\D(V'/r, q_r(V^{(\theta)}))\le \eps,\quad \forall r>1.\]
By Lemma \ref{lem:D-to-M} (ensuring $\eps(n, \theta)$ is small) we have that $0<\cH^n(S'\cap B_1)<\omega_n$, and hence $\partial^*S' \cap B_2\ne\emptyset$. Take $x\in \partial^*S'\cap B_2$. From Lemma \ref{lem:conv} and Lemma \ref{lem:tangent-cones}, we conclude that any tangent cone to $V'$ at $x$ must be $q(V^{(\theta)})$ for some Euclidean isometry $q$. Particularly, we have that 
\begin{equation}\label{eqn:propogation-smallness-2}
	\theta_{V'}(x)=\frac{1+\cos\theta}{2}
\end{equation}

By the standard monotonicity formula,  \eqref{eqn:propogation-smallness-1} and \eqref{eqn:propogation-smallness-2} together imply that $V'= x+q(V^{(\theta)})$ for some point $x\in \{x_1=0\}$ and $q\in O(n)$. Since $0\in \spt \sigma_{V'}$, we must have that $V'=q(V^{(\theta)})$.  We therefore must have
\[\D((V_i-y_i)/r, q(V^{(\theta)}))<\eps\quad \forall r\in (r_i/2,2r_i),\]
for all $i \gg 1$, contradicting our choice of $r_i$.
\end{proof}

To proceed, for $\theta\in (0,\frac{\pi}{2})$, we define $\Omega^\theta=\{x_1\cos\theta+x_{n+1}\sin\theta<0, x_1 < 0\}$. In the next Lemma, we work with conical free-boundary stationary varifolds in the Euclidean half space.

\begin{lemma}[Boundary behavior at touching points]\label{lem:touching}
Given $\theta_0 > 0$, there is an $\eps(n, \theta_0)$ with the following property.
Let $\theta\in(\theta_0,\pi/2-\theta_0)$ and let $V$ be a conical, stationary $(\theta, S)$-capillary varifold in $\R^{n+1}$ satisfying
\begin{gather}
\inf_{q \in O(\{ x_1 = 0 \})} \D(V, q(V^{(\theta)})) \leq \eps\label{eq:touch_hyp}
\end{gather}
\begin{enumerate}
    \item If there exists $\theta'$ such that
    \begin{equation}\label{eq:touch_hyp_fromabove}
        \spt\mu_V\subset\overline{\Omega^{\theta'}},
    \end{equation}
    then $\theta'\le\theta$;
    \item If there exists $\theta''$ such that
    \begin{equation}\label{eq:touch_hyp_frombelow}
    \spt\mu_{V_I}\subset\R^{n+1}\setminus\Omega^{\theta''},    
    \end{equation}
    then $\theta''\ge\theta$.
\end{enumerate}
\end{lemma}

\begin{proof}
    Throughout the proof, $\Psi(\eps)$ denotes some non-decreasing function (which may change from line to line), depending on $n$ and $\theta$, such that $\Psi(0)=0$.
    
	First note that since $V, V^{(\theta)}$ are conical, \eqref{eq:touch_hyp} implies $\theta_V(0) \leq (1+\cos\theta)/2 + \Psi(\eps)$.
    By contradiction, let us suppose that $V$ satisfies \eqref{eq:touch_hyp_fromabove} for some $\theta'>\theta$ or \eqref{eq:touch_hyp_frombelow} for some $\theta''<\theta$.
    
	We argue by induction.
    Assume $n = 1$.  Then any conical, stationary $(\theta,B)$-capillary varifold $V$ is a union of rays extending from the origin and, by \eqref{eq:touch_hyp}, $\Theta_V(0,1)\le(1+\cos\theta)/2+\Psi(\eps)$. Therefore $V$ must be one of: $0, \cos\theta [\ver], V^{(\pi/2)}, V^{(\theta)}$ or $q(V^{(\theta)})$ for $q(x_1, x_2) = (x_1, -x_2)$.  However all of these possibilities are precluded by \eqref{eq:touch_hyp}, Lemma \ref{lem:D-to-M} (ensuring $\eps(\theta)$ is sufficiently small) and either \eqref{eq:touch_hyp_fromabove} if $\theta'>\theta$ or \eqref{eq:touch_hyp_frombelow} if $\theta''<\theta$.
	
	Suppose now by inductive hypothesis the Theorem holds for $n-1$ in place of $n$.  By Lemma \ref{lem:closeness}, we have $\inf_{q \in O(n)} \D( (V - y)/r, q(V^{(\theta)})) \leq \Psi(\eps)$ for all $y \in \spt \sigma_V$ and all $r > 0$.  If we pick $y \in \spt\sigma_V \cap \{ x_1 = x_{n+1} = 0 \} \setminus \{0\}$, then any tangent cone $V'$ to $V$ at $y$ splits off a line $V' = [\R] \times V''$ with $V''$ being a $(\theta, B'')$-capillary varifold satisfying
	\begin{gather*}
	\inf_{q \in O(n)} \D(V'', q(V^{(\theta)})) \leq \Psi(\eps),\\
	\spt \mu_{V''} \subset \overline{\Omega^{\theta'}} \quad\text{or}\quad\spt \mu_{V''_I} \subset \R^{n+1} \setminus \Omega^{\theta''}.
	\end{gather*}
	By inductive hypothesis if we ensure $\eps(n, \theta)$ is sufficiently small then $V''$ cannot exist, and therefore we must have $\spt\sigma_V \cap \{ x_1 = x_{n+1} = 0\} \subset \{0\}$.  By Theorem \ref{thm:cap-first-var} we deduce
	\begin{gather}\label{eqn:touch-1}
	\spt \mu_{V_I} \cap \{ x_1 = x_{n+1} = 0 \} \subset \{ 0 \}
	\end{gather}
	also.

    By increasing $\theta'$ (resp. decreasing $\theta''$) we can assume that $\theta'$ is the largest angle for which $\spt \mu_V \subset \overline{\Omega^{\theta'}}$ (resp. $\theta''$ is the smallest angle for which $\spt \mu_{V_I} \subset \R^{n+1}\setminus \Omega^{\theta''}$).
    Assume that $V$ satisfies \eqref{eq:touch_hyp_fromabove} for some $\theta'>\theta$: for the remainder of the proof, the argument in the other case is identical.
    Since $\spt\mu_{V_I}$ is a closed cone, there exists $y\in\spt\mu_{V_I}\cap P_{\theta'}^-\cap \partial B_1$. However, by \eqref{eqn:touch-1}, it must hold that $y_1<0$. 
    Therefore near $y$, $V_I$ is stationary and, since $V$ is conical, $\spt \mu_{V_I}$ touches $P_{\theta'}^-$ from one side.  By the strong maximum principle \cite{solomon-white} we deduce $P_{\theta'}^- \subset \spt \mu_{V_I}$.  But this means $\spt \mu_{V_I} \cap \{ x_1 = x_{n+1} = 0 \} = \{ x_1 = x_{n+1} = 0\}$, contradicting \eqref{eqn:touch-1}.  This proves the Lemma with $n$ in place of $n-1$, and so by induction the Lemma holds for all $n$.
\end{proof}

We are now ready to establish the boundary maximum principle.

\begin{proof}[Proof of Theorem \ref{thm:boundary-max}]
Denote by $\theta'$ the $g$-angle between $G_\theta(u)$ and $\{x_1=0\}$ at $z$, i.e. so $\cos\theta' = g( \nu_{\theta}(u,g)|_z, \nu_g|_z )$. Write $\Omega_\theta(u) = \{(x_1, \ldots, x_{n+1}) \in C^\theta_{1/2} \cap B_{3/4} : x_{n+1} < -x_1\cot\theta + u(x')\}$ for the subgraph of $u$ in $B_{3/4}$, and note that up to rotation the tangent cone of $\Omega_\theta(u) \cap \{ x_1 \leq 0 \}$ at $z$ is $\Omega^{\theta'} = \{x_1\cos\theta' + x_{n+1}\sin\theta'<0,x_1 \leq0 \}$.

Let us first consider Case 1.  We claim that
\begin{equation}\label{eqn:bm-1}
\spt \mu_V \cap C^\theta_{1/2} \cap B_{3/4} \subset \overline{\Omega_\theta(u)}
\end{equation}
if $\delta(n, \theta)$ is chosen sufficiently small.  To see this, first note that for $\delta(n, \theta)$ small \eqref{eqn:mp-hyp1} implies $\Theta_V(x, 1/8) \leq 3/4 + \cos\theta/4$ for all $x \in B_{3/4} \cap \ver$.  Theorem \ref{thm:cap-first-var} then gives $\mu_{V_B} \llcorner B_{3/4} = \cos\beta[S]_g$ for some set $S \subset \ver$ satisfying
\begin{equation}\label{eqn:bm-2}
\del S \cap B_{3/4} \subset \spt \sigma_{V,g} \cap B_{3/4} = \spt \mu_{V_I} \cap \ver \cap B_{3/4}.
\end{equation}
From Lemma \ref{lem:conv}, after shrinking $\delta(n, \theta)$ as necessary, we have
\[
\spt \mu_{V_I} \cap B_{3/4} \subset B_{\cos\theta/100}(P_\theta),
\]
and (therefore, in combination with \eqref{eqn:bm-2}), we have
\begin{equation}\label{eqn:bm-3}
\{ x_1 = 0, x_{n+1} < -1/100\} \subset S \subset \{ x_1 = 0, x_{n+1} < 1/100\} \quad \text{ in } B_{3/4}.
\end{equation}
\eqref{eqn:bm-2}, \eqref{eqn:bm-3} together imply that for any $(x', x_{n+1}) \in \spt \mu_{V_B} \cap C^\theta_{1/2} \cap B_{3/4}$, there is a $(x', y_{n+1}) \in \spt \mu_{V_I} \cap \ver \cap C^\theta_{1/2} \cap B_{3/4}$ with $x_{n+1} \leq y_{n+1}$.  In other words, $\spt \mu_{V_B}$ lies below $\spt \mu_{V_I}$ in $C^\theta_{1/2} \cap B_{3/4}$, which in turn lies below $G_\theta(u)$, giving us our claim.

Now from Lemma \ref{lem:closeness} (provided $\delta(n, \theta)$ is small) any tangent cone $V'$ to $V$ at $z$ is non-zero (since $z \in \spt V_I \cap \ver$) and satisfies
\[
\inf_{q \in O(n)} \D(V', q(V^{(\beta(z))}))<\eps_0,
\]
where $\eps_0$ is the constant from Lemma \ref{lem:touching}.  On the other hand, from \eqref{eqn:bm-1}, $V'$ must be supported in $\overline{\Omega^{\theta'}}$, thus by Lemma \ref{lem:touching} we must have $\theta' \leq \beta(z)$.  This proves Case 1.

Case 2 is similar but easier.  In this case we clearly have $\spt \mu_{V_I} \cap C^\theta_{1/2} \cap B_{3/4} \subset C^\theta_{1/2} \cap B_{3/4} \setminus \Omega_\theta(u)$, and therefore any tangent cone $V'$ to $V$ at $z$ satisfies $\spt \mu_{V_I'} \subset \R^{n+1} \setminus \Omega_{\theta'}$.  Lemma \ref{lem:touching} implies $\theta' \geq \beta(z)$, proving Case 2.
\end{proof}

\section{Partial harnack}\label{sec:harnack}

In the next two sections we will make the following running assumptions on a varifold $V$.
\begin{ass}\label{ass:main}
We assume $\theta_0 \in (0, \pi/2)$, $g$ is a $C^1$ metric in $B_1$, $\beta$ a $C^1$ function on $B_1$, and $V$ is $(\beta, S)$-capillary varifold in $(B_1, g)$ satisfying
\begin{gather}
0 \in \spt {V_I}, \quad \Theta_V(0, 1) \leq 3/4 + \cos\theta_0/4 < 1,
\end{gather}
\end{ass}

Our main result in this section is Theorem \ref{thm:harnack}, the partial Harnack inequality. This will tell us that inhomogenous blow-ups of a sequence of varifolds $V_i \to V^{(\theta)}$ will converge in support to some $C^\alpha$ function.  The proof relies on the strategy of \cite{desilva:fbreg} (c.f. \cite{Vel19}), originally implemented to study the regularity of the one-phase Bernoulli problem. We first establish the decay of oscillation.

\begin{theorem}\label{thm:harnack-step}
There is a $\gamma(n, \theta_0) \in (0, 1/100)$ so that the following holds.  If $V$ satisfies Assumptions \ref{ass:main} and additionally
\begin{gather}
\max\{ \osc_{P_{\beta(0)}}(V_I, B_1), |\beta(0) - \theta_0|  \} \leq \gamma, \label{eqn:harnack-step-hyp1} \\
 \max\{||H^{tan}_{V, g}||_{L^\infty(B_1)}, |D\beta|_{C^0(B_1)}, |g - g_{eucl}|_{C^1(B_1)}\}  \leq \gamma \osc_{P_{\beta(0)}}(V_I, B_1), \label{eqn:harnack-step-hyp2}
\end{gather}
then
\begin{equation}\label{eqn:harnack-step-concl}
\osc_{P_{\beta(0)}}(V_I, B_\gamma) \leq (1-\gamma) \osc_{P_{\beta(0)}}(V_I, B_1).
\end{equation}
\end{theorem}

\begin{proof}
We first introduce some notation.  Let $\theta = \beta(0)$.  Given $x' \in \hor$, define
\[
|x'|_\theta := \sqrt{ (\sin \theta)^{-2} x_1^2 + x_2^2 + \ldots x_n^2}.
\]
Observe that $B^\theta_r(x_0') = \{ x' \in P_{\pi/2} : |x' - x_0'|_\theta < r \}$, and by a direct computation:
\[
\cL_\theta(|x'|^{-n}_\theta) = 2n|x'|^{-n-2}_\theta,
\]
where recall that $\cL_\theta(u) = \Delta u - \cos^2 \theta D_1^2 u$.

Fix $x_0' = -(\sin\theta/4) e_1$, and write $\eta = \osc_{P_\theta}(V_I, B_1)$.  After translating $V$ there is no loss in assuming that for all $x \in \spt V_I \cap B_{19/20}$ we have $|x \cdot \nu_\theta| \leq \eta$ (recall $\nu_\theta = \cos\theta e_1  +\sin\theta e_{n+1}$), and hence $|x_{n+1} + \cot\theta x_1| \leq \eta/\sin\theta$.  Provided $\eta, \gamma$ are sufficiently small (depending only on $n, \theta_0$), Allard's theorem implies that
\[
\spt V_I \cap C^\theta_{1/8}(x_0') \cap B_{9/10} = G_\theta(w)
\]
for some $C^{1, 1/2}$ function $w : B^\theta_{1/8}(x_0') \to \R$ satisfying
\begin{equation}\label{eqn:harnack-step-1}
|w|_{C^{1,1/2}} \leq c(n, \theta_0) \eta, \quad |w| \leq \eta/\sin\theta.
\end{equation}
Suppose $w(x_0') \leq 0$.  Then by \eqref{eqn:harnack-step-hyp1}, \eqref{eqn:harnack-step-1} there is a radius $r_0(n, \theta_0)$ (which we now fix) so that
\begin{equation}\label{eqn:harnack-step-1.1}
w \leq \frac{\eta}{2\sin\theta} \text{ on } B^\theta_{2r_0}(x_0').
\end{equation}

For $t \in \R$, define $u_{t} : \hor \to \R$ by
\[
u_{t}(x') = \left\{ \begin{array}{l l} t & 1/2 \leq |x' - x_0'|_\theta \\ t - \eta \bar c ( |x' - x_0'|^{-n}_\theta - 2^n) & r_0 \leq |x' - x_0'|_\theta \leq 1/2 \\ t - \eta \bar c ( r_0^{-n} - 2^n) & |x' - x_0'|_\theta \leq r_0 \end{array} \right. ,
\]
for $\bar c = \frac{1}{4\sin\theta} (r_0^{-n} - 2^n)^{-1}$, so that $u_t$ is continuous and
\begin{equation}\label{eqn:harnack-step-2}
u_t = t \text{ outside } B^\theta_{1/2}(x_0') \quad \text{and}\quad u_t = t - \frac{\eta}{4\sin\theta} \text{ inside } B^\theta_{r_0}(x_0').
\end{equation}
We will use $G_\theta(u_{t})$ as barrier hypersurfaces for $V_I$.

For $x' \in B^\theta_{1/2}(x'_0) \setminus B^\theta_{r_0}(x'_0)$, we compute 
\begin{equation}\label{eqn:harnack-step-3}
H_\theta(u_{t}, g)|_{x'} \leq -2n\eta \bar c |x' - x'_0|_\theta^{-n-2} + c(n, \theta_0)(\eta+\gamma) \leq -c(n,\theta_0)\eta
\end{equation}
provided $\gamma(n, \theta_0)$ is sufficiently small.  For $x' \in B^\theta_{1/2}(x'_0) \cap \{ x_1 = 0 \}$ we compute
\begin{align}
\nu_\theta(u_t, g)|_{x'} \cdot e_1 &\leq \cos\theta + \frac{\eta \bar c n}{4\sin\theta} |x' - x_0'|_\theta^{-n-2} (\cos^2 \theta \sin\theta-1) + c(n, \theta_0) \gamma \eta \nonumber \\
&\leq \cos\theta - \eta \bar c \sin\theta \label{eqn:harnack-step-4}
\end{align}
again provided $\gamma(n, \theta_0)$ is small.

We claim that $\spt V_I \cap B_{9/10}$ must lie below $G_\theta(u_{t_0})$ with $t_0 = \eta/\sin\theta$.  In other words we claim that for every $(x', x_{n+1}) \in \spt V_I \cap B_{9/10}$ we have $x_{n+1} \leq -\cot\theta x_1 + u_{t_0}(x')$.  To see this, we first observe by \eqref{eqn:harnack-step-1}, \eqref{eqn:harnack-step-2} that $\spt V_I \cap B_{9/10}$ lies below $G_\theta(u_{t_1})$ with $t_1 = 2\eta/\sin\theta$, and we second observe that both $\spt V_I \cap B_{9/10} \cap C^\theta_{2r_0}(x_0')$ (by \eqref{eqn:harnack-step-1.1}) and $\spt V_I \cap B_{9/10} \setminus C^\theta_{1/2}(x_0')$ lie below $G_\theta(u_{t_0})$.

Now consider the smallest $t_* \in [t_0, t_1]$ for which $\spt V_I \cap B_{9/10}$ lies below $G_\theta(u_{t_*})$, and suppose towards a contradiction that $t_* > t_0$.  Then by the previous paragraph and \eqref{eqn:harnack-step-1} there must be a point $x_* \in \spt V_I \cap G_\theta(u_{t_*}) \cap \overline{C^\theta_{1/2}(x'_0)} \setminus C^\theta_{2r_0}(x'_0)$.  By the strong maximum principle, \eqref{eqn:harnack-step-3}, and our assumption \eqref{eqn:harnack-step-hyp2} (provided $\gamma(n, \theta_0)$ is small), $x_* \not\in \{ x_1 < 0 \}$.  On the other hand, taking $\gamma(n, \theta_0)$ small, from Lemma \ref{lem:D-vs-osc} and Theorem \ref{thm:boundary-max} (applied to the function $x' \mapsto t_* - \eta \bar c (|x' - x'_0|_\theta^{-n} - 2^n)$) if $x_* \in \{ x_1 = 0 \}$ then we would have
\[
\nu_\theta(u_{t_*}, g)|_{x'} \cdot e_1 \geq \cos \beta(x_*) \geq \cos\theta - \gamma \eta
\]
which contradicts \eqref{eqn:harnack-step-4} for $\gamma(n, \theta_0)$ small.  So $x_*$ cannot exist, and we must have $t_* = t_0$, proving our claim.

To conclude, we observe that for $x' \in B^\theta_{1/8}(0) \subset B^\theta_{1/2}(x'_0)$ we have $|x' - x'_0|^{-n}_\theta - 2^n \geq 2^n/3$, and so 
\[
u_{t_0} \leq t_0 - \bar c \eta/2 \leq (1-\gamma) \eta/\sin\theta \quad \text{ on }\quad  B^\theta_{1/8}(0).
\]
Ensuring $\eta(n, \theta_0)$ is small, we deduce that $\osc_{P_\theta}(V_I, B_{1/16}) \leq (1-\gamma) \osc_{P_\theta}(V_I, B_1)$.  This proves the Theorem in the case when $w(x'_0) \leq 0$.  If $w(x'_0) \geq 0$ then we can use an essentially verbatim argument, replacing $-\eta$ with $\eta$ in the definition of $u_t$ and touching $\spt V_I \cap B_{9/10}$ from below by $G_\theta(u_t)$.
\end{proof}
\begin{theorem}[Partial Harnack]\label{thm:harnack}
Given $\tau > 0$, there are $c(n, \theta_0)$, $\eps(n, \theta_0, \tau)$, $\alpha(n, \theta_0)$ so that the following holds.  If $V$ satisfies Assumption \ref{ass:main} and additionally
\begin{gather}
E:= \max \{ \osc_{P_{\beta(0)}}(V_I, B_1) , ||H^{tan}_{V,g}||_{L^\infty(B_1)}, |D\beta|_{C^0(B_1)}, |g - g_{eucl}|_{C^1(B_1)} \} \leq \eps, \\
|\beta(0) - \theta_0| \leq \eps,
\end{gather}
then for all $x \in \spt V_I \cap B_{1/2}$ and all $\tau \leq r \leq 1/2$, we have
\begin{align}\label{eqn:harnack-concl}
\osc_{P_{\beta(0)}}(V_I, B_r(x)) \leq c r^\alpha E.
\end{align}
\end{theorem}

\begin{proof}
We use Theorem \ref{thm:harnack-step} to get H\"older decay centered at boundary points, and Allard's theorem to get H\"older decay at interior points up until the radius is comparable to $|x_1|$.  We combine the two decays to get decay for all radii $\geq \tau$.  There is no loss in assuming $\tau \leq 1/100$.  For ease of notation write $\theta = \beta(0)$.

First let $\gamma$ be the constant from Theorem \ref{thm:harnack-step}, and choose $K\ge 4/\gamma$.  For $x \in \spt V_I \cap \{ x_1 = 0 \} \cap B_{1/2} \equiv \spt \sigma_{V,g} \cap B_{1/2}$, define the function
\[
F(r) = \osc_{P_{\beta(x)}}(V_I, B_r(x)) + K E r.
\]
We claim that
\begin{equation}\label{eqn:harnack-1}
F(\gamma r) \leq (1-\gamma) F(r) \quad \text{for all $r$ such that} \quad \osc_{P_{\beta(x)}}(V_I, B_1)/\gamma \leq r \leq 1/4
\end{equation}
To see this, note that for such $r$ we trivially have $\osc_{P_{\beta(x)}}(V_I, B_r(x))/r \leq \gamma$.  If, in addition, we have
\begin{align}
&\max\{ r ||H^{tan}_{V,g}||_{L^\infty(B_r(x))} , r |D\beta|_{C^0(B_r(x))}, |g - g_{eucl}|_{C^0(B_r(x))}, r |Dg|_{C^0(B_r(x))} \} \nonumber \\
&\quad \leq \gamma \osc_{P_{\beta(x)}}(V_I, B_r(x))/r,  \label{eqn:harnack-1.5}
\end{align}
then we can apply Theorem \ref{thm:harnack-step} at scale $B_r(x)$ to get
\[
F(\gamma r) \leq (1-\gamma) \osc_{P_{\beta(x)}}(V_I, B_r(x)) + K E \gamma r \leq (1-\gamma) F(r)
\]
On the other hand, if \eqref{eqn:harnack-1.5} fails, then we can compute instead
\begin{align*}
F(\gamma r) &\leq KE r \left( \frac{1}{K\gamma} + \gamma\right) \leq (1-\gamma) F(r)
\end{align*}
This proves \eqref{eqn:harnack-1}.
By iterating \eqref{eqn:harnack-1}, and ensuring $\eps(n, \theta_0, \tau)$ is small, we deduce that with the choice $\alpha = \log(1-\gamma)/\log(\gamma) \in (0, 1)$ it holds
\begin{equation}\label{eqn:harnack-1.1}
\osc_{P_\theta}(V_I, B_r(x)) \leq \osc_{P_{\beta(x)}}(V_I, B_r(x)) + c(n) r |D\beta|_{C^0(B_1)} \leq c(n, \theta_0) r^\alpha E
\end{equation}
for all $x \in \spt V_I \cap \{ x_1 = 0 \} \cap B_{1/2}$ and all $\tau \leq r \leq 1/2$, provided we ensure $\eps(n, \tau)$ is small enough to guarantee
\[
\osc_{\beta(x)}(V_I, B_1) \leq \osc_{\beta(0)}(V_I, B_1) + c(n) |D\beta|_{C^0(B_1)} \leq c(n) \eps \leq \tau.
\]

Second, since $V \to V^{(\theta_0)}$ as $\eps \to 0$ (by Lemma \ref{lem:D-vs-osc}) and $\theta \to \theta_0$, then by Allard's theorem and provided $\eps(n, \theta_0, \tau)$ is sufficiently small we can write
\[
\spt V_I \cap C^\theta_{9/10} \cap \{ x_1 < -\tau \} = G_\theta(w),
\]
for $w : B^\theta_{9/10} \cap \{ x_1 < -\tau \} \to \R$ a $C^1$ function satisfying the following estimate: for any $x = (x', x_{n+1}) \in G_\theta(w)$, we have
\begin{multline}
\sup_{B^\theta_{|x_1|/4}(x')} \left( |x_1|^{-1} |w| + |Dw| \right)\\
\leq c \left( \frac{\osc_{P_\theta}(V_I, B_{|x_1|}(x))}{|x_1|} + |x_1| ||H_{V,g}^{tan}||_{L^\infty(B_{|x_1|}(x))} + |x_1| |Dg|_{C^0(B_{|x_1|}(x))} \right)
\end{multline}
for some $c(n,\theta_0)$.
Therefore (using that $\alpha \in (0, 1)$) we deduce that
\begin{equation}\label{eqn:harnack-2}
\osc_{P_\theta}(V_I, B_r(x)) \leq c\left(\frac{r}{|x_1|}\right)^\alpha \Big(\osc_{P_\theta}(V_I, B_{|x_1|}(x)) + E|x_1|^2 \Big)
\end{equation}
for all $x \in \spt V_I \cap B_{1/2} \cap \{ x_1 < -\tau\}$ and all $0 < r < |x_1|$, where $c=c(n,\theta_0)$.

Thirdly, again ensuring $\eps(n, \theta_0, \tau)$ is small, by Lemma \ref{lem:conv} we can assume that 
\begin{equation}\label{eqn:harnack-3}
d_H( \spt V_I \cap \{ x_1 = 0 \} \cap B_{9/10}, \{ x_1 = x_{n+1} = 0 \} \cap B_{9/10} ) \leq \tau.
\end{equation}
Take $x \in \spt V_I \cap B_{1/2}$.
From \eqref{eqn:harnack-3} and the assumption $\osc_{P_{\beta(0)}}(V_I,B_1)\le\eps\le\tau$, we can choose $y \in \spt V_I \cap \{ x_1 = 0 \} \cap B_{1/2}$ so that $|x - y| \leq |x_1|/\sin\beta(0) + 2\tau\leq c(\theta_0)|x_1|+2\tau$, where the latter inequality holds true provided $\eps$ is small enough depending on $\theta_0$, since $|\beta-\theta_0|_{C^0}\le\eps$.

For $r \geq \max(\tau, |x_1|)$, we can use \eqref{eqn:harnack-1.1} and our choice of $y$ to estimate
\[
\osc_{P_\theta}(V_I, B_r(x)) \leq \osc_{P_{\theta}}(V_I, B_{r+|x - y|}(y)) \leq c (r+|x - y|)^\alpha E \leq c(n, \theta_0) r^\alpha E,
\]
while for $\tau \leq r \leq |x_1|$, it holds $|x-y|\le c(\theta_0)|x_1|$, thus by \eqref{eqn:harnack-2}, \eqref{eqn:harnack-1.1} we get
\begin{align*}
\osc_{P_\theta}(V_I, B_r(x)) &\leq c (r/|x_1|)^\alpha \osc_{P_\theta}(V_I, B_{|x_1|+|x - y|}(y)) + c r^\alpha E \\
&\leq c (r/|x_1|)^\alpha |x_1|^\alpha E + c r^\alpha E \leq c(n, \theta_0) r^\alpha E.
\end{align*}
This completes the proof of Theorem \ref{thm:harnack}.
\end{proof}

\section{Decay estimate}\label{sec:decay}

We continue to make the same Assumptions \ref{ass:main} on $V$ as in Section \ref{sec:harnack}.
The following (Theorem \ref{thm:decay-r}) is the main regularity/excess decay estimate, and it follows from an iteration of Theorem \ref{thm:decay}, proved below. 
\begin{theorem}[Excess decay estimate]\label{thm:decay-r}
There are $\eps(n, \theta_0)$, $c(n, \theta_0)$ positive so that if $V$ satisfies Assumptions \ref{ass:main} and additionally satisfies
\begin{equation}\label{eqn:decay-r-hyp}
E := \max \{ \osc_{P_{\theta_0}}(V_I, B_1), ||H^{tan}_{V, g}||_{L^\infty(B_1)}, |D\beta|_{C^0(B_1)}, |g - g_{eucl}|_{C^1(B_1)} \} \leq \eps, 
\end{equation}
then for all $x \in \spt V_I \cap B_{1/8}$, we can find an $n$-dimensional plane $Q_x$ with $|Q_x-P_{\theta_0}|\le cE$ so that for $0 < r < 1/2$ we have
\begin{equation}\label{eqn:decay-r-concl}
\osc_{Q_x}(V_I, B_r(x)) \leq c r^{3/2} E.
\end{equation}
\end{theorem}

\begin{remark}
Of course there is nothing special in the power $3/2$, any number in $(1,2)$ would equally work.
\end{remark}

\begin{proof}[Proof of Theorem \ref{thm:decay-r}]
For ease of notation write $\theta = \theta_0$.  By Lemma \ref{lem:D-vs-osc}, as $\eps \to 0$ we have $V \to V^{(\theta)}$ in $B_1$ and $\beta \to \theta$ in $C^0(B_1)$, and so there is no loss in assuming that
\begin{gather}\label{eqn:decay-r-1}
\Theta_V(x, 1/2) \leq (1 + \cos\theta)/2 + \eps' \quad \forall x \in \spt V_I \cap \{ x_1 = 0 \} \cap B_{1/4}, \\
|\beta - \theta|_{C^0(B_1)} \leq \eps'
\end{gather}
for any particular $\eps'(n, \theta)$ which we shall fix later.

\vspace{3mm}

\emph{Step 1: Decay towards Euclidean metric on the boundary.}  Fix $x \in \spt V_I \cap \{ x_1 = 0 \} \cap B_{1/4}$, and suppose $g(x) = g_{eucl}$.  Choose $\eta(n, \theta)$ so that $c\eta^2 \leq \eta^{3/2}$ where $c(n, \theta)$ is the constant from Theorem \ref{thm:decay}, and let $\eps_1(n, \theta, \eta)$ be the corresponding $\eps$.  For a $n$-plane $Q$, $0 < r < 1$, define
\begin{align*}
E(Q, x, r) &= \max \{ r^{-1} \osc_{Q}(V_I, B_r(x)),  \eps_1^{-1} r ||H_{V, g}||_{L^\infty(B_r(x))}, \\
&\quad\quad \eps_1^{-1} r |D\beta|_{C^0(B_r(x))}, \eps_1^{-1} r |Dg|_{C^0(B_r(x))} \},
\end{align*}
We claim we can find a $Q_x$ so that, for all $0<r<1/4$,
\begin{equation}\label{eqn:decay-r-2}
|Q_x - P_{\beta(x)}| \leq c E(P_{\beta(x)}, x, 1/4), \quad E(Q_x, x, r) \leq c r^{1/2} E(P_{\beta(x)}, x, 1/4),
\end{equation}
for $c = c(n, \theta)$.

Let us prove our claim.  First note that by \eqref{eqn:decay-r-1}, monotonicity \eqref{eqn:mono-weak-concl}, and our assumptions on $g$, provided $\eps'(n, \theta)$ is chosen sufficiently small we have
\begin{gather}
\Theta_V(x, r) \leq 3/4 + \cos \theta/4 \quad \forall 0 < r < 1/2, \\
|\beta(x) - \theta_0| \leq \eps_1.
\end{gather}
So, for any particular $0 < r < 1/2$, if we define
\begin{gather}
V_{x, r} = (T_{x, r})_\sharp V, \quad g_{x, r}(y) = g(x + ry), \quad \beta_{x, r}(y) = \beta(x + ry), \\
\text{ for } T_{x, r}(y) = (y - x)/r : (B_r(x), g) \to (B_1, g_{x, r}),
\end{gather}
(note that $g_{x, r}$ is chosen so that $JT_{x, r} = r^{-n}$ on $V$) then $V_{x, r}$ as a varifold in $(B_1, g_{x, r})$ will also satisfy Assumption \ref{ass:main}, with $g_{x, r}$, $\beta_{x, r}$ in place of $g, \beta$.  Moreover, since $g_{x, r}(0) = g_{eucl}$, we have
\[
|g_{x, r} - g_{eucl}|_{C^0(B_1)} \leq c(n) r |Dg|_{C^0(B_r(x))}.
\]

For $i \in \{ 0, 1, 2, \ldots \}$, set $r_i = \eta^i/4$, and define $Q_0 = P_{\beta(x)}$, $E_0 = E(Q_0, x, 1/4) \equiv E(Q_0, x, r_0) \leq c(n) E$.  In light of the previous discussion, ensuring $\eps(\eps_1, n, \theta)$ is sufficiently small, we can apply Theorem \ref{thm:decay} to $V_{x, r_i}$, $g_{x, r_i}, \beta_{x, r_i}$, and inductively define $Q_i$ so that
\begin{gather}
\osc_{Q_{i+1}}(V_I, B_r(x)) \leq \eta^{3/2} E_i, \quad |Q_{i+1} - Q_i| \leq c(n, \theta) E_i
\end{gather}
where $E_i = E(Q_i, x, r_i)$.  In particular, we have
\[
E_i \leq \eta^{i/2} E_0,
\]
and so $Q_i \to Q_x$.  Given $r \in (0, 1/4)$, choose the greatest $i$ so that $r \leq r_i$, and then we can estimate
\[
r^{-1} \osc_{Q_x}(V_I, B_r(x)) \leq c E_i  + c|Q_x - Q_i| \leq c r_i^{1/2} E_0 \leq c(n, \theta) r^{1/2} E_0,
\]
and $|Q_x - P_{\beta(x)}| \leq c(n, \theta) E_0$.  Our required inequality \eqref{eqn:decay-r-2} follows from the above and the inequality $|D\beta| \leq E$.

\vspace{3mm}

\emph{Step 2: Decay towards non-Euclidean metric on the boundary.}  If $x \in \spt V_I \cap \{ x_1 = 0 \} \cap B_{1/4}$ but $g(x) \neq g_{eucl}$, then we can choose an affine map $L$ satisfying
\begin{equation}
L(x) = x, \quad |DL - Id| \leq c|g(x) - g_{eucl}| \leq c(n)\eps, \quad L( \{x_1 \leq 0 \}) = \{x_1 \leq 0 \},
\end{equation}
so that if $\tilde g = (L^{-1})^* g$ then $\tilde g(x) = g_{eucl}$.  If we let $\tilde V = L_\sharp V$ (where we view $L$ as an isometry $(B_1, g) \to (L(B_1), \tilde g))$ and $\tilde \beta = \beta \circ L^{-1}$, then by \eqref{eqn:decay-r-1}, \eqref{eqn:mono-weak-concl}, \eqref{eqn:osc-A} the translated/scaled $\tilde V_{x, 1/4}$, $\tilde g_{x, 1/4}$, $\tilde \beta_{x, 1/4}$ satisfy Assumption \ref{ass:main}, and hypotheses \eqref{eqn:decay-r-hyp} with $c(n)\eps$ in place of $\eps$.

We can therefore apply Step 1 to deduce the existence of an $n$-plane $\tilde Q_x $ for which
\begin{equation}\label{eqn:decay-r-2.5}
|\tilde Q_x - P_\theta| \leq c \tilde E, \quad \osc_{\tilde Q_x}(\tilde V_I, B_r(x)) \leq c r^{3/2} \tilde E \quad \forall 0 < r < 1/4,
\end{equation}
where $c(n, \theta)$ and 
\begin{equation}\label{eqn:decay-r-3}
\tilde E = \max \{ \osc_{P_{\theta}}(\tilde V_I, B_{1/4}(x)), ||H_{\tilde V, \tilde g}||_{L^\infty(B_{1/4}(x))}, |D\tilde \beta|_{C^0(B_{1/4}(x))}, |D\tilde g|_{C^0(B_{1/4}(x))} \}.
\end{equation}

Define $Q_x=L^{-1}\tilde Q_x$.
Then from \eqref{eqn:osc-A}, \eqref{eqn:decay-r-2.5}, \eqref{eqn:decay-r-3}, ensuring $\eps(n, \theta)$ is small, we deduce for any $0 < r < 1/8$ the inequalities
\[
\osc_{Q_x}(V_I, B_r(x)) \leq 2 \osc_{\tilde Q_x}(\tilde V_I, B_{2r}(x)) \leq 4 c r^{3/2} \tilde E \leq 8 c r^{3/2} E
\]
(for $c$ as in \eqref{eqn:decay-r-2}), and
\[
|Q_x - P_\theta| \leq c(n) |g(x) - g_{eucl}| + |\tilde Q_x - P_\theta| \leq c \tilde E + c(n) E \leq c(n, \theta) E.
\]
This completes Step 2, and implies \eqref{eqn:decay-r-concl} for any $x \in \spt V_I \cap \{ x_1 =  0\} \cap B_{1/4}$.

\vspace{3mm}

\emph{Step 3: Decay away from the boundary.}  We first claim that $\spt V_I \cap \ver \cap B_{1/8}$ is contained in the graph of a $C^{1,1/2}$ function $u :\R^{n-1} \cap B_{1/8} \to \R$ with $|u|_{C^{1,1/2}} \leq c(n, \theta) E$, where $\R^{n-1}=\{x_1=x_{n+1}=0\}$.

To see this, write $\Sigma' = \spt V_I \cap \ver$.  From Step 2, for every $x \in \Sigma' \cap B_{1/4}$ we have an $n$-dimensional plane $Q_x$ satisfying $|Q_x - P_\theta| \leq c(n, \theta)E$ so that
\begin{equation}\label{eqn:decay-r-5}
|\pi_{Q_x}^\perp(y - x)| \leq c E |x - y|^{3/2} \quad \forall y \in \Sigma' \cap B_{1/2}.
\end{equation}
For $\eps(n)$ sufficiently small, we deduce there is a function $u : \pi_{\R^{n-1}}(\Sigma' \cap B_{1/4}) \to \R$ so that
\[
\Sigma' \cap B_{1/4} = \{ (0, x', u(x')) : x' \in \pi_{\R^{n-1}}(\Sigma' \cap B_{1/4}) \},
\]
and with the property that for every $x', y' \in \pi_{\R^{n-1}}(\Sigma' \cap B_{1/4})$, there is a linear function $\ell_{x'} : \R^{n-1} \to \R$ so that
\begin{equation}\label{eqn:decay-r-6}
|u(x')| \leq c(n, \theta) E, \quad |u(y') - u(x') - \ell_{x'}(y' - x')| \leq c(n, \theta) |y' - x'|^{3/2} E.
\end{equation}
On the other hand, for any $x \in \Sigma' \cap B_{1/4}$ we have by \eqref{eqn:decay-r-5}, Lemma \ref{lem:D-vs-osc}, and Lemma \ref{lem:conv} that
\[
\spt (V_{x, r, I}) \cap B_1 \equiv (\spt V_I - x)/r \cap B_1 \to Q_x\cap\{x_1\le0\}
\]
in the Hausdorff distance as $r \to 0$, and hence $\pi_{\R^{n-1}}( (\Sigma' - x)/r) \to \R^{n-1}$ in the local Hausdorff distance as $r \to 0$ also.  We deduce we must have $\pi_{\R^{n-1}}(\Sigma' \cap B_{1/4}) \supset \R^{n-1} \cap B_{1/8}$, and hence by \eqref{eqn:decay-r-6} $u : \R^{n-1} \cap B_{1/8} \to \R$ is $C^{1,1/2}$ with $|u|_{C^{1,1/2}} \leq c(n, \theta)E$.  This proves our first claim.

Take $x \in \spt V_I \cap B_{1/8}$, and from our claim above we can choose $y \in \spt V_I \cap \ver \cap B_{1/8}$ with $(y_2, \ldots, y_n) = (x_2, \ldots, x_n)$.  Note this means that $|x - y| = (x_1^2 + x_{n+1}^2)^{1/2} \leq 1/8$.  From Step 2 we have an $n$-plane $Q_y=q(P_{\theta_y})$ for some $q\in O(n)$ and some $\theta_y$ with $|q_y - Id| + |\theta_y-\theta| \leq c(n, \theta)E$ so that
\begin{equation}\label{eqn:decay-r-7}
\osc_{Q_y}(V_I, B_{2|x - y|}(y)) \leq c(n, \theta)E |x - y|^{3/2}.
\end{equation}
So $x \in B_{c(n, \theta) E |x - y|}(y + Q_y) \subset B_{2 c(n, \theta) E |x - y|}(y + P_\theta)$, and hence ensuring $\eps(n, \theta)$ is sufficiently small we get
\begin{equation}\label{eqn:decay-r-8}
2|x - y|\cos\theta > |x_1| > |x - y| \cos\theta/2 =: R.
\end{equation}
By \eqref{eqn:decay-r-7} and Lemma \ref{lem:D-vs-osc}, as $\eps\to0$ we have $V\to q_y(V^{(\theta_y)})$ in $B_{2|x-y|}(y)$, hence by taking $\eps(n,\theta)$ smaller, if needed, we have the bound
\begin{equation}\label{eqn:decay-r-9}
\Theta_V(x, R) < 3/2.
\end{equation}

From \eqref{eqn:decay-r-7}, \eqref{eqn:decay-r-8}, \eqref{eqn:decay-r-9} we can apply Allard's theorem \ref{thm:allard} in $B_{R}(x)$ to get that
\[
\spt V_I \cap B_{R/2}(x) = \graph_{Q_y}(w) \cap B_{R/2}(x)
\]
for $w : Q_y \cap  B_{R/2}(x) \to \R$ a $C^{1, 1/2}$ function satisfying satisfying the scale-invariant estimates
\begin{align}
&|R^{-1} w(R \cdot)|_{C^{1,1/2}} \nonumber \\
&\leq c(n) \max\{ R^{-1} \osc_{Q_y}(V_I, B_{R}(x)), R ||H_{V, g}||_{L^\infty(B_{R}(x))}, R |Dg|_{C^0(B_{R}(x))} \} \nonumber \\
&\leq c(n, \theta) E R^{1/2}. \label{eqn:decay-r-10}
\end{align}

Let now $Q_x$ be the tangent plane to $\graph_{Q_y}(w)$ at $x$.  Then from \eqref{eqn:decay-r-10} we can assume 
\[
|Q_x - Q_y| \leq c(n, \theta) E R^{1/2}
\]
(which implies $|Q_x - P_\theta| \leq c(n, \theta) E$), and for $0 < r \leq R/4$ we have the estimate
\[
r^{-1} \osc_{Q_x}(V_I, B_r(x)) \leq c(n, \theta) E r^{1/2}.
\]

If $R/4 < r < 1/8$ we can use \eqref{eqn:decay-r-7}, \eqref{eqn:decay-r-10} to estimate instead
\begin{align}
r^{-1} \osc_{Q_x}(V_I, B_r(x)) 
&\leq c E R'^{1/2} + c r^{-1} \osc_{Q_y}(V_I, B_{r + |x - y|}(y)) \\
&\leq c E R'^{1/2} + c E r^{-1} (r + R)^{3/2}  \\
&\leq c(n ,\theta) E r^{1/2}.
\end{align}
And of course \eqref{eqn:decay-r-concl} is trivial if $1/8 \leq r < 1/2$.  This complete the proof of Theorem \ref{thm:decay-r}.
\end{proof}

\begin{theorem}\label{thm:decay}
For any $\eta > 0$, there are $\eps(n, \eta, \theta_0)$, $c(n, \theta_0)$ positive so that if $V$ satisfies Assumptions \ref{ass:main}, and it additionally satisfies
\begin{gather}
E := \max\left\{ \osc_{P_{\beta(0)}}(V_I, B_1), \frac{||H^{tan}_{V, g}||_{L^\infty(B_1)}}{\eps} , \frac{|D\beta|_{C^0(B_1)}}{\eps},  \frac{|g - g_{eucl}|_{C^1(B_1)}}{\eps}  \right\} \leq \eps, \\\label{eqn:decay-hyp} 
|\beta(0) - \theta_0| \leq \eps,
\end{gather}
then we can find a $q \in O(\{ x_1 = 0 \})$ with $|q - Id| \leq c E$ so that 
\begin{gather}\label{eqn:decay-concl}
\osc_{q(P_{\beta(0)})}(V_I, B_{\eta}) \leq c \eta^2 E .
\end{gather}
\end{theorem}

\begin{proof}
For ease of notation write $\theta = \theta_0$.  Suppose, towards a contradiction, the Theorem is false.  Then for any $\eta(n, \theta), c(n, \theta)$ predetermined, there are sequences $\eps_i \to 0$, $g_i$, $\beta_i$, and $V_i$ $n$-rectifiable varifolds in $(B_1, g_i)$ satisfying Assumption \ref{ass:main} and \eqref{eqn:decay-hyp} (with $V_i, E_i, \beta_i, g_i \eps_i$ in place of $V, E, \beta, g, \eps$), but for which \eqref{eqn:decay-concl} fails for all $q \in O(n)$ with $|q - I| \leq c E_i$.

We note that in $B_1$, $V_i \to V^{(\theta)}$ as varifolds by Lemma \ref{lem:D-vs-osc}, and hence by Lemma \ref{lem:conv} $V_{i, I} \to [P^-_{\theta}]$ as varifolds and $\spt V_{i, I} \to P^-_{\theta}$ in the local Hausdorff distance.  Define $\theta_i = \beta_i(0)$, so that $\theta_i \to \theta$ also.  

\vspace{3mm}

\emph{Step 1: Inhomogenous blow-up to graph.}  For each $i$ define the map $F_i : \R^{n+1} \to \R^{n+1}$ by
\[
F_i(x_1, \ldots, x_n, x_{n+1}) = (x_1, \ldots, x_n, E_i^{-1} (x_{n+1} + x_1 \cot\theta_i)),
\]
and then define the closed sets
\[
\Sigma_i = F_i(\spt V_{i, I}) \cap C^\theta_{1/4}(0).
\]
By \eqref{eqn:decay-hyp}, each $\Sigma_i$ is a closed subset of $C^\theta_{1/4}(0) \cap \{ |x_{n+1}| \leq 1/\sin\theta_i \}$, and $0 \in \Sigma_i$, and so after passing to a subsequence we can assume $\Sigma_i \to \Sigma$ in the Hausdorff metric, for some non-empty closed subset $0 \in \Sigma \subset C^\theta_{1/4}(0) \cap \{ |x_{n+1}| \leq 1/\sin\theta \}$.  For $x' \in B^\theta_{1/4} \cap \{ x_1 \leq 0 \}$, define $\textbf{u}(x') = \{ y : (x', y) \in \Sigma \}$.

We first claim that $\textbf{u}(x') \neq \emptyset$.  If we had $\textbf{u}(x') = \emptyset$ for some $x' = \pi_{\hor}(x)$, $x \in P^-_\theta$, then by closedness there would be a radius $r > 0$ for which $C^\theta_r(x') \cap \Sigma = \emptyset$.  Therefore for $i \gg 1$ we would have $B_{r/2}(x) \cap \spt V_{i, I} = \emptyset$, which contradicts the fact that $\spt V_{i, I} \to P^-_\theta$ as closed sets in $B_1$.  So we have $\# \textbf{u}(x') \geq 1$ for all $x' \in B^\theta_{1/4} \cap \{ x_1 \leq 0 \}$.

We second claim that in fact $\textbf{u}(x') = \{ u(x') \}$ for $u : C^\theta_{1/4} \cap \{ x_1 \leq 0 \} \to \R$ a $C^\alpha$ function satisfying $|u|_{C^\alpha} \leq c(n, \theta)$, with $\alpha$ as in Theorem \ref{thm:harnack}.  To see this, note that by from Theorem \ref{thm:harnack} and our hypotheses there is a sequence $\tau_i \to 0$ so that
\begin{equation}\label{eqn:decay-1}
\osc_{P_{\theta_i}}(V_{I, i}, B_r(x)) \leq c r^\alpha E_i \quad \forall  x \in \spt V_{i, I} \cap B_{1/2}, \quad \forall \tau_i \leq r \leq 1/2.
\end{equation}
Take $x' \neq y' \in B^\theta_{1/4} \cap \{ x_1 \leq 0 \}$, and $x_{n+1} \in \textbf{u}(x'), y_{n+1} \in \textbf{u}(y')$.  We aim to show that
\begin{equation}\label{eqn:decay-2}
|x_{n+1} - y_{n+1}| \leq c(n, \theta) |x' - y'|^\alpha.
\end{equation}
From this, and the fact $|x_{n+1}| \leq 2/\sin\theta$, our second claim will follow.

To prove \eqref{eqn:decay-2}, there is no loss in assuming $|x' - y'| \leq 1/8$, as otherwise the inequality is trivial.  Choose $x_i, y_i \in \spt V_{i, I}$ so that $F_i(x_i) \to (x', x_{n+1})$ and $F_i(y_i) \to (y', y_{n+1})$.  In other words, if we write $x_i = (x_i', x_{i, n+1}) = (x_{i, 1}, x_{i, 2}, \ldots, x_{i, n+1})$ then $x_i' \to x'$ and $E_i^{-1} (x_{i, n+1} + x_{i, 1} \cot\theta_i) \to x_{n+1}$, and the same for $y$.

Applying \eqref{eqn:decay-1} with the balls $B_{2|x' - y'|}(x_i)$, then for $i \gg 1$ we get
\[
|(x_i - y_i) \cdot \nu_{\theta_i} | \leq c(n, \theta) |x' - y'|^\alpha E_i, 
\]
which implies \eqref{eqn:decay-2} after dividing by $E_i \sin\theta_i$ and taking $i \to \infty$.  This proves our second claim.

\vspace{3mm}

\emph{Step 2: Graph is viscosity solution.}  We claim that $u$ is a viscosity solution to
\begin{equation}\label{eqn:decay-3}
\cL_\theta u = 0 \text{ in } B^\theta_{1/4} \cap \{ x_1 < 0 \}, \quad \del_1 u = 0 \text{ on } B^\theta_{1/4} \cap \{ x_1 = 0 \}.
\end{equation}

Suppose $U \subset B^\theta_{1/4} \cap \{ x_1 < 0 \}$ is an open set, and $\phi : U \to \R$ is a smooth function satisfying $\phi \geq u$ on $U$ and $\phi = u$ at $\bar x' \in U$.  We can pick numbers $c_i \to 0$ and points $x_i \to (\bar x', u(\bar x'))$ so that in $U \times \R$, $\graph(\phi + c_i)$ touches $\Sigma_i$ from above at $x_i$.  Therefore $F_i^{-1}(\graph(\phi + c_i)) \equiv G_{\theta_i}(E_i(\phi + c_i))$ touches $\spt V_{i, I}$ from above at $F_i^{-1}(x_i)$ in $U \times \R$.

By the strong maximum principle, writing $x_i = (x_i', x_{i, n+1})$, we conclude that
\[
-\eps_i E_i \leq -||H^{tan}_{V_i, g_i}|| \leq H_{\theta_i}( E_i(\phi + c_i), g_i)|_{x_i} \leq E_i \cL_\theta(\phi)|_{x_i'} + c(\phi, n, \theta)( E_i^2 + \eps_i E_i)
\]
Taking $i \to \infty$ we deduce that $\cL_\theta(\phi)|_{\bar x'} \geq 0$.  The same argument shows that if $\phi$ touches $u$ at $\bar x'$ from below instead then $\cL_\theta(\phi)|_{\bar x'} \leq 0$.  We deduce that $\cL_\theta(u) = 0$ in the viscosity sense.

Suppose now $U$ is relatively open in $B^\theta_{1/4} \cap \{ x_1 \leq 0 \}$, and $\phi : U \to \R$ is a smooth function which touches $u$ from above at some $\bar x' \in U \cap \{ x_1 = 0 \}$.  For $\delta > 0$ small and $\Lambda > 0$ large (to be fixed momentarily), consider the function $\phi_\delta(x') = \phi - \Lambda x_1'^2 - \delta x_1'$.  For $\delta'(\Lambda, \delta)$ small, on $U \cap \{ - \delta' \leq x_1' \leq 0 \}$ we have $\phi_\delta \geq \phi$ and $\phi_\delta(\bar x') = \phi(\bar x')$, and so $\phi_\delta$ touches $u$ from above at $\bar x'$ also.

As before, pick $c_i \to 0$ and $x_i \to (\bar x', u(\bar x'))$ so that $F_i^{-1}(\graph(\phi_\delta + c_i)) \equiv G_{\theta_i}( E_i(\phi_\delta + c_i))$ touches $\spt V_{i, I}$ from above at $y_i = F_i^{-1}(x_i)$.  Fix $\Lambda(\phi, n, \theta)$ sufficiently large so that
\[
\cL_\theta(\phi_\delta) = \cL_\theta \phi - 2 \sin^2 \theta \Lambda < -\Lambda/2 < 0 \text{ on } U,
\]
and hence for $i \gg 1$ we have
\[
H_{\theta_i}(E_i(\phi_\delta + c_i), g_i)|_{x_i} \leq -E_i \Lambda/2 + c(\phi, n, \theta) (E_i^2 + \eps_i E_i) < -\eps_i E_i \leq - ||H^{tan}_{V_i, g_i}||.
\]
By the strong maximum principle we deduce that necessarily $x_i \in \{x_1 = 0 \}$ for all $i \gg 1$.

But now by Lemma \ref{lem:D-vs-osc} and Theorem \ref{thm:boundary-max}, we must have
\begin{align*}
\cos\theta_i - \eps_i E_i &\leq \cos\beta_i(x_i) \\
&\leq \nu_{\theta_i}(E_i (\phi_\delta + c_i), g_i)|_{x_i} \cdot e_1 \\
&\leq \cos\theta_i + E_i (\cos^2\theta\sin\theta - 1)\del_1 \phi_\delta|_{x_i'} + c(\phi, n, \theta) (E_i^2 + \eps_i E_i),
\end{align*}
and hence taking $i \to \infty$ we get $\del_1 \phi(\bar x') \leq \delta$.  Since $\delta$ is arbitrary we deduce $\del_1 \phi(\bar x') \leq 0$.  A verbatim argument shows that if $\phi$ touches from below then instead we have $\del_1 \phi(\bar x') \geq 0$.  We deduce $\del_1 u = 0$ along $B_{1/4}^\theta \cap \{ x_1 = 0 \}$ in the viscosity sense.

\emph{Step 3: $C^{1,\alpha}$ decay.}  From Step 2, we know that $u$ solves \eqref{eqn:decay-3} in the viscosity sense, and hence $u$ solves \eqref{eqn:decay-3} in the classical sense.  Since $|u| \leq c(n, \theta)$, standard Schauder estimates (e.g. after reflection about $\{x_1 = 0\}$) give that
\[
\sup_{B^\theta_{1/8} \cap \{ x_1 \leq 0 \} } \Big(|Du| + |D^2 u|\Big) \leq c(n, \theta),
\]
and therefore
\[
|u(x') - Du|_0(x')| \leq c(n, \theta) |x'|^2 \text{ on } B^\theta_{1/8} \cap \{ x_1 \leq 0 \}.
\]
By our Hausdorff convergence $\Sigma_i \to \graph(u)$ on $C^\theta_{1/4}$, we deduce that that for any $\eta \in (0, 1/8)$, we have for all $i \gg 1$:
\begin{equation}\label{eqn:decay-4}
\sup \{ |x_{n+1} - Du|_0(x')| : (x', x_{n+1}) \in \Sigma_i  \cap C^\theta_\eta \} \leq c(n, \theta) \eta^2.
\end{equation}

The Neumann condition $\del_1 u = 0$ implies $Du|_0(e_1) = 0$, so we can find $q_i \in O(\{ x_1 = 0 \})$ and linear functions $\ell_i : \hor \to \R$ satisfying
\begin{equation}\label{eqn:decay-5}
q_i(P_{\theta_i}) = G_{\theta_i}(\ell_i), \quad |\ell_i - E_i Du|_0| \leq c(n, \theta) E_i^2, \quad |q_i - Id| \leq c(n, \theta) E_i.
\end{equation}
From the definition of $\Sigma_i, F_i$, and from \eqref{eqn:decay-4}, \eqref{eqn:decay-5} we deduce that
\[
\osc_{q_i(P_{\theta_i})}(V_{I, i}, B_\eta) \leq c\eta^{2} E_i + c E_i^2 \leq c(n, \theta) \eta^2 E_i \quad \forall i \gg 1 .
\]
This gives \eqref{eqn:decay-concl}, which violates our contradiction hypothesis for $i$ large.
\end{proof}

\section{Regularity and Main theorems}\label{sec:proofs-of-main}

Our main regularity theorem is the following.
\begin{theorem}[Allard-type regularity]\label{thm:main}
Given $\theta \in (0, \pi/2)$, there are constants $c(n, \theta)$, $\eps(n, \theta)$ so that the following holds.  Let $g$ be a $C^1$ metric on $B_1$, $\beta$ a $C^1$ function on $B_1$, and $V$ a $(\beta, S)$-capillary varifold in $(B_1, g)$.  Suppose
\begin{gather}
0 \in \spt V_I, \quad \Theta_V(0, 1) \leq  3/4 + \cos\theta/4 , \\
E := \max\{ \osc_{P_{\theta}}(V_I, B_1), ||H^{tan}_{V, g}||_{L^\infty(B_1)}, |D\beta|_{C^0(B_1)}, |g - g_{eucl}|_{C^1(B_1)}\} \leq \eps,
\end{gather}
Then there is a $C^{1,1/2}$ function $u : B^{\theta}_{1/16}(0) \to \R$ satisfying
\begin{gather}
\spt V_I \cap B_{1/16} = G_{\theta}(u) \cap B_{1/16}, \\
|u|_{C^{1,1/2}} \leq c E .
\end{gather}
\end{theorem}

\begin{proof}
From Theorem \ref{thm:decay-r}, if we let $U = \pi_{\hor}(\spt V_I \cap B_{1/8})$, then there is a function $u : U \to \R$ so that
\[
\spt V_I \cap B_{1/8} = G_\theta(u), 
\]
and with the property that for every $x', y' \in U$, we have
\[
|u(x')| \leq c E, \quad |u(y') - u(x') - \ell_{x'}(y' - x')| \leq c |x' - y'|^{3/2} E
\]
for some linear map $\ell_{x'} : \hor \to \R$ satisfying $|\ell_{x'}| \leq c E$.  Here $c = c(n, \theta)$.

It will suffice to prove that $U \supset B_{1/16}^\theta \cap \{ x_1 \leq 0 \}$, as it will then follow from standard arguments that $u$ is $C^{1,1/2}$ with norm bounded by $c(n, \theta)E$.  Note that by Lemma \ref{lem:D-vs-osc}, provided $\eps(n, \theta)$ is small the set $U$ is closed and $1/100$-dense in $B_{1/32}^\theta \cap \{ x_1 \leq 0 \}$.  Suppose $x' \in B^\theta_{1/16} \cap \{ x_1 \leq 0 \} \setminus U$.  Choose $\rho > 0$ so that $B_\rho(x') \cap U = \emptyset$ but there is $y' \in U \cap \del B_\rho(x')$.

Choose $y \in \spt V_I \cap B_{1/32}$ with $\pi_{\hor}(y) = y'$.  By Lemma \ref{lem:conv} (if $y_1 = 0$) or by Theorem \ref{thm:allard-compact} and \eqref{eqn:cap-first-var-concl3} (if $y_1 < 0$), the constancy theorem, and the decay estimate \eqref{eqn:decay-r-concl}, we have that as $r \to 0$ the rescaled supports
\[
\spt(V_{y, r, I}) \cap B_1 \equiv (\spt V_I - y)/r \cap B_1 \to Q_y\cap\{x_1\le0\}
\]
in the Hausdorff distance as $r \to 0$.  Therefore $\pi_{\hor}( (\spt V_I - y)/r) \to U_y$ where $U_y = \hor \cap \{ x_1 < 0 \}$ if $y_1 = 0$ and $U_y = \hor$ if $y_1 < 0$.  But on the other hand, our choice of $\rho$ would necessitate that $U_y \subset \{ z \in \hor : z \cdot (x - y) \leq 0 \}$, which is a contradiction.
\end{proof}

One corollary of the regularity theorem and the classification of cones, which was not highlighted in the Introduction, is the following cone rigidity/density gap.  Compare to Lemma \ref{lem:tangent-cones}.
\begin{cor}\label{cor:density-gap}
Given $\theta \in (0, \pi/2)$, $\Lambda > 0$, there is an $\eps(n, \theta, \Lambda)$ so that the following holds.  Let $V$ be a non-zero, conical, stationary $(\theta, S)$-capillary varifold in $\R^{n+1}$ satisfying
\begin{gather}\label{eqn:density-gap-hyp}
\Theta_V(0, 1) \leq (1+\cos\theta)/2 + \eps, \quad \Lambda < \mu_V(B_1 \cap \R^n) < \cos\theta\omega_n.
\end{gather}
Then up to rotation $V = V^{(\theta)}$.
\end{cor}

\begin{proof}
Proof by contradiction: suppose there are sequences $\eps_i \to 0$ and $V_i$ satisfying \eqref{eqn:density-gap-hyp} so that (up to rotation) $V_i \neq V^{(\theta)}$.  Since every $\mu_{V_i} \neq \cos\theta [\R^n]$ and $\neq 0$, from Theorem \ref{thm:fb-first-var} we have that $\spt \sigma_{V_i} \equiv \spt V_{I, i} \cap \R^n \neq 0$.  By Lemma \ref{lem:conv}, up to a subsequence, the $V_i \to V$ for some conical, stationary $(\theta, S)$-capillary varifold in $\R^{n+1}$ satisfying
\[
\Theta_V(0, 1) \leq (1+\cos\theta)/2, \quad V_I \neq 0, \quad \mu_V(\R^n) > 0.
\]
If $\mu_{V}\llcorner\R^n = \cos\theta [\R^n]$, then $V_I$ is a non-zero stationary free-boundary cone with density $\theta_{V_I}(0) \leq (1-\cos\theta)/2 < 1/2$.  However, by standard monotonicity we must have $\theta_{V_I}(0) \geq 1/2$, which is a contradiction.  Therefore we must have $\mu_V(\R^n\cap B_1) < \cos\theta \omega_n$.  

We can apply Lemma \ref{lem:tangent-cones} to get $V = V^{(\theta)}$, and hence by Theorem \ref{thm:teaser3} we have for $i \gg 1$ every $V_{I,i}$ is a $C^{1}$ surface meeting $\R^n$ at angle $\theta$.  Since each $V_i$ is conical we must have $V_i = q_\sharp V^{(\theta)}$, for some rotation $q$.
\end{proof}

\vspace{3mm}

We now prove the theorems and corollaries outlaid in the Introduction.
\begin{proof}[Proof of Theorem \ref{thm:teaser1}]
This is proven in Theorem \ref{thm:cap-first-var} and Theorem \ref{thm:refined}.
\end{proof}

\begin{proof}[Proof of Theorem \ref{thm:teaser1.free.bdry}]
The proof is the same as in the capillary case, simply replace $H_{\beta, g}$ with $H^{tan}_V|_{x_1 = 0}$.
\end{proof}

\begin{proof}[Proof of Theorem \ref{thm:teaser2}]
By Lemma \ref{lem:D-vs-osc}, as $\delta \to 0$ then $V \to V^{(\theta)}$ in $B_1$ and $\beta \to \theta$ in $C^1(B_1)$.  If $\theta > \pi/2$, we can replace $V$ with $V' = V - \cos\beta[\R^n]_g$, and then $V'$ is a $(\pi-\beta, \R^n \setminus S)$-capillary varifold in $(B_1, g)$ with the property that $V' \to q_\sharp V^{(\pi - \theta)}$ in $B_{1}$ as $\delta \to 0$, where $q$ is reflection about the line $\{x_{n+1} = 0\}$.  Therefore, taking $\delta(n, \theta)$ sufficiently small there is no loss in assuming $\theta < \pi/2$, and then Theorem \ref{thm:teaser2} is a direct consequence of Theorem \ref{thm:main} (for the regularity) and Theorem \ref{thm:cap-first-var} (for the closeness of $\spt V_I$ to $0$).
\end{proof}

\begin{proof}[Proof of Theorem \ref{thm:teaser3}]
Since by our assumption $V \to V^{(\theta)}$ and $\cos\beta \to \theta$ as $\delta \to 0$, we have $\Theta_V(0, 1) + (\cos\beta(0))_- \leq 3/4 + |\cos\theta|/4 < 1$ for $\delta(n, \theta)$ small.  Lemma \ref{lem:conv} then implies $\spt V_I \to P_\theta^-$ as $\delta \to 0$, and so we can apply Theorem \ref{thm:teaser2} (in a slightly smaller ball) for $\delta(n, \theta)$ small.
\end{proof}

\begin{proof}[Proof of Corollary \ref{cor:teaser1}]
Suppose $\beta(x) \in (0, \pi) \setminus \{ \pi/2 \}$, and some tangent cone of $V$ at $x \in \R^n$ is $V^{(\beta(x))}$, i.e. if one changes coordinates so that $g(x) = g_{eucl}$, then there is a sequence $r_i \to 0$ so that $V_i := (\eta_{x, r_i})_\sharp V \to V^{(\beta(x))}$, where we interpret $V_i$ as $(\beta_i, S_i)$-capillary varifolds in the metric $g_i(x) := g(rx)$, with $\beta_i(x) = \beta(rx)$.

Without loss of generality, let us simply suppose $x = 0$ and $g(0) = g_{eucl}$.  For any $\eps > 0$ small and $i \gg 1$, we have
\[
\Theta_{V_i}(0, 1) + (\cos\beta(0))_- \leq \Theta_V(0, 1) + (\cos\beta(0))_- + \eps = (1+|\cos\beta(0)|)/2 + \eps < 1.
\]
Therefore we can use Lemma \ref{lem:conv} to get $\spt V_{I, i} \to P_\theta^-$ in the local Hausdorff distance, and hence for any $\eps > 0$ and $i \gg 1$ we have
\[
0 \in \spt V_{I, i}, \quad \osc_{P_{\beta(0)}}(V_{I, i}, B_1) \leq \eps.
\]
And of course by scaling we have
\[
\max \{ ||H^{tan}_{V_i, g_i}||_{L^\infty(B_1)}, |g_i - g_{eucl}|_{C^1(B_1)}, |D\beta_i|_{C^0(B_1)} \} \to 0
\]
Now apply Theorem \ref{thm:teaser2}.
\end{proof}

\begin{proof}[Proof of Theorem \ref{thm:teaser4}]
For the first assertion it suffices to work in a small neighborhood of some point $x \in \spt V_I \cap \R^n \cap \{ \Theta_{V, g}(x) + (\cos\beta(x))_- < 1\}$.  By upper-semi-continuity of $\Theta_{V, g}$ and monotonicity \eqref{eqn:mono-weak-concl}, we can assume without loss of generality that $x = 0$, and that
\[
\Theta_V(0, 1) + (\cos\beta(0))_- < 1 .
\]
Openness of the regular set follows from the definition, so we will focus on proving denseness in a neighborhood of $0$.

By Theorem \ref{thm:refined}, there is a radius $\gamma > 0$ so that for $\haus^{n-1}$-a.e. $x \in B_\gamma \cap \R^n \cap \{ \beta \neq \pi/2 \} \cap \del^*S$, every tangent cone to $V$ at $x$ is (some rotation of) $V^{(\beta(x))}$, which implies by Corollary \ref{cor:teaser1} that every such $x$ is a regular point.

On the other hand, Theorem \ref{thm:refined} says that for $\haus^{n-1}$-a.e. $x \in \spt V_I \cap \R^n \cap B_\gamma \setminus (\{ \beta \neq \pi/2 \} \cap \del^* S)$ we have $\eta_{V_I, g}(x) \perp_g \R^n$, and every tangent cone to $V_I$ at $x$ is some rotation of $[P_{\pi/2}]$.  Therefore $V_I$ is a free-boundary varifold in $B_\gamma \setminus \overline{ \{ \beta \neq \pi/2 \} \cap \del^*S }$ with bounded mean curvature, whose tangent at $\haus^{n-1}$-a.e. boundary point is a free-boundary half-plane with multiplicity-one.  The Allard-type regularity theorem of \cite{GrJo} implies $V_I$ is regular at $\haus^{n-1}$-a.e. $x \in \spt V_I \cap \R^n \cap B_\gamma \setminus \overline{ \{ \beta \neq \pi/2 \} \cap \del^*S }$.  This proves denseness near $0$.

We prove the second assertion of the Theorem.  After replacing $S$ with $\R^n \setminus S$ and $\beta$ with $\pi - \beta$ as in Remark \ref{example.swap.region}, we can without loss of generality assume $\cos\beta(0) > 0$ (since $V_I$ remains unchanged).  After shrinking our ball we can therefore work under the additional assumption that
\[
\Theta_{V, g}(x) < 1/2 + \cos\beta(x) \quad \text{and}\quad \cos\beta(x) > 0 \quad \forall x \in B_1 \cap \R^n.
\]

Since $\Theta_{V, g}(x) = (1+\cos\beta(x))/2$ for $\haus^{n-1}$-a.e. $x \in \del^*S \cap B_\gamma$, by upper-semi-continuity and our assumptions we have
\[
1/2 < (1+\cos\beta(x))/2 \leq \Theta_{V, g}(x) < 1/2 + \cos\beta(x) \quad \forall x \in B_\gamma \cap \overline{\del^* S}.
\]
But $\Theta_{V, g}(x) \in \{ 1/2, 1/2 + \cos\beta(x)\}$ for $\haus^{n-1}$-a.e. $x \in \spt V_I \cap \R^n\cap B_\gamma \setminus \del^* S$, and so we must have $\haus^{n-1}(B_\gamma \cap \overline{\del^* S} \setminus \del^*S) = 0$.  In $B_\gamma \setminus \overline{\del^* S}$ the interior varifold $V_I$ is a free-boundary varifold, and so by the same reasoning as before using \cite{GrJo} we have $\haus^{n-1}$-a.e. $x \in \spt V_I \cap \R^n \cap B_\gamma \setminus \overline{\del^*S}$ is a regular point.
\end{proof}

\section{Appendix: free-boundary first variation control} \label{sec:app1}


Recall that a signed $n$-rectifiable varifold $V$ in $(B_1 \subset \R^{n+1}, g)$ is said to have free-boundary in $\{ x_1 = 0 \}$ if there is a vector field $H^{tan}_{V, g} \in L^1_{loc}(B_1, \R^{n+1}; |\mu_V|)$ so that $H_{V,g}^{tan}$ is tangential to $\{x_1=0\}$ at $|\mu_V|$-a.e. point in $\{x_1=0\}$ and
\begin{equation}\label{eqn:fb-cond}
\delta_g V(X) = -\int g(H_{V,g}^{tan} , X) d\mu_V
\end{equation}
for all $X \in C^1_c(B_1, \R^{n+1})$ which are tangential to $\{ x_1 = 0 \}$.

\begin{theorem}\label{thm:fb-first-var}
Let $g$ be a $C^1$ metric on $B_1$, and let $V$ be an $n$-rectifiable signed varifold in $(B_1, g)$ supported in $\{ x_1 \leq 0 \}$ with $g$-free-boundary in $\{ x_1 = 0\}$ and $g$-mean curvature $H^{tan}_{V, g}$, and assume that $\mu_{V_I}\equiv\mu_V\llcorner\{x_1<0\}$ is a positive measure. 
Then $V$ has locally-finite $g$-first-variation in $B_1$, and there exists a Radon measure $\sigma_{V,g} \perp |\mu_V|$ on $B_1$, supported in $\{ x_1 = 0 \}$, so that
\begin{equation}\label{eqn:fb-first-var-concl1}
\delta_g V(Y) = -\int g(H_{V, g}^{tan}, Y) d\mu_V - \int_{\{x_1 = 0\}} g(H_{\ver, g}, Y) d\mu_V + \int g(\nu_g, Y) d\sigma_{V,g}
\end{equation}
where $H_{\R^n, g}$ is the mean curvature vector of $\R^n \subset (\R^{n+1}, g)$ and $\nu_g$ is defined in \eqref{eq:unit_normal_definition}.

Furthermore, for any $\gamma < 1$ we have the boundary measure mass bound
\begin{equation}\label{eqn:fb-first-var-concl2}
\sigma_{V, g}(B_\gamma) \leq c ||H^{tan}_{V, g}||_{L^1(\mu_{V_I})} + c (|Dg|_{C^0(B_1)} + 1)\mu_{V_I}(B_1)
\end{equation}
where $c = (n, \gamma, |g|_{C^0(B_1)})$.
\end{theorem}

\begin{remark}
Theorem \ref{thm:fb-first-var} says that if $V$ has bounded tangential mean curvature $H^{tan}_{V, g}$, then $V$ has bounded total mean curvature $H_{V, g} = H_{V, g}^{tan} + H_{\ver, g}$.  Since $H_{V, g}$ coincides with $H^{tan}_{V, g}$ for tangential fields and $|H_{V, g}| \leq |H^{tan}_{V, g}| + c(n) |Dg|$ (and we will always working in a $C^1$ metric), in this paper we simply say $V$ has bounded mean curvature to indicate either tangential or total mean curvature is bounded.
\end{remark}

\begin{remark}\label{rem:no-V_I}
If $V_I \llcorner B_r(x) = 0$ for some $B_r(x) \subset B_1$, then $\sigma_{V, g} \llcorner B_r(x) = 0$ also.  To see this, take any $X \in C^1_c(B_r(x))$ and write $$X = (X - g(X, \nu_g) \nu_g) + g(X, \nu_g) \nu_g =: X'' + X',$$
where $\nu_g = g^{1j} e_j/\sqrt{g^{11}}$.
Then $X''$ is tangential to $\R^n$, and since $V \equiv V_B$ in $B_r(x)$ we can argue as in the proof of Theorem \ref{thm:fb-first-var}:
\begin{align}
\delta_g V(X) &= \delta_g V(X'') + \delta_g V(X') \\
&= -\int g(H_{V, g}^{tan} , X'') +  g(X, \nu_g) \mathrm{div}_{\ver, g}(\nu_g) d\mu_{V_B} \\
&= -\int g(H_{V, g}^{tan}, X'') + g(H_{\ver, g}, X') d\mu_{V_B} \\
&= -\int g(H_{V, g}^{tan} + H_{\R^n, g}, X) d\mu_V.
\end{align}
\end{remark}

\begin{proof}[Proof of Theorem \ref{thm:fb-first-var}]
Write $|g|_{C^0} = \Gamma$.  Define the vector field
\[
e_g = \grad_g(x_1) \equiv \sum_{j=1}^{n+1}g^{1j} e_j,
\]
and note that $e_g$ is $g$-orthogonal to $\{ x_1 = 0 \}$.

Take any $h \in C^1_c(B_{1})$, any $\tau \in (0, \infty)$, let $\phi : \R \to \R$ be a cutoff function which $\equiv 0$ in a neighborhood of $0$ and $\equiv 1$ on $[1, \infty)$, and then define the vector field $X : B_{1} \to \R^{n+1}$ by
\[
X = \phi(-x_1/\tau) h(x) e_g(x).
\]
Trivially $X$ is compactly supported in $B_{1} \cap \{ x_1 \neq 0 \}$, and so by the free-boundary condition \eqref{eqn:fb-cond} we have
\[
\delta_g V(X) = -\int g(H_{V, g}^{tan} ,  X) d\mu_V = -\int \phi h H_{V, g} ^{tan} \cdot e_1 d\mu_V
\]
We compute
\begin{align*}
-\int \phi\left(\frac{-x_1}{\tau}\right) h H_{V, g}^{tan}\cdot e_1 d\mu_V
&= \int \mdiv_{V, g}(X) d\mu_V \\
&= \int \left(-\frac{\phi'}{\tau} h g(\pi_{V, g}(\grad_g(x_1)), e_g) + \phi \mdiv_{V, g}(h e_g)\right) d\mu_V \\
&= \int \left(-\frac{\phi'}{\tau} h |\pi_{V, g}(e_g)|^2_g + \phi \mdiv_{V, g}(h e_g) \right)d\mu_V.
\end{align*}
By the dominated convergence theorem we can assume without loss of generality that $\phi(s) = \max\{ \min\{ s, 1\}, 0\}$, and we therefore obtain
\[
\frac{1}{\tau} \int_{\{-\tau < x_1 < 0\} } h |\pi_{V, g}(e_g)|^2_g d\mu_V = \int_{\{ x_1 < 0 \}} \left(\phi h H_{V, g}^{tan} \cdot e_1 + \phi \mdiv_{V, g}(h e_g) \right)d\mu_V.
\]
Once again by dominated convergence, taking $\tau \to 0$ we deduce
\begin{align} 
&\lim_{\tau \to 0} \frac{1}{\tau} \int_{\{-\tau < x_1 < 0\}} h |\pi_{V, g}(e_g)|^2_g d\mu_V \\
&\qquad= \int_{\{x_1 < 0\}} \left(h H_{V, g}^{tan} \cdot e_1 + \mdiv_{V, g}(h e_g) \right)d\mu_V \label{eqn:fb-first-var-0.5} \\
&\qquad\leq c(n) |g|_{C^0} |h|_{C^1} \Big(||H_{V, g}^{tan}||_{L^1(\mu_{V_I})} + (|Dg|_{C^0(B_1)} + 1) \mu_{V_I}(B_1)\Big)
\end{align}
where $\mu_{V_I}=\mu_V\llcorner\{x_1<0\}$.

Now take $\gamma < 1$, and fix $\Phi$ any cutoff function which is $\equiv 1$ in $B_\gamma$ and supported in $B_1$.  We claim that
\begin{equation}\label{eqn:F-def}
F(h) := \lim_{\tau \to 0} \frac{1}{\tau}\int_{\{-\tau < x_1 < 0\}} h |\pi_{V, g}(e_g)|^2 d\mu_V
\end{equation}
is a Radon measure on $B_\gamma$, i.e. $F(h) = \int h dF$.  To see this, simply observe that
\begin{align}\label{eqn:F-bounds}
F(h) &\leq |h|_{C^0(B_\gamma)} F(\Phi) \\
&\leq |h|_{C^0(B_\gamma)} c(n, \gamma, \Gamma)(||H_{V, g}^{tan}||_{L^1(\mu_{V_I})} + (|Dg|_{C^0(B_1)} + 1) \mu_{V_I}(B_1)),
\end{align}
which implies that $F$ extends to a continuous linear function on $C^0(B_\gamma)$.  From the definition \eqref{eqn:F-def} (and since $\spt V \subset \{ x_1 \leq 0 \}$) it's clear that $\spt F \subset \{ x_1 = 0 \}$, but also observe that $F$ only depends on $V \llcorner \{ x_1 < 0 \}$, and so in fact $F \perp |\mu_V|$.

Now let $Y$ be any $C^1_c$ vector field in $B_\gamma$, and define
\[
Y' := \frac{g(Y, e_g)}{|e_g|^2_g} e_g \equiv \frac{Y^1}{g^{11}} g^{1j} e_j.
\]
Then from \eqref{eqn:fb-first-var-0.5} we have
\begin{align}
&\int_{\{x_1 < 0\} } \left(\mdiv_{V, g}(Y') + g(H_{V, g}^{tan}, Y') \right)d\mu_V \nonumber \\
&= F( g(Y, e_g)/|e_g|^2_g) \nonumber \\
&\leq c(n, \gamma, \Gamma) |Y|_{C^0(B_\gamma)} ( ||H_{V, g}^{tan}||_{L^1(\mu_{V_I})} + (|Dg|_{C^0(B_1)} + 1)\mu_{V_I}(B_1)). \label{eqn:fb-first-var-1}
\end{align}

Since $V$ is $n$-rectifiable we have $T_x V = T_x \ver$ for $\mu_V$-a.e. $x_1 = 0$.  Therefore, because $e_g/|e_g|_g$ is the $g$-unit normal to $\{ x_1 = 0 \}$, we can compute
\begin{align}
\int_{\{x_1 = 0\}} \mdiv_{V, g}(Y') d\mu_V &= \int_{\{x_1 = 0\}} \frac{g(Y, e_g)}{|e_g|_g} \mdiv_{\ver, g}\left( \frac{e_g}{|e_g|_g} \right) d\mu_V \nonumber \\
&= \int_{\{x_1 = 0\}} - g(H_{\ver, g}, Y') d\mu_V, \label{eqn:fb-first-var-2}
\end{align}
where $H_{\ver, g}$ is the mean curvature vector of $\ver$ in $(\R^{n+1}, g)$.

Noting that both $Y'' := Y - Y'$ and $H^{tan}_{V, g}$ are tangential to $\{ x_1 = 0 \}$, while $Y'$ is $g$-normal to $\{ x_1 = 0\}$, we can combine \eqref{eqn:fb-first-var-1}, \eqref{eqn:fb-first-var-2} and the free-boundary condition \eqref{eqn:fb-cond} to obtain
\begin{align}
&\int \mdiv_{V, g}(Y) d\mu_V\\
&\quad= \int \mdiv_{V, g}(Y'') d\mu_V + \int_{\{x_1 = 0\}} \mdiv_{V, g}(Y') d\mu_V + \int_{\{x_1 < 0\}} \mdiv_{V, g}(Y') d\mu_V \\
&\quad= -\int g(H_{V, g}^{tan}, Y) d\mu_V - \int_{\{x_1 = 0\}} g(H_{\ver, g}, Y) d\mu_V + F( g(Y, e_g)/|e_g|^2_g) \\
&\quad\leq c(n, \gamma, \Gamma)|Y|_{C^0(B_\gamma)}( ||H_{V, g}^{tan}||_{L^1(B_1)} + (|Dg|_{C^0(B_1)} + 1) |\mu_V|(B_1)).
\end{align}
We deduce that $V$ has locally-bounded first variation in $B_1$, and hence by standard Radon-Nikodyn we can write
\[
\int \mdiv_{V, g}(Y) d\mu_V = -\int g(H, Y) d\mu_V + \int g(\eta, Y) d\sigma_V
\]
for some $\sigma_V \perp |\mu_V|$, and some vector field $H \in L^1_{loc}(|\mu_V|)$, and some $g$-unit, $\sigma_V$-measurable vector field $\eta$.  From the free-boundary condition \eqref{eqn:fb-cond} it's clear that: $H = H^{tan}_{V, g}$ $|\mu_V|$-a.e. $x_1 < 0$; $g(H, X) = g(H^{tan}_{V, g}, X)$ for $|\mu_V|$-a.e. $x_1 = 0$ and any $X \in \ver$;  $\sigma_{V}$ must be supported in $\ver$; and $\eta = \pm \nu_g$ $\sigma_V$-a.e.

To prove \eqref{eqn:fb-first-var-concl1} it remains to show that $g(H, e_g) = g(H_{\ver, g} , e_g)$ on $\{ x_1 = 0 \}$ and $\eta = + \nu_g$.  Plugging $Y = h e_g$ into the first variation formula we have
\begin{align}
&-\int g(H, e_g) h d\mu_V + \int g(\eta, e_g) h d\sigma_{V} \\
&= -\int_{\{x_1 < 0\}} g(H^{tan}_{V, g}, e_g)  h d\mu_V - \int_{\{x_1 = 0\}} g(H_{\ver, g}, e_g) h d\mu_V + F( h )
\end{align}
for any $h \in C^0_c(B_1)$.  Therefore since $H = H^{tan}_{V, g}$ $\mu_V$-a.e. $x_1 < 0$ and since $F \perp |\mu_V|$, we deduce 
\begin{equation}\label{eqn:sigma-vs-F}
\int_{\{x_1 = 0\}} g(H - H_{\ver, g}, e_g) h d\mu_V = 0, \quad \int g(\eta, e_g) h d\sigma_{V} = F(h) \geq 0
\end{equation}
for all $h \geq 0$, which implies our claim.

The mass bound \eqref{eqn:fb-first-var-concl2} follows from \eqref{eqn:F-bounds} and \eqref{eqn:sigma-vs-F}: for any $h \in C^0_c(B_\gamma)$ non-negative we have
\begin{align}
\sigma_{V}(h) &\leq c(n, \Gamma) \int g(\nu_g, e_g) h d\sigma_{V} \\
&= c(n, \Gamma) F(h) \\
&\leq c(n, \gamma, \Gamma) |h|_{C^0} ( ||H^{tan}_{V, g}||_{L^1(\mu_{V_I})} + (|Dg|_{C^0(B_1)} + 1)\mu_{V_I}(B_1)). \qedhere
\end{align}
\end{proof}

\bibliography{citations2.bib}

\providecommand{\bysame}{\leavevmode\hbox to3em{\hrulefill}\thinspace}
\providecommand{\MR}{\relax\ifhmode\unskip\space\fi MR }
\providecommand{\MRhref}[2]{%
  \href{http://www.ams.org/mathscinet-getitem?mr=#1}{#2}
}
\providecommand{\href}[2]{#2}
\begin{thebibliography}{10}

\bibitem{allard1972}
William~K. Allard, \emph{On the {First} {Variation} of a {Varifold}}, The
  Annals of Mathematics \textbf{95} (1972), no.~3, 417.

\bibitem{ChaiWang-dihedral}
Xiaoxiang Chai and Gaoming Wang, \emph{Dihedral rigidity in hyperbolic
  3-space}, Trans. Amer. Math. Soc. \textbf{377} (2024), no.~2, 807--840.
  \MR{4688535}

\bibitem{ChoEdeLi2024}
Otis Chodosh, Nick Edelen, and Chao Li, \emph{Improved regularity for
  minimizing capillary hypersurfaces}, 2024.

\bibitem{demasi}
Luigi De~Masi, \emph{Rectifiability of the free boundary for varifolds},
  Indiana Univ. Math. J. \textbf{70} (2021), no.~6, 2603--2651. \MR{4359920}

\bibitem{DeMasiDePhilippis}
Luigi De~Masi and Guido De~Philippis, \emph{Min-max construction of minimal
  surfaces with a fixed angle at the boundary}, 2021.

\bibitem{DePhil-Fusco-Morini}
Guido De~Philippis, Nicola Fusco, and Massimiliano Morini, \emph{Regularity of
  capillarity droplets with obstacle}, Trans. Am. Math. Soc. (2024) (en).

\bibitem{DePhil-Gas-Schu}
Guido De~Philippis, Carlo Gasparetto, and Felix Schulze, \emph{{A Short Proof
  of {Allard}’s and {Brakke}’s Regularity Theorems}}, International
  Mathematics Research Notices \textbf{2024} (2023), no.~9, 7594--7613.

\bibitem{DePhilippisMaggi}
Guido De~Philippis and Francesco Maggi, \emph{Regularity of free boundaries in
  anisotropic capillarity problems and the validity of {Y}oung's law}, Arch.
  Ration. Mech. Anal. \textbf{216} (2015), no.~2, 473--568. \MR{3317808}

\bibitem{desilva:fbreg}
Daniela De~Silva, \emph{Free boundary regularity for a problem with right hand
  side}, Interfaces and Free Boundaries \textbf{13} (2011), 223--238.

\bibitem{Ferreri-Tortone-Velichkov-2023}
Lorenzo Ferreri, Giorgio Tortone, and Bozhidar Velichkov, \emph{A capillarity
  one-phase {Bernoulli} free boundary problem}, 2023.

\bibitem{Finn}
Robert Finn, \emph{Equilibrium capillary surfaces}, Grundlehren der
  mathematischen Wissenschaften [Fundamental Principles of Mathematical
  Sciences], vol. 284, Springer-Verlag, New York, 1986. \MR{816345}

\bibitem{GiLa}
Alexandre Girouard and Jean Lagacé, \emph{Large {Steklov} eigenvalues via
  homogenisation on manifolds}, Invent. math. \textbf{226} (2021),
  1011--–1056.

\bibitem{GrJo}
M.~Gruter and J.~Jost, \emph{Allard type regularity results for varifolds with
  free boundaries}, Ann. Sc. Norm. Super. Pisa Cl. Sci. (5) \textbf{13} (1986),
  129--169.

\bibitem{Gruter}
Michael Gr\"{u}ter, \emph{Optimal regularity for codimension one minimal
  surfaces with a free boundary}, Manuscripta Math. \textbf{58} (1987), no.~3,
  295--343. \MR{893158}

\bibitem{Jia-Zhang-2024}
Xiaohan Jia and Xuwen Zhang, \emph{Quantitative {Alexandrov} theorem for
  capillary hypersurfaces in the half-space}, 2024.

\bibitem{KaTo}
Takashi Kagaya and Yoshihiro Tonegawa, \emph{A fixed contact angle condition
  for varifolds}, Hiroshima Mathematical Journal \textbf{47} (2017), no.~2,
  139--153.

\bibitem{stern-karpukhin}
Mikhail Karpukhin and Daniel Stern, \emph{From {Steklov} to {Laplace}: free
  boundary minimal surfaces with many boundary components}, 2021.

\bibitem{King-Maggi-Stuvard-2022}
Darren King, Francesco Maggi, and Salvatore Stuvard, \emph{Plateau's problem as
  a singular limit of capillarity problems}, Communications on Pure and Applied
  Mathematics \textbf{75} (2022), no.~3, 541--609.

\bibitem{Li-Polyhedron}
Chao Li, \emph{A polyhedron comparison theorem for 3-manifolds with positive
  scalar curvature}, Invent. Math. \textbf{219} (2020), no.~1, 1--37.
  \MR{4050100}

\bibitem{LiZhouZhu}
Chao Li, Xin Zhou, and Jonathan~J. Zhu, \emph{Min-max theory for capillary
  surfaces}, 2022.

\bibitem{Maggi-Valdinoci_2017}
Francesco Maggi and Enrico Valdinoci, \emph{Capillarity problems with nonlocal
  surface tension energies}, Communications in Partial Differential Equations
  \textbf{42} (2017), no.~9, 1403--1446.

\bibitem{savin07}
Ovidiu Savin, \emph{Small {Perturbation} {Solutions} for {Elliptic}
  {Equations}}, Communications in Partial Differential Equations \textbf{32}
  (2007), no.~4, 557--578.

\bibitem{Sim}
Leon Simon, \emph{Lectures on geometric measure theory}, Proceedings of the
  Centre for Mathematical Analysis, Australian National University, vol.~3,
  Australian National University Centre for Mathematical Analysis, Canberra,
  1983. \MR{MR756417 (87a:49001)}

\bibitem{Simon-cylindrical}
\bysame, \emph{Cylindrical tangent cones and the singular set of minimal
  submanifolds}, J. Differential Geom. \textbf{38} (1993), no.~3, 585--652.
  \MR{1243788}

\bibitem{solomon-white}
Bruce Solomon and Brian White, \emph{A strong maximum principle for varifolds
  that are stationary with respect to even parametric elliptic functionals},
  Indiana University Mathematics Journal \textbf{38} (1989), no.~3, 683--691.

\bibitem{Taylor}
Jean~E. Taylor, \emph{Boundary regularity for solutions to various capillarity
  and free boundary problems}, Comm. Partial Differential Equations \textbf{2}
  (1977), no.~4, 323--357. \MR{487721}

\bibitem{Valdinoci-survey}
Enrico Valdinoci, \emph{Towards a long-range theory of capillarity}, 2024.

\bibitem{Vel19}
Bozhidar Velichkov, \emph{Regularity of the one-phase free boundaries}, 2023,
  cvgmt preprint.

\bibitem{WangGaoming}
Gaoming Wang, \emph{Allard-type regularity for varifolds with prescribed
  contact angle}, 2024.

\bibitem{Wang-Xia-2019}
Guofang Wang and Chao Xia, \emph{Uniqueness of stable capillary hypersurfaces
  in a ball}, Math. Ann. \textbf{374} (2019), no.~3-4, 1845--1882 (en).

\bibitem{Young}
Thomas Young, \emph{An essay on the cohesion of fluids}, Philos. Trans. Roy.
  Soc. London \textbf{95} (1805), 65--87.

\end{thebibliography}
\bibliographystyle{amsplain}

\end{document}